\documentclass[10pt]{amsart}
\usepackage{amssymb,amstext,amsmath,amscd,amsthm,amsfonts,enumerate,latexsym,stmaryrd,multicol,geometry,graphicx}
\usepackage[usenames]{color}
\usepackage[all]{xy}
\newtheorem{thm}{Theorem}[section]
\newtheorem{lem}[thm]{Lemma}
\newtheorem{prop}[thm]{Proposition}
\newtheorem{cor}[thm]{Corollary}
\theoremstyle{definition}
\newtheorem{dfn}[thm]{Definition}
\newtheorem{ques}[thm]{Question}
\newtheorem{setup}[thm]{Setup}
\newtheorem{rem}[thm]{Remark}
\newtheorem{conv}[thm]{Convention}

\theoremstyle{remark}

\newtheorem*{claim*}{Claim}
\newtheorem*{ac}{Acknowlegments}

\numberwithin{equation}{thm}
\def\A{\mathcal{A}}
\def\aa{\boldsymbol{a}}
\def\add{\operatorname{\mathsf{add}}}
\def\ann{\operatorname{ann}}
\def\C{\mathcal{C}}
\def\c{\mathsf{C}}

\def\cm{\mathsf{CM}}
\def\codepth{\operatorname{codepth}}
\def\codim{\operatorname{codim}}
\def\cone{\operatorname{cone}}
\def\D{\mathcal{D}}
\def\d{\operatorname{\mathsf{D}}}
\def\dd{\mathrm{D}}
\def\db{\operatorname{\mathsf{D^b}}}
\def\depth{\operatorname{depth}}
\def\dpf{\operatorname{\mathsf{D^{perf}}}}
\def\ds{\operatorname{\mathsf{D_{sg}}}}
\def\E{\mathcal{E}}
\def\edim{\operatorname{edim}}
\def\ext{\operatorname{\mathsf{ext}}}
\def\Ext{\operatorname{Ext}}
\def\F{\mathsf{F}}
\def\fpd{\operatorname{\mathsf{fpd}}}
\def\G{\mathsf{E}}
\def\g{\mathsf{G}}
\def\gdim{\operatorname{Gdim}}
\def\ge{\geqslant}
\def\gp{\mathsf{GP}}
\def\grade{\operatorname{grade}}
\def\H{\mathsf{H}}
\def\h{\mathrm{H}}
\def\height{\operatorname{ht}}
\def\Hom{\operatorname{Hom}}
\def\I{\mathcal{I}}
\def\id{\mathrm{id}}
\def\im{\operatorname{Im}}
\def\inc{\mathrm{inc}}
\def\ipd{\operatorname{IPD}}
\def\K{\mathrm{K}}
\def\k{\operatorname{\mathsf{K}}}

\def\lcm{\underline{\mathsf{CM}}}
\def\ld{\operatorname{\underline{\mathsf{D}}}}
\def\le{\leqslant}
\def\lgp{\underline{\mathsf{GP}}}
\def\lten{\otimes^\mathbf{L}}
\def\m{\mathfrak{m}}
\def\Mod{\operatorname{Mod}}
\def\mod{\operatorname{\mathsf{mod}}}
\def\N{\mathbb{N}}
\def\NE{\operatorname{NE}}
\def\nf{\operatorname{NF}}
\def\OO{\mathcal{O}}
\def\P{\mathsf{P}}
\def\p{\mathfrak{p}}
\def\pd{\operatorname{pd}}

\def\PP{\mathbb{P}}
\def\Proj{\operatorname{Proj}}
\def\proj{\operatorname{\mathsf{proj}}}
\def\Q{\mathsf{Q}}
\def\q{\mathfrak{q}}
\def\R{\mathsf{R}}
\def\res{\operatorname{\mathsf{res}}}
\def\rfd{\operatorname{Rfd}}

\def\rhom{\mathbf{R}\mathrm{Hom}}
\def\s{\operatorname{\mathsf{S}}}
\def\sing{\operatorname{Sing}}
\def\spec{\operatorname{Spec}}
\def\supp{\operatorname{Supp}}
\def\syz{\Omega}
\def\T{\mathcal{T}}
\def\thick{\operatorname{\mathsf{thick}}}
\def\tr{\operatorname{Tr}}
\def\V{\operatorname{V}}

\def\X{\mathcal{X}}
\def\xx{\boldsymbol{x}}
\def\Y{\mathcal{Y}}

\def\Z{\mathbb{Z}}
\def\ZZ{\mathcal{Z}}
\begin{document}
\title[Classifying preaisles of derived categories of complete intersections]{Classifying preaisles of derived categories\\
of complete intersections}
\author{Ryo Takahashi}
\address{Graduate School of Mathematics, Nagoya University, Furocho, Chikusaku, Nagoya 464-8602, Japan}
\email{takahashi@math.nagoya-u.ac.jp}
\urladdr{https://www.math.nagoya-u.ac.jp/~takahashi/}
\thanks{The author was partly supported by JSPS Grant-in-Aid for Scientific Research 23K03070}
\begin{abstract}
Let $R$ be a commutative noetherian ring.
Denote by $\mod R$ the category of finitely generated $R$-modules, by $\db(R)$ the bounded derived category of $\mod R$, and by $\ds(R)$ the singularity category of $R$.
The main result of this paper provides, when $R$ is a complete intersection, a complete classification of the preaisles of $\db(R)$ containing $R$ and closed under direct summands, which includes as restrictions the classification of thick subcategories of $\ds(R)$ due to Stevenson, and the classification of resolving subcategories of $\mod R$ due to Dao and Takahashi.
\end{abstract}
\maketitle
\section{Introduction}

A {\em thick subcategory} of a triangulated category is a full triangulated subcategory closed under direct summands.
Classifying thick subcategories of a triangulated category has been one of the most central subjects shared by many areas of mathematics including representation theory, homotopy theory, algebraic geometry and commutative/noncommutative algebra; see \cite{BS,BCR,BIK,BIKP,DHS,FP,H,HS,dm,N,St,St2,stcm,thd,dlr,Th} for instance.
A significant work in commutative algebra is a classification of thick subcategories of derived categories of complete intersections by Stevenson \cite{St}.
We introduce a setup to explain Stevenson's theorem.

\begin{setup}\label{95}
Let $(R,V)$ be a pair, where $R$ and $V$ satisfy either of the following two conditions.
\begin{enumerate}[(1)]
\item
$R$ is a commutative noetherian ring which is locally a hypersurface, and $V$ is the singular locus of $R$.
\item
$R$ is a quotient ring of the form $S/(\aa)$ where $S$ is a regular ring of finite Krull dimension and $\aa=a_1,\dots,a_c$ is a regular sequence, and $V$ is the singular locus of the zero subscheme of $a_1x_1+\cdots+a_cx_c\in\Gamma(X,\OO_X(1))$ where $X=\PP_S^{c-1}=\Proj(S[x_1,\dots,x_c])$.
\end{enumerate}
\end{setup}

Under this setup, Stevenson \cite{St} proved the following classification theorem of thick subcategories.

\begin{thm}[Stevenson]\label{90}
Let $(R,V)$ be as in Setup \ref{95}.
Then there are one-to-one correspondences
$$
{\left\{\begin{matrix}
\text{thick subcategories}\\
\text{of $\ds(R)$}
\end{matrix}\right\}}
\cong
{\left\{\begin{matrix}
\text{thick subcategories}\\
\text{of $\db(R)$ containing $R$}
\end{matrix}\right\}}
\overset{\rm(a)}{\cong}
{\left\{\begin{matrix}
\text{specialization-closed}\\
\text{subsets of $V$}
\end{matrix}\right\}}.
$$
\end{thm}

\noindent
Here, $\db(R)$ denotes the bounded derived category of the category $\mod R$ of finitely generated $R$-modules, and $\ds(R)$ stands for the {\em singularity category} of $R$, that is to say, the Verdier quotient of $\db(R)$ by the full subcategory $\dpf(R)$ of perfect complexes, i.e., $\ds(R)=\db(R)/\dpf(R)$.

A {\em resolving subcategory} of an abelian cateory is a full subcategory containing projectives and closed under direct summands, extensions and syzygies.
This notion has been studied in various approaches so far; see \cite{APST,AS,radius,dim,crspd,HKV,H-Z,KS,sg,SSW,res,stcm,arg,crs} for instance.
In commutative algebra, Dao and Takahashi \cite{crspd} gave a complete classification of the resolving subcategories of $\mod R$ under the setup introduced above.

\begin{thm}[Dao--Takahashi]\label{91}
Let $(R,V)$ be as in Setup \ref{95}.
Then there is a one-to-one correspondence
$$
{\left\{\begin{matrix}
\text{resolving subcategories}\\
\text{of $\mod R$}
\end{matrix}\right\}}
\overset{\rm(b)}{\cong}
{\left\{\begin{matrix}
\text{grade-consistent}\\
\text{functions on $\spec R$}
\end{matrix}\right\}}
\times
{\left\{\begin{matrix}
\text{specialization-closed}\\
\text{subsets of $V$}
\end{matrix}\right\}}.
$$
\end{thm}

\noindent
Here, a {\em grade-consistent function} on $\spec R$ is an order-preserving map $f:\spec R\to\N$ which satisfies the inequality $f(\p)\le\grade(\p)$ for every $\p\in\spec R$.

The notion of a {\em $t$-structure} in a triangulated category has been introduced by Be\u{\i}linson, Bernstein and Deligne \cite{BBD} in the 1980s.
As with classifying thick subcategories and resolving subcategories mentioned above, classifying $t$-structures in a given triangulated category $\T$, which is equivalent to classifying {\em aisles} of $\T$, has been an important fundamental problem.
Actually, this problem has almost been settled for $\db(R)$ for a commutative noetherian ring $R$.
Indeed, if $R$ has a dualizing complex, then the aisles of $\db(R)$ were completely classified by Alonso Tarr\'{i}o, Jerem\'{i}as L\'opez and Saor\'{i}n \cite{AJS} in terms of the filtrations by supports that satisfy the weak Cousin condition.
Recently, this has been extended by Takahashi \cite{ft} to the case where $R$ has finite Krull dimension such that $\spec R$ is a CM-excellent scheme in the sense of \v{C}esnavi\v{c}ius \cite{Ce}.

Now that classifying aisles of $\db(R)$ has almost been completed, what we should consider next is classifying {\em preaisles} of $\db(R)$, which are defined as full subcategories closed under extensions and positive shifts.
An aisle is none other than a preaisle whose inclusion functor has a right adjoint, but there exists a big difference between being an aisle and being a preaisle.
Classifying preaisles is thus much harder than classifying aisles, and so it would be reasonable to impose some appropriate assumptions on the preaisles we try to classify.

The main result of this paper is the following theorem.
This theorem provides a classification of preaisles of $\db(R)$ that satisfy some mild and natural conditions.
Also, the theorem includes both the classification of thick subcategories by Stevenson and the classification of resolving subcategories by Dao and Takahashi.

\begin{thm}\label{1}
Let $(R,V)$ be a pair as in Setup \ref{95}.
Then there are one-to-one correspondences
$$
{\left\{\begin{matrix}
\text{preaisles of $\db(R)$}\\
\text{containing $R$ and closed}\\
\text{under direct summands}
\end{matrix}\right\}}
\cong
{\left\{\begin{matrix}
\text{resolving}\\
\text{subcategories}\\
\text{of $\db(R)$}
\end{matrix}\right\}}
\overset{(*)}{\cong}
{\left\{\begin{matrix}
\text{order-preserving}\\
\text{maps from $\spec R$}\\
\text{to $\N\cup\{\infty\}$}
\end{matrix}\right\}}\times
{\left\{\begin{matrix}
\text{specialization-closed}\\
\text{subsets of $V$}
\end{matrix}\right\}}.
$$
The restriction of $(*)$ to the thick subcategories of $\db(R)$ containing $R$ is identified with {\rm(a)} in Theorem \ref{90}.
Each resolving subcategory $\X$ of $\mod R$ equals the restriction to $\mod R$ of the smallest resolving subcategory $\widetilde\X$ of $\db(R)$ containing $\X$.
The composition of $(*)$ with the map $\X\mapsto\widetilde\X$ coincides with {\rm(b)} in Theorem \ref{91}.
\end{thm}

\noindent
Here, a {\em resolving subcategory} of $\db(R)$ is a full subcategory containing $R$ and closed under direct summands, extensions and negative shifts, which we shall newly introduce in this paper.
We adopt this name because it is viewed as a triangulated category version of a resolving subcategory of an abelian cateory stated above.

This paper is organized as follows.
In Section 2, together with several preliminaries for later sections, we give the precise definition of a resolving subcategory of $\db(R)$ and states its basic properties.
In Sections 3, 5 and 6, we mainly deal with perfect complexes.
In Section 3, we classify those resolving subcategories of $\dpf(R)$ whose objects are perfect complexes locally with nonpositive projective dimension; it turns out that they are totally ordered.
In Section 4, we introduce the key notion of NE-loci in $\db(R)$, which are regarded as extensions of nonfree loci in $\mod R$.
We find out several fundamental properties of NE-loci.
In Section 5, we provide a complete classification of the resolving subcategories of $\dpf(R)$ in terms of order-preserving maps from $\spec R$ to $\N\cup\{\infty\}$.
The proof uses induction whose basis is formed by the classification theorem obtained in Section 3.
The use of Koszul complexes and NE-loci is crucial here.
In Section 6, we realize results of Dao and Takahashi \cite{crspd} about modules of finite projective dimension, as restrictions of our results about perfect complexes, which are obtained in Sections 3 and 5.
In Section 7, we consider classifying preaisles of $\dpf(R)$ containing $R$ and closed under direct summands.
We also compare our results with a classification theorem of aisles given in \cite{ft}.
From Section 8 to 10 we mainly handle locally complete intersection rings.
In Section 8, we prove that the resolving subcategories of $\db(R)$ bijectively correspond to the direct product of the resolving subcategories of perfect complexes and the resolving subcategories of maximal Cohen--Macaulay complexes.
In Section 9, applying the result obtained in Section 8, we provide complete classifications of the resolving subcategories of $\db(R)$ and the preaisles of $\db(R)$ containing $R$ and closed under direct summands.
We also observe that this classification restricts to the classification of thick subcategories of $\ds(R)$ given in Theorem \ref{90}.
In Section 10, we realize the classification of resolving subcategories of $\mod R$ given in Theorem \ref{91} as a restriction of the classification of resolving subcategories of $\db(R)$ given in Section 9.

Finally, we should emphasize that some of our methods to investigate resolving subcategories of $\db(R)$ are similar to methods given in the literature to investigate resolving subcategories of $\mod R$, but we do need to invent and develop a lot of new techinques to obtain our results.
We should also emphasize that our main result, Theorem \ref{1}, is obtained at the end of this paper, by using results given in all the previous sections.

\section{Resolving subcategories of triangulated categories}

In this section, we state basic definitions which are used throughout this paper.
Mimicking the definition of a resolving subcategory of an abelian category, we define a resolving subcategory of a triangulated category.
We also explore fundamental properties of resolving subcategories.
We begin with giving our convention.

\begin{conv}
All subcategories are assumed to be strictly full.
An object $X$ of a category $\C$ is identified with the subcategory of $\C$ consisting of $X$.
An exact triangle $A\to B\to C\to A[1]$ is often abbreviated to $A\to B\to C\rightsquigarrow$.
Let $R$ be a commutative noetherian ring with identity.
For a prime ideal $\p$ of $R$, we denote by $\kappa(\p)$ the residue field of the local ring $R_\p$, that is, $\kappa(\p)=R_\p/\p R_\p$.
Subscripts and superscripts may be omitted if there is no danger of confusion.
\end{conv}

In the next two Definitions, we explain basic closedness conditions in an additive/abelian/triangulated category, and introduce certain subcategories determined by a given subcategory.

\begin{dfn}
Let $\C$ be an additive category, and let $\X$ be a subcategory of $\C$.
\begin{enumerate}[(1)]
\item
We say that $\X$ is {\em closed under finite direct sums} provided that for any finite number of objects $X_1,\dots,X_n$ in $\X$ the direct sum $X_1\oplus\cdots\oplus X_n$ is also in $\X$.
This is equivalent to saying that the direct sum of any two objects in $\X$ also belongs to $\X$.
\item
We say that $\X$ is {\em closed under direct summands} provided that if $X$ is an object in $\X$ and $Y$ is a direct summand of $X$ in $\A$, then $Y$ is also in $X$.
\item
We denote by $\add_\C\X$ the {\em additive closure} of $\X$, that is, the smallest subcategory of $\C$ containing $\X$ and closed under finite direct sums and direct summands.
\item
Assume $\C$ is abelian (resp. triangulated).
We say that $\X$ is {\em closed under extensions} provided for an exact sequence $0\to L\to M\to N\to0$ (resp. exact triangle $L\to M\to N\rightsquigarrow$) in $\C$, if $L,N\in\X$, then $M\in\X$.
\item
Suppose that $\C$ is either abelian or triangulated.
The {\em extension closure} $\ext_\C\X$ of $\X$ is defined as the smallest subcategory of $\C$ containing $\X$ and closed under direct summands and extensions.
\end{enumerate}
\end{dfn}

\begin{dfn}\label{100}
Let $\T$ be a triangulated category, and let $\X$ be a subcategory of $\T$.
\begin{enumerate}[(1)]
\item
For any $n\in\Z$ denote by $\X[n]$ the subcategory of $\T$ consisting of objects of the form $X[n]$ with $X\in\X$.
\item
We say that $\X$ is {\em closed under positive shifts} (resp. {\em closed under negative shifts}) if $\X[n]$ is contained in $\X$ for all $n>0$ (resp. $n<0$), which is equivalent to saying that $\X[1]$ (resp. $\X[-1]$) is contained in $\X$.
\item
We say that $\X$ is {\em thick} if $\X$ is a nonempty triangulated subcategory of $\T$ closed under direct summands.
We denote by $\thick_\T\X$ the {\em thick closure} of $\X$, namely, the smallest thick subcategory of $\T$ containing $\X$.
\end{enumerate}
\end{dfn}

Here we compare closedness under positive/negative shifts with other conditions regarding subcategories of a triangulated category.

\begin{prop}\label{42}
Let $\T$ be a triangulated category.
Let $\X$ be a subcategory of $\T$.
\begin{enumerate}[\rm(1)]
\item
Suppose $\X$ is closed under extensions and contains the zero object of $\T$.
Then the following are equivalent.
\begin{enumerate}[\rm(a)]
\item
The subcategory $\X$ of $\T$ is closed under positive (resp. negative) shifts.
\item
If $A\to B\to C\rightsquigarrow$ is an exact triangle with $A,B\in\X$ (resp. $B,C\in\X$), then $C$ (resp. $A$) is in $\X$.
\end{enumerate}
\item
Suppose that $\X$ is nonempty.
Then $\X$ is a thick subcategory of $\T$ if and only if $\X$ is closed under direct summands, extensions, positive shifts and negative shifts.
\end{enumerate}
\end{prop}

\begin{proof}
(1) The exact triangle $A\to B\to C\rightsquigarrow$ induces exact triangles $B\to C\to A[1]\rightsquigarrow$ and $C[-1]\to A\to B\rightsquigarrow$.
This shows that (a) implies (b).
For each object $X\in\T$ there exist exact triangles $X\to0\to X[1]\rightsquigarrow$ and $X[-1]\to0\to X\rightsquigarrow$ in $\T$.
This shows that (b) implies (a).

(2) By assumption, there exists an object $X$ in $\X$.
Suppose that $\X$ is closed under direct summands and extensions.
Then $\X$ contains $0$, since $0$ is a direct summand of $X$.
Now the assertion follows from (1).
\end{proof}

We introduce categories and subcategories which we basically use in this paper.

\begin{dfn}
We denote by $\mod R$ the category of finitely generated $R$-modules, by $\proj R$ the subcategory of $\mod R$ consisting of projective modules, and by $\fpd R$ the subcategory of $\mod R$ consisting of modules of finite projective dimension.
Also, we denote by $\d(R)$ the bounded derived category of $\mod R$ which is denoted by $\db(R)$ in Section 1, and by $\k(R)$ the bounded homotopy category of $\proj R$.
Via the natural fully faithful functors, we regard $\mod R$ and $\k(R)$ as (strictly full) subcategories of $\d(R)$.
In particular, we identify $\k(R)$ with the derived category $\dpf(R)$ of perfect $R$-complexes that appears in Section 1.
Here, a {\em perfect} complex is defined to be a bounded complex of finitely generated projective modules.
We thus have inclusions
$$
\proj R\subseteq\fpd R=\k(R)\cap\mod R,\qquad
\fpd R\subseteq\k(R)\subseteq\d(R),\qquad
\fpd R\subseteq\mod R\subseteq\d(R).
$$
\end{dfn}

Now we remind ourselves of the definition of a resolving subcategory of the module category.

\begin{dfn}
Let $\X$ be a subcategory of $\mod R$.
\begin{enumerate}[(1)]
\item
We say that $\X$ is {\em resolving} if it satisfies the following four conditions.
\begin{quote}
(i) $\X$ contains $R$.\quad
(ii) $\X$ is closed under direct summands.\quad
(iii) $\X$ is closed under extensions.\\
(iv) For an exact sequence $0\to A\to B\to C\to0$ in $\mod R$ with $B,C\in\X$, one has $A\in\X$.
\end{quote}
Conditions (i) and (ii) imply that a resolving subcategory of $\mod R$ contains the zero object $0$ of $\mod R$.
Condition (i) can be replaced with the condition that $\X$ contains $\proj R$.
Condition (iv) can be replaced with the condition that $\X$ is closed under syzygies; see \cite[Remark 2.3]{res}.
\item
The {\em resolving closure} $\res_{\mod R}\X$ of $\X$ is the smallest resolving subcategory of $\mod R$ containing $\X$.
\end{enumerate}
\end{dfn}

Imitating this definition, we shall introduce the notion of a resolving subcategory of a triangulated category.

\begin{dfn}
Let $\T$ be a triangulated subcategory of $\d(R)$ containing $R$.
Let $\X$ be a subcategory of $\T$.
\begin{enumerate}[(1)]
\item
We say that $\X$ is {\em resolving} if it satisfies the following four conditions.
\begin{quote}
(i) $\X$ contains $R$.\quad
(ii) $\X$ is closed under direct summands.\quad
(iii) $\X$ is closed under extensions.\\
(iv) For an exact triangle $A\to B\to C\rightsquigarrow$ in $\T$ with $B,C\in\X$, one has $A\in\X$.
\end{quote}
Conditions (i) and (ii) imply that a resolving subcategory of $\T$ contains the zero object $0$ of $\T$.
Condition (i) can be replaced with the condition that $\X$ contains $\proj R$.
Condition (iv) can be replaced with the condition that $\X$ is closed under negative shifts; see Proposition \ref{42}(1).
\item
The {\em resolving closure} $\res_\T\X$ of $\X$ is defined to be the smallest resolving subcategory of $\T$ containing $\X$.
\end{enumerate}
\end{dfn}

In the next proposition we explore the relationship between resolving closures and shifts.
It turns out that compatibility of taking the resolving closure and taking a shift is subtle; see also Remark \ref{99} given later.

\begin{prop}\label{52}
Let $\T$ be a triangulated subcategory of $\d(R)$ containing $R$.
\begin{enumerate}[\rm(1)]
\item
For each object $X$ of $\T$ and each integer $n$, there is an equality $\res_\T\{X[i]\mid i\in\Z\}=\res_\T\{X[i]\mid i\ge n\}$.
\item
Let $\X$ be a subcategory of $\T$, and let $n$ be an integer.
\begin{enumerate}[\rm(a)]
\item
Let $n\le0$.
Then there is an inclusion $(\res_\T\X)[n]\subseteq\res_\T(\X[n])$.
\item
Let $n\ge0$.
If $\X$ is resolving, then so is $\X[n]$.
More generally, $(\res_\T\X)[n]=\res_\T(\X[n]\cup\{R[n]\})$.
\end{enumerate}
\end{enumerate}
\end{prop}

\begin{proof}
(1) Set $\X=\res_\T\{X[i]\mid i\ge n\}$.
Fix $j\in\Z$.
If $j\ge n$, then $X[j]$ is clearly in $\X$.
If $j<n$, then $j-n<0$ and one has $X[j]=(X[n])[j-n]\in\res( X[n])\subseteq\X$.
Hence $X[j]\in\X$ for all $j\in\Z$, and the assertion follows.

(2a) Consider the subcategory $\Y=\{Y\in\T\mid Y[n]\in\res_\T(\X[n])\}$ of $\T$.
Since $\res_\T(\X[n])$ is resolving, it contains $R$.
As $n\le0$, we have $R[n]\in\res_\T(\X[n])$.
Hence $R$ belongs to $\Y$.
Let $Y$ be an object in $\Y$ and $Z$ a direct summand of $Y$.
Then $Y[n]$ is in $\res_\T(\X[n])$ and $Z[n]$ is a direct summand of $Y[n]$.
Hence $Z[n]$ is in $\res_\T(\X[n])$, which implies $Z\in\Y$.
Let $A\to B\to C\rightsquigarrow$ be an exact triangle in $\T$ with $C\in\Y$.
Then there is an exact triangle $A[n]\to B[n]\to C[n]\rightsquigarrow$ and $C[n]$ is in $\res_\T(\X[n])$.
Hence $A[n]\in\res_\T(\X[n])$ if and only if $B[n]\in\res_\T(\X[n])$.
Therefore, $A\in\Y$ if and only if $B\in\Y$.
Consequently, $\Y$ is a resolving subcategory of $\T$.
Since $\Y$ contains $\X$, we see that $\Y$ contains $\res_\T\X$.
It follows that $(\res_\T\X)[n]\subseteq\res_\T(\X[n])$.

(2b) To show the first assertion, suppose that $\X$ is resolving.
As $\X$ is closed under negative shifts, we have that $\X[-1]\subseteq\X$, and that $R[-n]\in\X$ since $R\in\X$ and $-n\le0$.
Hence $(\X[n])[-1]=(\X[-1])[n]\subseteq\X[n]$ and $R=(R[-n])[n]\in\X[n]$, that is, $\X[n]$ is closed under negative shifts and contains $R$.
Let $A\to B\to C\rightsquigarrow$ be an exact triangle in $\T$ with $A,C\in\X[n]$.
Then there is an exact triangle $A[-n]\to B[-n]\to C[-n]\rightsquigarrow$ in $\T$ and $A[-n],C[-n]\in\X$.
Since $\X$ is closed under extensions, it contains $B[-n]$.
Hence $B=(B[-n])[n]$ belongs to $\X[n]$, and therefore $\X[n]$ is closed under extensions.
Let $K$ be an object in $\X[n]$ and $L$ a direct summand of $K$.
Then $K[-n]$ is in $\X$ and $L[-n]$ is a direct summand of $K[-n]$.
As $\X$ is closed under direct summands, $L[-n]$ is in $\X$.
Hence $L$ belongs to $\X[n]$, and therefore $\X[n]$ is closed under direct summands.
Consequently, $\X[n]$ is a resolving subcategory of $\T$.

Now we prove the second assertion.
Replacing $\X$ with $\X\cup\{R\}$, we may assume $R\in\X$.
We want to deduce $(\res_\T\X)[n]=\res_\T(\X[n])$.
As $\res_\T\X\supseteq\X$, we have $(\res_\T\X)[n]\supseteq\X[n]$.
As $(\res_\T\X)[n]$ is resolving by the first assertion, $(\res_\T\X)[n]$ contains $\res_\T(\X[n])$.
To show the opposite inclusion, consider the subcategory $\Y=\{Y\in\T\mid Y[n]\in\res_\T(\X[n])\}$ of $\T$.
Since $R\in\X$, we get $R[n]\in\X[n]\subseteq\res_\T(\X[n])$, which implies $R\in\Y$.
An analogous argument as in the proof of (1) shows $\Y$ is a resolving subcategory of $\T$.
Since $\Y$ contains $\X$, it contains $\res_\T\X$.
Thus $(\res_\T\X)[n]\subseteq\res_\T(\X[n])$.
We now obtain the equality $(\res_\T\X)[n]=\res_\T(\X[n])$.
\end{proof}

We define the minimum resolving subcategory, which plays a crucial role in the proofs of our main results.

\begin{dfn}
Let $\T$ be a triangulated subcategory of $\d(R)$ containing $R$.
We set $\E_\T=\res_\T0$ and call it the {\em minimum resolving subcategory} of $\T$.
It is minimum in the sense that every resolving subcategory of $\T$ contains $\E_\T$.
We simply write $\E_R=\E_{\d(R)}$.
\end{dfn}

The resolving closure $\res_\T\X$ of a subcategory $\X$ of $\T$, particularly the minimum resolving subcategory $\E_\T=\res_\T0$ of $\T$, depends on which triangulated subcategory $\T$ of $\d(R)$ is taken as the ambient category.
The proposition below collects properties of resolving subcategories, the second and third of which produce sufficient conditions for $\T$ to satisfy $\res_\T\X=\res_{\d(R)}\X$ for every subcategory $\X$ of $\T$.

\begin{prop}\label{43}
The following assertions hold true.
\begin{enumerate}[\rm(1)]
\item
Let $\T$ be a triangulated subcategory of $\d(R)$ containing $R$.
Then $\E_\T=\res_\T0=\res_\T R=\res_\T(\proj R)$.
\item
Let $\T$ be a thick subcategory of $\d(R)$ containing $R$.
Then the resolving subcategories of $\T$ are the resolving subcategories of $\d(R)$ contained in $\T$.
Hence $\res_\T\X=\res_{\d(R)}\X$ for any subcategory $\X$ of $\T$.
\item
The equality $\k(R)=\thick_{\d(R)}R$ holds.
Hence, there is an equality $\res_{\k(R)}\X=\res_{\d(R)}\X$ for any subcategory $\X$ of $\k(R)$.
In particular, it holds that $\E_{\k(R)}=\E_R$.
\item
If $\X$ is a resolving subcategory of $\d(R)$, then $\X\cap\mod R$ is a resolving subcategory of $\mod R$.
If $\X$ is a resolving subcategory of $\k(R)$, then $\X\cap\mod R$ is a resolving subcategory of $\mod R$ contained in $\fpd R$.
\end{enumerate}
\end{prop}

\begin{proof}
(1) The assertion follows from the inclusions $\proj R\subseteq\E_\T=\res_\T0\subseteq\res_\T R\subseteq\res_\T(\proj R)$.

(2) Let $\X$ be a subcategory of $\T$ with $R\in\X$ and $\X[-1]\subseteq\X$.
Since $\T$ is closed under extensions as a subcategory of $\d(R)$, we see that $\X$ is closed under extensions as a subcategory of $\T$ if and only if $\X$ is closed under extensions as a subcategory of $\d(R)$.
If $A,B$ are objects of $\d(R)$ with $A\oplus B\in\X$, then $A\oplus B\in\T$, which implies $A,B\in\T$ since $\T$ is thick.
It is seen that $\X$ is closed under direct summands as a subcategory of $\T$ if and only if $\X$ is closed under direct summands as a subcategory of $\d(R)$.
The first assertion follows.

Since $\res_\T\X$ is a resolving subcategory of $\T$ containing $\X$, by the first assertion it is a resolving subcategory of $\d(R)$ contained in $\T$ and containing $\X$.
Hence $\T\supseteq\res_\T\X\supseteq\res_{\d(R)}\X$.
Therefore, $\res_{\d(R)}\X$ is a resolving subcategory of $\d(R)$ contained in $\T$ and containing $\X$, so that it is a resolving subcategory of $\T$ containing $\X$ by the first assertion again.
This implies that $\res_{\d(R)}\X\supseteq\res_\T\X$.
The second assertion now follows.

(3) It is a well-known fact that $\k(R)=\thick_{\d(R)}R$; see \cite[Proposition 1.4(2)]{dm} for instance.
In particular, $\k(R)$ is a thick subcategory of $\d(R)$ containing $R$.
It follows from (2) that $\res_{\k(R)}\X=\res_{\d(R)}\X$ for any subcategory $\X$ of $\k(R)$.
We get the equalities $\E_{\k(R)}=\res_{\k(R)}0=\res_{\d(R)}0=\E_{\d(R)}=\E_R$.

(4) Let $\X$ be a resolving subcategory of $\d(R)$.
Then $R$ belongs to $\X\cap\mod R$.
If $M\in\X\cap\mod R$ and $N$ is a direct summand in $\mod R$ of $M$, then $M$ is in $\X$ and $N$ is a direct summand in $\d(R)$ of $M$, so that $N$ is in $\X$ and hence $N\in\X\cap\mod R$.
Let $0\to A\to B\to C\to0$ be an exact sequence in $\mod R$ with $C\in\X\cap\mod R$.
Then there is an exact triangle $A\to B\to C\rightsquigarrow$ in $\d(R)$ and $C\in\X$.
Hence $A\in\X$ if and only if $B\in\X$, so that $A\in\X\cap\mod R$ if and only if $B\in\X\cap\mod R$.
Thus, $\X\cap\mod R$ is a resolving subcategory of $\mod R$.

Let $\X$ be a resolving subcategory of $\k(R)$.
By (2) and (3), $\X$ is a resolving subcategory of $\d(R)$ contained in $\k(R)$.
Hence $\X\cap\mod R$ is a resolving subcategory of $\mod R$ contained in $\k(R)\cap\mod R=\fpd R$.
\end{proof}

Here we state simple observations about representing each closure as an intersection of subcategories.

\begin{prop}
Let $\C$ be an additive category, and let $\X$ be a subcategory of $\C$.
\begin{enumerate}[\rm(1)]
\item
The additive closure $\add_\C\X$ is equal to the intersection of all subcategories of $\C$ that contain $\X$ and are closed under finite direct sums and direct summands.
\item
Assume that $\C$ is either abelian or triangulated.
Then the extension closure $\ext_\C\X$ is equal to the intersection of all subcategories of $\C$ that contain $\X$ and are closed under direct summands and extensions.
\item
Assume that $\C$ is triangulated.
Then the thick closure $\thick_\C\X$ is equal to the intersection of all thick subcategories of $\C$ containing $\X$.
\item
Assume that $\C$ is a triangulated subcategory of $\d(R)$ containing $R$.
Then the resolving closure $\res_\C\X$ is equal to the intersection of all resolving subcategories of $\C$ containing $\X$.
In particular, $\E_\C$ coincides with the intersection of all resolving subcategories of $\C$.
\end{enumerate}
\end{prop}

\begin{proof}
Let $\PP$ be a property of subcategories of $\C$ which is {\em preserved under intersections}, that is, for a family $\{\X_\lambda\}_{\lambda\in\Lambda}$ of subcategories of $\C$, if each $\X_\lambda$ satisfies $\PP$, then so does the intersection $\bigcap_{\lambda\in\Lambda}\X_\lambda$. 
For a subcategory $\X$ of $\C$, let $\PP_\C(\X)$ be the {\em $\PP$-closure} of $\X$, that is, the smallest subcategory of $\C$ satisfying $\PP$ and containing $\X$.
Then $\PP_\C(\X)$ coincides with the intersection of all subcategories of $\C$ that satisfy $\PP$ and contain $\X$.
\end{proof}

Next we recall the definitions of projective dimension and depth for complexes, and of Koszul complexes.

\begin{dfn}
\begin{enumerate}[(1)]
\item
The {\em supremum} $\sup X$ and {\em infimum} $\inf X$ of an object $X\in\d(R)$ is defined by $\sup X=\sup\{i\in\Z\mid\h^iX\ne0\}$ and $\inf X=\inf\{i\in\Z\mid\h^iX\ne0\}$.
\item
The {\em projective dimension} $\pd_RX$ of an object $X\in\d(R)$ is the infimum of integers $n$ such that $X\cong P$ in $\d(R)$ for some perfect $R$-complex $P$ with $P^{-i}=0$ for all integers $i>n$.
One has $\pd X\in\Z\cup\{\pm\infty\}$ and $\pd X\ge-\inf X$.
Also, $\pd X<\infty$ if and only if $X\in\k(R)$.
One does not necessarily have $\pd X\le n$ even if $X\cong P$ in $\d(R)$ for some complex $P$ of finitely generated projective modules with $P^{-i}=0$ for all $i>n$; see \cite[2.6.P]{AF}.
We refer to \cite[1.2.P, 1.7, 2.3.P, 2.4.P and 2.7.P]{AF} for details of projective dimension.
\item
For each integer $n$, we denote by $\k^n(R)$ the subcategory of $\k(R)$ consisting of perfect complexes having projective dimension at most $n$.
\item
For a sequence $\xx=x_1,\dots,x_n$ we denote by $\K(\xx,R)$ the Koszul complex of $\xx$ over $R$.
When the ambient ring $R$ is clear, we simply write $\K(\xx)$.
\item
Let $R$ be a local ring with residue field $k$.
For an object $X$ of $\d(R)$, we denote by $\depth_RX$ the {\em depth} of $X$, which is defined by the equality $\depth_RX=\inf\rhom_R(k,X)$.
\end{enumerate}
\end{dfn}

We make a collection of basic properties of projective dimension and depth which are frequently used later.

\begin{prop}\label{27}
\begin{enumerate}[\rm(1)]
\item
Let $X$ be an object of $\d(R)$, and let $r$ be an integer.
Then $\pd_R(X[r])=\pd_RX+r$.
When the ring $R$ is local, the equality $\depth_R(X[r])=\depth_RX-r$ holds.
\item
Let $R$ be a local ring with residue field $k$.
Let $X\in\d(R)$.
One has the equality $\pd_RX=-\inf(X\lten_Rk)$.
Also, the {\em Auslander--Buchsbaum formula} holds:
If $\pd_RX<\infty$, then $\pd_RX=\depth R-\depth_RX$.
\item
Let $A\to B\to C\rightsquigarrow$ be an exact triangle in $\d(R)$.
Then the following inequalities hold true,  where for the latter ones we assume that the ring $R$ is local.
$$
\begin{cases}
\pd_RB\le\sup\{\pd_RA,\pd_RC\},\\
\pd_RA\le\sup\{\pd_RB,\pd_RC-1\},\\
\pd_RC\le\sup\{\pd_RB,\pd_RA+1\},
\end{cases}
\begin{cases}
\depth_RB\ge\inf\{\depth_RA,\depth_RC\},\\
\depth_RA\ge\inf\{\depth_RB,\depth_RC+1\},\\\depth_RC\ge\inf\{\depth_RB,\depth_RA-1\}
\end{cases}
$$
\item
Let $X$ and $Y$ be objects of $\d(R)$.
Then one has the equality $\pd_R(X\oplus Y)=\sup\{\pd_RX,\pd_RY\}$.
When $R$ is local, one also has the equality $\depth_R(X\oplus Y)=\inf\{\depth_RX,\depth_RY\}$.
\item
For all nonnegative integers $n$, the subcategory $\k^n(R)$ of $\k(R)$ is resolving.
\item
There is an equality $\k^0(R)=\E_R$.
In particular, the equality $\E_R\cap\mod R=\proj R$ holds.
\item
Let $R$ be a local ring with maximal ideal $\m$.
Let $\xx=x_1,\dots,x_n$ be a sequence of elements of $R$.
If $x_i\in\m$ for all $i$, then $\pd_R\K(\xx)=n$.
If $x_i\notin\m$ for some $i$, then $\K(\xx)\cong0$ in $\k(R)$ and $\pd_R\K(\xx)=-\infty$.
\end{enumerate}
\end{prop}

\begin{proof}
(1) We easily deduce the assertion from the definitions of projective dimension and depth.

(2) The first assertion follows from \cite[(A.5.7.2)]{C}.
The second assertion is stated in \cite[(1.5)]{CFF} for example.

(3) Suppose that $R$ is a local ring with residue field $k$.
The two exact triangles
$$
\rhom_R(k,A)\to\rhom_R(k,B)\to\rhom_R(k,C)\rightsquigarrow,\qquad
A\lten_Rk\to B\lten_Rk\to C\lten_Rk\rightsquigarrow
$$
give rise to the inequalities $\inf\rhom_R(k,B)\ge\inf\{\inf\rhom_R(k,A),\inf\rhom_R(k,C)\}$ and $\inf(B\lten_Rk)\ge\inf\{\inf(A\lten_Rk),\inf(C\lten_Rk)\}$.
Therefore we have $\depth B\ge\inf\{\depth A,\depth C\}$, and by (2) we get
$$
\begin{array}{l}
\pd_RB=-\inf(B\lten_Rk)\le-\inf\{\inf(A\lten_Rk),\inf(C\lten_Rk)\}\\
\phantom{\pd_RB=-\inf(B\lten_Rk)}=\sup\{-\inf(A\lten_Rk),-\inf(C\lten_Rk)\}=\sup\{\pd_RA,\pd_RC\}.
\end{array}
$$
Now we consider the case where $R$ is nonlocal.
Using \cite[Proposition 5.3.P]{AF} and the local case, we get
$$
\textstyle\pd_RB=\sup_{\p\in\spec R}\{\pd_{R_\p}B_\p\}\le\sup_{\p\in\spec R}\{\sup\{\pd_{R_\p}A_\p,\pd_{R_\p}C_\p\}\}\le\sup\{\pd_RA,\pd_RC\}.
$$
Applying the argument given so far to the exact triangles $C[-1]\to A\to B\rightsquigarrow$ and $B\to C\to A[1]\rightsquigarrow$ and using (1), we obtain the remaining four inequalities.

(4) Suppose that the ring $R$ is local, and let $k$ be the residue field of $R$.
Using (2) for the former, we have
$$
\begin{array}{l}
\pd_R(X\oplus Y)=-\inf((X\oplus Y)\lten_Rk)=-\inf((X\lten_Rk)\oplus(Y\lten_Rk))\\
\quad=-\inf\{\inf(X\lten_Rk),\inf(Y\lten_Rk)\}
=\sup\{-\inf(X\lten_Rk),-\inf(Y\lten_Rk)\}=\sup\{\pd_RX,\pd_RY\},\\
\depth_R(X\oplus Y)=\inf\rhom_R(k,X\oplus Y)=\inf(\rhom_R(k,X)\oplus\rhom_R(k,Y))\\
\quad=\inf\{\inf\rhom_R(k,X),\inf\rhom_R(k,Y)\}
=\inf\{\depth_RX,\depth_RY\}.
\end{array}
$$

Now let the ring $R$ be nonlocal.
Applying (3) to the exact triangle $X\to X\oplus Y\to Y\rightsquigarrow$ gives $\pd_R(X\oplus Y)\le\sup\{\pd_RX,\pd_RY\}$.
Assume that $\pd_R(X\oplus Y)<\sup\{\pd_RX,\pd_RY\}$.
We may assume $\pd_RX\ge\pd_RY$.

We claim that if $\pd_{R_\p}X_\p<\infty$ for all prime ideals $\p$ of $R$, then $\pd_RX<\infty$.
Indeed, putting $t=\inf X$ and $s=\sup X$, we find a complex $P=(\cdots\to P^t\to\cdots\to P^s\to0)$ of finitely generated projective $R$-modules such that $P\cong X$ in $\d(R)$.
Let $C$ be the cokernel of the map $P^{t-1}\to P^t$.
Let $Q=(0\to P^{t+1}\to\cdots\to P^s\to0)$ be the truncation of $P$, which is a perfect complex.
There is an exact triangle $Q\to P\to C[-t]\rightsquigarrow$.
For each $\p\in\spec R$ we have $\pd_{R_\p}Q_\p<\infty$ and  $\pd_{R_\p}P_\p=\pd_{R_\p}X_\p<\infty$, so that $\pd_{R_\p}C_\p<\infty$.
It follows from \cite[Lemma 4.5]{BM} that $\pd_RC<\infty$.
As $\pd_RQ<\infty$, we get $\pd_RX=\pd_RP<\infty$.
The claim thus follows.

The claim and \cite[Proposition 5.3.P]{AF} produce a prime ideal $\p$ such that $\pd_RX=\pd_{R_\p}X_\p\le\infty$.
We have
$$
\begin{array}{l}
\pd_{R_\p}X_\p\le\sup\{\pd_{R_\p}X_\p,\pd_{R_\p}Y_\p\}=\pd_{R_\p}(X_\p\oplus Y_\p)\\
\phantom{\pd_{R_\p}X_\p\le\sup\{\pd_{R_\p}X_\p,\pd_{R_\p}Y_\p\}}\le\pd_R(X\oplus Y)<\sup\{\pd_RX,\pd_RY\}=\pd_RX=\pd_{R_\p}X_\p,
\end{array}
$$
where the first equality holds since the ring $R_\p$ is local.
We now get a contradiction, and therefore, the equality $\pd_R(X\oplus Y)=\sup\{\pd_RX,\pd_RY\}$ holds.

(5) As $n$ is nonnegative, $R$ belongs to $\k^n(R)$.
The assertion is shown to hold by using (3) and (4).

(6) As $\E_R$ contains $R$ and is closed under negative shifts, it contains $R[i]$ for all $i\le0$.
Hence we will get the required equality $\k^0(R)=\E_R$ once we prove that the following inclusions hold.
$$
\k^0(R)\subseteq\ext_{\k(R)}\{R[i]\mid i\le0\}\subseteq\E_R\subseteq\k^0(R).
$$
The second inclusion holds since $\E_R$ is closed under extensions, while the last inclusion comes from the fact that $\k^0(R)$ is a resolving subcategory and $\E_R$ is a minimum resolving subcategory.
To show the first inclusion, pick an object $P$ in $\k^0(R)$.
We may assume that $P=(0\to P^0\to P^1\to\cdots\to P^s\to0)$.
Then $P$ belongs to $\ext_{\k(R)}\{P^s[-s],\dots,P^1[-1],P^0\}$, which is contained in $\ext_{\k(R)}\{R[i]\mid i\le0\}$.
Thus the first inclusion follows.

(7) If $x_i\in\m$ for all $i$, then $\pd_R\K(\xx)=n$ by (2).
If $x_i\notin\m$ for some $i$, then $\K(x_i)\cong0$ in $\d(R)$, and hence $\K(\xx)\cong\K(x_1)\lten_R\cdots\lten_R\K(x_i)\lten_R\cdots\lten_R\K(x_n)\cong0$ in $\d(R)$.
Hence $\K(\xx)\cong0$ in $\k(R)$.
\end{proof}

Here, let us present an application of the above proposition.
The corollary below is thought of as a derived category version of \cite[Proposition 1.12(2)]{stcm}.

\begin{cor}
Let $R$ be a local ring.
Let $X$ and $Y$ be complexes that belong to $\d(R)$.
Suppose that $X$ is in the resolving closure $\res_{\d(R)}Y$.
Then there is an inequality $\depth_RX\ge\inf\{\depth_RY,\depth R\}$.
\end{cor}

\begin{proof}
Let $\ZZ$ be the subcategory of $\d(R)$ consisting of objects $Z$ such that $\depth Z\ge\inf\{\depth Y,\depth R\}$.
It is evident that $\ZZ$ contains $Y$ and $R$.
Using the depth equality in Proposition \ref{27}(4), we see that $\ZZ$ is closed under direct summands.
Also, the first depth inequality in Proposition \ref{27}(3) shows that $\ZZ$ is closed under extensions.
By the depth equality in Proposition \ref{27}(1), it follows that $\ZZ$ is closed under negative shifts.
Consequently, $\ZZ$ is a resolving subcategory of $\d(R)$ containing $Y$.
Hence $\ZZ$ contains $\res_{\d(R)}Y$, and therefore $X$ belongs to $\ZZ$.
Now the assertion of the corollary follows.
\end{proof}

By definition, a thick subcategory of $\d(R)$ containing $R$ is a resolving subcategory of $\d(R)$.
The converse of this statement is not necessarily true.
Actually, we state and prove the following proposition, which gives rise to an example of a resolving subcategory of $\d(R)$ that is not a thick subcategory of $\d(R)$.

\begin{prop}\label{79}
The equality $\res_{\d(R)}(\mod R)=\{X\in\d(R)\mid\h^{<0}X=0\}$ of subcategories of $\d(R)$ holds.
Thus, the resolving subcategory $\res_{\d(R)}(\mod R)$ of $\d(R)$ is not thick; it is not closed under positive shifts.
\end{prop}

\begin{proof}
Let $\X$ be the subcategory of $\d(R)$ consisting of complexes $X$ with $\h^{<0}X=0$.
Evidently, $\X$ contains $\mod R$.
In particular, $\X$ contains $R$.
It is straightforward to verify that $\X$ is closed under direct summands, extensions, and negative shifts.
Hence $\X$ is a resolving subcategory of $\d(R)$ containing $\mod R$.
Therefore, $\X$ contains $\res_{\d(R)}(\mod R)$.
Conversely, pick $X\in\X$.
Since $\h^{<0}X=0$, we may assume $X=(0\to X^0\to X^1\to\cdots\to X^n\to0)$; see \cite[(A.1.14)]{C}.
There is a series $\{X^{i}[-i]\to X_{i}\to X_{i-1}\rightsquigarrow\}_{i=0}^n$ of exact triangles in $\d(R)$ with $X_n=X$ and $X_{-1}=0$.
The object $X^i[-i]$ is in $\res_{\d(R)}(\mod R)$ for all $0\le i\le n$, since $X^i\in\mod R$ and $-i\le0$.
It is observed that $X$ belongs to $\res_{\d(R)}(\mod R)$.
Therefore, $\X$ is contained in $\res_{\d(R)}(\mod R)$.

As for the last assertion of the proposition, we have $R\in\X$, but $R[1]\notin\X$ since $\h^{-1}(R[1])=R\ne0$.
\end{proof}

We close the section by stating a remark on the second assertion of Proposition \ref{52}.

\begin{rem}\label{99}
Let $\T$ be a triangulated subcategory of $\d(R)$ containing $R$.
Let $X$ and $Y$ be objects of $\T$.
Assume that $X$ belongs to $\res Y$.
Then $X[n]$ belongs to $\res(Y[n])$ if $n\le0$ by Proposition \ref{52}(2a).
However, $X[n]$ does not necessarily belong to $\res(Y[n])$, if $n>0$.
In fact, we have the following observations.
\begin{enumerate}[(1)]
\item
Let $X=R$ and $Y=R[-1]$.
Then $X\in\res_{\d(R)}Y$, but $X[1]=R[1]\notin\E_R=\res_{\d(R)}R=\res_{\d(R)}(Y[1])$; see Proposition \ref{43}(1).
Indeed, we have $\pd R[1]=1$ and $R[1]\notin\k^0(R)=\E_R$ by Proposition \ref{27}(1)(6).
\item
Suppose that there exists an exact triangle $X\to E\to Y\rightsquigarrow$ in $\T$ such that $E\in\E_\T$.
Then $X$ belongs to $\res_\T Y$.
An exact triangle $X[1]\to E[1]\to Y[1]\rightsquigarrow$ in $\T$ is induced.
If $E[1]$ is in $\res_\T(Y[1])$, then $X[1]$ is in $\res_\T(Y[1])$.
However, as we have seen in (1), the object $E[1]$ may not belong to $\res_\T(Y[1])$.
\end{enumerate}
\end{rem}

\section{Classification of resolving subcategories of $\k_0(R)$}

In this section, we classify the resolving subcategories of $\k(R)$ consisting of perfect complexes locally of projective dimension at most zero.
In the next section we give a classification of resolving subcategories of $\k(R)$ by using an inductive argument, and what we get in the present section forms its basis.
First of all, we recall some basic notions and introduce certain subcategories of $\d(R)$ and $\k(R)$.

\begin{dfn}\label{86}
Let $R$ be a local ring with maximal ideal $\m$ and residue field $k$.
\begin{enumerate}[(1)]
\item
We denote by $\edim R$ the {\em embedding dimension} of $R$, that is, the number of elements in a minimal system of generators of $\m$, which is equal to the dimension of the $k$-vector space $\m\otimes_Rk=\m/\m^2$.
\item
For a minimal system of generators $\xx$ of $\m$, we set $\K_R=\K(\xx,R)$ and call it the {\em Koszul complex} of $R$.
\item
We denote by $\d_0(R)$ the subcategory of $\d(R)$ consisting of complexes $X$ with $X_\p\in\E_{R_\p}$ (or in other words, $\pd_{R_\p}X_\p\le0$ by Proposition \ref{27}(6)) for all prime ideals $\p$ of $R$ such that $\p\ne\m$.
We set
$$
\k_0(R)=\k(R)\cap\d_0(R),\qquad
\k_0^n(R)=\k^n(R)\cap\k_0(R)=\k^n(R)\cap\d_0(R)\ \ \text{for }n\in\Z.
$$
\end{enumerate}
\end{dfn}

Here are a couple of statements about the notions just introduced.

\begin{prop}\label{21}
Let $R$ be a local ring.
Then the following statements are true.
\begin{enumerate}[\rm(1)]
\item
Let $X\in\d(R)$.
If $\h^iX$ has finite length as an $R$-module for all $i\in\Z$, then $X[i]\in\d_0(R)$ for all $i\in\Z$.
\item
The Koszul complex $\K_R$ of $R$ is uniquely determined up to complex isomorphism.
\item
Put $e=\edim R$ and $K=\K_R$.
One then has that $K[i]\in\k_0^{e+i}(R)\setminus\k_0^{e+i-1}(R)$ for each integer $i$.
\item
It holds that $\d_0(R)$ is a resolving subcategory of $\d(R)$.
Hence $\k_0(R)$ is a resolving subcategory of $\k(R)$, and so is $\k_0^n(R)$ for every nonnegative integer $n$.
\end{enumerate}
\end{prop}

\begin{proof}
(1) Let $\p$ be a nonmaximal prime ideal of $R$.
Let $i$ be an integer.
Then $\H^j((X[i])_\p)=(\H^{j+i}X)_\p=0$ for all $j\in\Z$, which means that $(X[i])_\p\cong0$ in $\d(R_\p)$.
Hence $(X[i])_\p$ belongs to $\E_{R_\p}$, so that $X[i]\in\d_0(R)$.

(2) The assertion is shown in \cite[page 52]{BH}.

(3) The complex $K[i]$ is in $\k_0(R)$ by (1).
Since $\pd K=e$, we have $\pd K[i]=e+i$ by Proposition \ref{27}(1).

(4) The first statement is deduced by using the fact that $\E_{R_\p}$ is a resolving subcategory of $\d(R_\p)$ for each prime ideal $\p$ of $R$.
The second statement follows from the first statement, Propositions \ref{27}(5), \ref{43}(2)(3) and the fact that the resolving property is preserved under taking intersections.
\end{proof}

We recall the definitions of an $R$-linear additive category and an ideal quotient of such a category.

\begin{dfn}
Let $\C$ be an {\em $R$-linear} additive category, that is, an additive category whose hom-sets are $R$-modules and composition of morphisms is $R$-bilinear.
\begin{enumerate}[(1)]
\item
An {\em ideal} $\I$ of $\C$ is by definition a family $\{\I(X,Y)\}_{X,Y\in\C}$ of $R$-submodules of $\Hom_\C(X,Y)$ such that $bfa\in\I(W,Z)$ for all $a\in\Hom_\C(W,X)$, $f\in\I(X,Y)$, $b\in\Hom_\C(Y,Z)$ and $W,X,Y,Z\in\C$.
The {\em ideal quotient} $\C/\I$ of $\C$ by $\I$ is by definition the category whose objects are those of $\C$ and morphisms are given by $\Hom_{\C/\I}(X,Y)=\Hom_\C(X,Y)/\I(X,Y)$ for $X,Y\in\C$.
Note that $\C/\I$ is an $R$-linear additive category.
\item
Let $\D$ be a subcategory of $\C$.
For two objects $X,Y$ of $\C$, let $[\D](X,Y)$ be the subset of $\Hom_\C(X,Y)$ consisting of morphisms that factor through some finite direct sums of objects in $\D$.
Then it is easy to observe that $[\D]$ is an ideal of $\C$, and hence the ideal quotient $\C/[\D]$ is defined.
\end{enumerate}
\end{dfn}

Now we can define the category $\ld(R)$, which plays an important role in the remainder of this section.

\begin{dfn}
We denote by $\ld(R)$ the ideal quotient of $\d(R)/[\E_R]$.
The hom-set $\Hom_{\ld(R)}(X,Y)$ is a finitely generated $R$-module for all $X,Y\in\ld(R)$, as it is a factor of the finitely generated $R$-module $\Hom_{\d(R)}(X,Y)$.
\end{dfn}

Let us investigate when an object and a morphism in $\d(R)$ are zero in the category $\ld(R)$.

\begin{prop}\label{23}
\begin{enumerate}[\rm(1)]
\item
A morphism in $\d(R)$ is zero in $\ld(R)$ if and only if it factors through an object in $\E_R$.
\item
Let $X\in\d(R)$.
The following are equivalent:\ \ 
{\rm(a)}\,$X\cong0$ in $\ld(R)$;\ {\rm(b)}\ $\Hom_{\ld(R)}(X,X)=0$;\ {\rm(c)}\ $X\in\E_R$.
\end{enumerate}
\end{prop}

\begin{proof}
(1) A morphism $f:X\to Y$ in $\d(R)$ is zero in $\ld(R)$ if and only if $f$ belongs to $[\E_R](X,Y)$, if and only if $f$ factors through some finite direct sum of objects in $\E_R$.
Since $\E_R$ is closed under finite direct sums, the last condition is equivalent to saying that $f$ factors through some object in $\E_R$.

(2) It is clear that (a) implies (b).
Assume (b).
Then by (1) the identity morphism $X\to X$ factors through some object $E\in\E_R$.
It is seen that $X$ is a direct summand of $E$.
As $\E_R$ is closed under direct summands, $X$ is in $\E_R$.
Therefore, (b) implies (c).
Finally, assume (c).
Then every morphism from/to $X$ factors through $X\in\E_R$.
Hence $\Hom_{\ld(R)}(X,Y)=0$ and $\Hom_{\ld(R)}(Y,X)=0$ for every object $Y\in\ld(R)$.
The former (resp. latter) means that $X$ is an initial (resp. a final) object of the additive category $\ld(R)$.
Thus (a) follows.
\end{proof}

We define the localization of a given subcategory of $\d(R)$ by a multiplicatively closed subset of $R$.

\begin{dfn}
Let $\X$ be a subcategory of $\d(R)$.
For a multiplicatively closed subset $S$ of $R$, we define the subcategory $\X_S$ of $\d(R_S)$ by $\X_S=\{X_S\mid X\in\X\}$.
When $S=R\setminus\p$ with $\p\in\spec R$, we set $\X_\p=\X_S$.
\end{dfn}

In the lemma below we study the structure of localizations of morphisms in the category $\ld(R)$.

\begin{lem}\label{16}
\begin{enumerate}[\rm(1)]
\item
Let $\p$ be a prime ideal of $R$.
Taking the localization of a morphism in $\d(R)$ at $\p$ induces an isomorphism $\Hom_{\ld(R)}(X,Y)_\p\to\Hom_{\ld(R_\p)}(X_\p,Y_\p)$ of $R_\p$-modules for all objects $X,Y\in\d(R)$.
\item
Suppose that $R$ is a local ring.
Let $X$ and $Y$ be objects of $\d(R)$.
Assume that either of the objects $X$ and $Y$ belongs to the subcategory $\d_0(R)$ of $\d(R)$.
Then the $R$-module $\Hom_{\ld(R)}(X,Y)$ has finite length. 
\end{enumerate}
\end{lem}

\begin{proof}
(1) Localization at $\p$ gives an isomorphism $\phi:\Hom_{\d(R)}(X,Y)_\p\to\Hom_{\d(R_\p)}(X_\p,Y_\p)$; see \cite[Lemma (A.4.5)]{C}.
If $P=(0\to P^0\to\cdots\to P^n\to0)$ is a perfect $R$-complex, then $P_\p=(0\to P^0_\p\to\cdots\to P^n_\p\to0)$ is a perfect $R_\p$-complex.
By Proposition \ref{27}(6), the subcategory $(\E_R)_\p$ of $\d(R_\p)$ is contained in $\E_{R_\p}$.
Thus $\phi$ restricts to an injection $\psi:([\E_R](X,Y))_\p\to[\E_{R_\p}](X_\p,Y_\p)$.
Let $Q=(0\to R_\p^{\oplus r_0}\to\cdots\to R_\p^{\oplus r_n}\to0)$ be a perfect $R_\p$-complex with $Q^i=R_\p^{\oplus r_i}$ for each $i\in\Z$.
Then we easily find a perfect $R$-complex $P=(0\to R^{\oplus r_0}\to\cdots\to R^{\oplus r_n}\to0)$ such that $P_\p$ is isomorphic to $Q$ as an $R_\p$-complex; see \cite[Lemma 4.2(1) and its proof]{ddc}.
Hence $P$ belongs to $\E_R$, and the equality $(\E_R)_\p=\E_{R_\p}$ follows.
Consider the decomposition
$$
(X_\p\xrightarrow{\frac{f}{s}}Y_\p)=(X_\p\xrightarrow{\frac{g}{t}}E_\p\xrightarrow{\frac{h}{u}}Y_\p)
$$
of a morphism $\frac{f}{s}:X_\p\to Y_\p$ in $\d(R_\p)$, where $f,g,h$ are morphisms in $\d(R)$, $s,t,u\in R\setminus\p$ and $E\in\E_R$.
Then $vutf=vshg$ for some $v\in R\setminus\p$.
We have $\frac{f}{s}=\frac{vutf}{vuts}$ and $(X\xrightarrow{vutf}Y)=(X\xrightarrow{g}E\xrightarrow{vsh}Y)$.
This shows that the injection $\psi$ is surjective.
Consequently, $\phi$ induces an isomorphism $\Hom_{\ld(R)}(X,Y)_\p\to\Hom_{\ld(R_\p)}(X_\p,Y_\p)$.

(2) Fix a nonmaximal prime ideal $\p$ of $R$.
Then $\Hom_{\ld(R)}(X,Y)_\p$ is isomorphic to $\Hom_{\ld(R_\p)}(X_\p,Y_\p)$ by (1), while either $X_\p$ or $Y_\p$ belongs to $\E_{R_\p}$.
It follows from Proposition \ref{23}(2) that $\Hom_{\ld(R_\p)}(X_\p,Y_\p)=0$.
\end{proof}

Here, we are necessary to prove an elementary lemma concerning a gerenal triangulated category, which produces a certain exact triangle of mapping cones.

\begin{lem}\label{15}
Let $\T$ be a triangulated category.
\begin{enumerate}[\rm(1)]
\item
Suppose that there exists a commutative diagram 
$$
\xymatrix@R-1pc@C+2pc{
A\ar[r]\ar@{=}[d]& B\ar[r]\ar[d]& C\ar[r]\ar[d]& A[1]\ar@{=}[d]\\
A\ar[r]& B'\ar[r]& C'\ar[r]& A[1] 
}
$$
of exact triangles in $\T$.
Then there exists an exact triangle $B\to B'\oplus C\to C'\to B[1]$ in $\T$.
\item
Let $X\xrightarrow{f}Y\xrightarrow{g}Z$ be morphisms in $\T$.
Then there exists an exact triangle $\cone(gf)\to\cone(g)\oplus X[1]\to Y[1]\to\cone(gf)[1]$ in $\T$.
\end{enumerate}
\end{lem}

\begin{proof}
The first assertion immediately follows from \cite[Lemma 1.4.3]{N2}.
To show the second, let $X\xrightarrow{f}Y\xrightarrow{g}Z$ be morphisms in $\T$.
Then we have a commutative diagram of exact triangles at the lower left, which induces a commutative diagram of exact triangles at the lower right.
$$
\xymatrix@R-1pc{
X\ar[r]^{gf}\ar[d]_f& Z\ar[r]\ar@{=}[d]& \cone(gf)\ar[r]\ar[d]& X[1]\ar[d]\\
Y\ar[r]^g& Z\ar[r]& \cone(g)\ar[r]& Y[1]
}\qquad\qquad
\xymatrix@R-1pc{
Z\ar[r]\ar@{=}[d]& \cone(gf)\ar[r]\ar[d]& X[1]\ar[r]\ar[d]& Z[1]\ar@{=}[d]\\
Z\ar[r]& \cone(g)\ar[r]& Y[1]\ar[r]& Z[1] 
}
$$
By the first assertion, we get an exact triangle $\cone(gf)\to\cone(g)\oplus X[1]\to Y[1]\to\cone(gf)[1]$ in $\T$.
\end{proof}

Now, applying the previous two lemmas, we consider when a given object of the derived category $\d(R)$ belongs to the resolving closure of the (derived) tensor product with a Koszul complex.

\begin{lem}\label{17}
\begin{enumerate}[\rm(1)]
\item
For elements $x,y\in R$ there is an exact triangle $\K(x)\to\K(xy)\to\K(y)\rightsquigarrow$ in $\k(R)$.
\item
Let $X$ be an object of $\d(R)$ and let $x$ be an element of $R$.
Suppose that the morphism $X\xrightarrow{x}X$ in $\ld(R)$ defined by multiplication by $x$ is zero.
Then $X$ belongs to $\res_{\d(R)}(\K(x)\otimes_RX[-1])$.
\item
Suppose that $(R,\m)$ is local.
Let $X$ be an object in $\d_0(R)$.
Let $\xx=x_1,\dots,x_n$ be a sequence of elements in $\m$.
Then $X$ belongs to $\res_{\d(R)}(\K(\xx)\otimes_RX[-n])$.
In particular, $X$ is in $\res_{\d(R)}(\K_R\otimes_RX[-\edim R])$.
\end{enumerate}
\end{lem}

\begin{proof}
(1) The assertion is shown by applying the octahedral axiom to $(R\xrightarrow{xy}R)=(R\xrightarrow{x}R\xrightarrow{y}R)$.

(2) There exist morphisms $X\xrightarrow{f}E\xrightarrow{g}X$ in $\d(R)$ such that $E$ is in $\E_R$ and the composition $gf$ is equal to the mutiplication morphism $X\xrightarrow{x}X$ in $\d(R)$.
By Lemma \ref{15}(2) we have an exact triangle $\cone(gf)\to\cone(g)\oplus X[1]\to E[1]\rightsquigarrow$ in $\d(R)$.
The object $\cone(gf)$ is the mapping cone of the morphism $X\xrightarrow{x}X$, which is isomorphic to $\K(x)\otimes X$.
We get an exact triangle $\K(x)\otimes X[-1]\to\cone(g)[-1]\oplus X\to E\rightsquigarrow$ in $\d(R)$.
It follows that $X\in\ext_{\d(R)}\{\K(x)\otimes X[-1],E\}\subseteq\res_{\d(R)}(\K(x)\otimes X[-1])$.

(3) Lemma \ref{16}(2) says that the $R$-module $\Hom_{\ld(R)}(X,X)$ has finite length, and hence it is annihilated by some power $\m^r$ of $\m$.
Fix an element $x\in\m$.
Then the multiplication morphism $X\xrightarrow{x^r}X$ in $\ld(R)$ is zero.
By (2) the object $X$ is in $\res_{\d(R)}(\K(x^r)\otimes X[-1])$, which is contained in $\res_{\d(R)}(\K(x)\otimes X[-1])$ by (1).
It follows that $\res_{\d(R)}X$ is contained in $\res_{\d(R)}(\K(x)\otimes X[-1])$.
We observe that there is a sequence of inclusions
$$
\begin{array}{l}
\res_{\d(R)}X
\subseteq\res_{\d(R)}(\K(x_1)\otimes X[-1])
\subseteq\res_{\d(R)}(\K(x_2)\otimes\K(x_1)\otimes X[-2])\\
\phantom{\res_{\d(R)}X}\subseteq\cdots\subseteq\res_{\d(R)}(\K(x_n)\otimes\cdots\otimes\K(x_1)\otimes X[-n])
=\res_{\d(R)}(\K(\xx)\otimes X[-n]).
\end{array}
$$
Thus, $X$ belongs to the resolving closure $\res_{\d(R)}(\K(\xx)\otimes_RX[-n])$.
\end{proof}

Applying the above lemma, we can prove the following proposition about Koszul complexes, which is thought of as an essential part of the proof of the main result of this section.

\begin{prop}\label{14}
Let $R$ be a local ring.
Put $e=\edim R$ and $K=\K_R$.
\begin{enumerate}[\rm(1)]
\item
For every integer $n\ge0$ there is an equality $\k_0^n(R)=\res_{\k(R)}(K[n-e])$.
\item
Let $F$ be an object of $\k(R)$, and put $t=\pd_RF$.
Then the object $\K[t-e]$ belongs to $\res_{\k(R)}F$.
\end{enumerate}
\end{prop}

\begin{proof}
(1) We have $K[n-e]\in\k_0^n(R)$, so that $\res_{\k(R)}(K[n-e])\subseteq\k_0^n(R)$; see (3) and (4) of Proposition \ref{21}.
Conversely, pick an object $P\in\k_0^n(R)$.
Lemma \ref{17}(3) and Proposition \ref{43}(3) imply that $P$ belongs to $\res_{\k(R)}(K\otimes P[-e])$.
We may assume that the perfect complex $P$ has the form $(0\to P^{-n}\to P^{-n+1}\to\cdots\to P^s\to0)$.
Then it holds that $P\in\ext_{\k(R)}\{P^s[-s],\dots,P^{-n}[n]\}\subseteq\ext_{\k(R)}\{R[i]\mid-s\le i\le n\}$, and hence
$$
(K\otimes P)[-e]\in\ext_{\k(R)}\{K[i-e]\mid-s\le i\le n\}\subseteq\res_{\k(R)}(K[n-e]).
$$
Therefore, the complex $P$ belongs to $\res_{\k(R)}(K[n-e])$.

(2) We shall prove that $K[i-e]\in\res F$ for all $i\le t$.
For this we use induction on $t$.
When $t\le0$, for each $i\le t$ we have $\pd K[i-e]=\pd K+(i-e)=i\le t\le0$ by (1) and (7) of Proposition \ref{27}, and hence $K[i-e]\in\E_R\subseteq\res F$.
Let $t>0$.
We may assume that $F=(0\to F^{-t}\xrightarrow{\partial}F^{-t+1}\to\cdots\to F^s\to0)$, where $F^{-t},F^{-t+1},\dots,F^s$ are free, $F^{-t}\ne0$, $-t+1\le s$ and $\im\partial\subseteq\m F^{-t+1}$.
Since $\pd F[-1]=t-1<t$, the induction hypothesis implies $K[j-e]\in\res F[-1]$ for all $j\le t-1$.
As $F[-1]$ belongs to $\res F$, we see that
\begin{equation}\label{8}
\text{the object $K[j-e]$ belongs to $\res_{\k(R)}F$ for all integers $j\le t-1$.}
\end{equation}
It remains to show that $K[t-e]$ is in $\res F$.
Let $P=(0\to F^{-t}\xrightarrow{\partial}F^{-t+1}\to0)$ be a truncation of $F$.
Then there is an exact triangle $F^{-t}\xrightarrow{\partial}F^{-t+1}\to P[1-t]\rightsquigarrow$ in $\k(R)$.
Tensoring $K$ over $R$ gives an exact triangle
\begin{equation}\label{4}
F^{-t}\otimes K\xrightarrow{\partial\otimes K}F^{-t+1}\otimes K\to (P\otimes K)[1-t]\rightsquigarrow
\end{equation}
in $\k(R)$.
Write $F^{-t}=R^{\oplus n}$ and $F^{-t+1}=R^{\oplus m}$.
The inclusion $\im\partial\subseteq\m F^{-t+1}$ implies that the map $\partial:R^{\oplus n}\to R^{\oplus m}$ is represented by an $m\times n$ matrix $(a_{ij})$ with $a_{ij}\in\m$.
The chain map $\partial\otimes K:K^{\oplus n}\to K^{\oplus m}$ is given by the matrix $(a_{ij})$.
The multiplication morphism $K\xrightarrow{a_{ij}}K$ is zero in $\k(R)$ by \cite[Proposition 2.3(3)]{dm}, and so is the morphism $K^{\oplus n}\xrightarrow{(a_{ij})}K^{\oplus m}$; the matrix $(s_{ij})$ of null-homotopies $s_{ij}$ of $K\xrightarrow{a_{ij}}K$ is a null-homotopy of $K^{\oplus n}\xrightarrow{(a_{ij})}K^{\oplus m}$.
It follows from \eqref{4} that $(P\otimes K)[1-t]$ is isomorphic to the direct sum of $F^{-t+1}\otimes K=K^{\oplus m}$ and $(F^{-t}\otimes K)[1]=(K[1])^{\oplus n}$.
As $F^{-t}\ne0$, we have $n>0$ and the complex $K[1]$ is a direct summand of $(P\otimes K)[1-t]$ as an object of $\k(R)$.
Applying the functor $[t-e-1]$ shows that
\begin{equation}\label{5}
\text{the object $K[t-e]$ is a direct summand of the object$(P\otimes K)[-e]$ in $\k(R)$.}
\end{equation}
Let $Q=(0\to F^{-t+2}\to\cdots\to F^s\to0)$ be another truncation of $F$.
There is an exact triangle $Q\to F\to P\rightsquigarrow$, which induces an exact triangle $(Q\otimes K)[-e]\to(F\otimes K)[-e]\to(P\otimes K)[-e]\rightsquigarrow$.
This shows that
\begin{equation}\label{6}
\text{the object $(P\otimes K)[-e]$ belongs to $\ext_{\k(R)}\{(F\otimes K)[-e],(Q\otimes K)[1-e]\}$.}
\end{equation}
The Koszul complex $K=(0\to K^{-e}\to\cdots\to K^0\to0)$ is in $\ext\{K^{-i}[i]\mid0\le i\le e\}$, which is contained in $\ext\{R[i]\mid0\le i\le e\}$.
Applying $(F\otimes-)[-e]$ shows $(F\otimes K)[-e]$ is in $\ext\{F[i]\mid-e\le i\le0\}$, which implies
\begin{equation}\label{7}
\text{the object $(F\otimes K)[-e]$ belongs to $\res_{\k(R)}F$.}
\end{equation}
The perfect complex $Q=(0\to F^{-t+2}\to\cdots\to F^s\to0)$ is in $\ext_{\k(R)}\{F^i[-i]\mid-t+2\le i\le s\}$, which is contained in $\ext_{\k(R)}\{R[-i]\mid-t+2\le i\le s\}$.
Hence the object $(Q\otimes K)[1-e]$ belongs to the subcategory
$$
\ext_{\k(R)}\{K[1-e-i]\mid-t+2\le i\le s\}=\ext_{\k(R)}\{K[i]\mid(1-s)-e\le i\le(t-1)-e\}
$$
of $\k(R)$.
By \eqref{8}, this extension closure is contained in the resolving closure $\res_{\k(R)}F$.
Therefore,
\begin{equation}\label{9}
\text{the object $(Q\otimes K)[1-e]$ belongs to $\res_{\k(R)}F$.}
\end{equation}
It follows from \eqref{5}, \eqref{6}, \eqref{7} and \eqref{9} that $K[t-e]$ is in $\res_{\k(R)}F$ as desired.
\end{proof}

\begin{rem}
We may wonder if the Hopkins--Neeman classification theorem \cite{H,N} can be applied to deduce Proposition \ref{14}(2).
Actually, the proof of the Hopkins--Neeman theorem provides a certain integer $m$ such that $K[m]$ belongs to $\res_{\k(R)}F$.
However, this is done by applying the smash nilpotence theorem \cite[Theorem 1.1]{N}, which relies on the fact that $R$ is noetherian, so that $m$ cannot be described concretely.
Note that there is no meaning for us unless $m>-e$, since we know that $K[i]\in\E_R\subseteq\res_{\k(R)}F$ for all $i\le-e$.
\end{rem}

The following two corollaries are direct consequences of the above proposition.

\begin{cor}\label{19}
Suppose that $R$ be a local ring.
Then the following two statements hold true.
\begin{enumerate}[\rm(1)]
\item
There is an equality $\k_0(R)=\res_{\k(R)}\{\K_R[i]\mid i\in\Z\}$ of subcategories of $\k(R)$.
\item
Let $F$ be an object in $\k_0(R)$ and assume $\pd_RF=t\ge0$.
Then the equality $\res_{\k(R)}F=\k_0^t(R)$ holds.
\end{enumerate}
\end{cor}

\begin{proof}
Put $e=\edim R$, $K=\K_R$ and set $\X=\res\{K[i]\mid i\in\Z\}$.
Proposition \ref{21}(3) implies that $\X\subseteq\k_0(R)$.
Fix $F\in\k_0(R)$ and set $t=\pd F$.
If $t\le0$, then $F\in\E_R\subseteq\X$ by Proposition \ref{27}(6).
Let $t\ge0$.
We have
$$
F\in\k_0^t(R)=\res_{\k(R)}K[t-e]\subseteq\X\cap\res_{\k(R)}F\subseteq\res_{\k(R)}F\subseteq\k_0^t(R),
$$
where the equality and the first inclusion follow from Proposition \ref{14}, and the other inclusions are obvious.
We thus obtain the equalities $\X=\k_0(R)$ and $\k_0^t(R)=\res_{\k(R)}F$.
\end{proof}

\begin{cor}\label{20}
Let $R$ be a local ring.
Let $\X$ be a resolving subcategory of $\k(R)$ contained in $\k_0(R)$.
Suppose that one has $\sup_{X\in\X}\{\pd_RX\}=\infty$.
Then the equality $\X=\k_0(R)$ holds true.
\end{cor}

\begin{proof}
Assume that $\X$ is strictly contained in $\k_0(R)$.
Then there exists an object $Y\in\k_0(R)$ such that $Y\notin\X$.
Put $u=\pd_RY$.
As $Y$ is a nonzero object of $\k(R)$, we have that $-\infty<u<\infty$.
Since $\sup_{X\in\X}\{\pd X\}=\infty$, there exists an object $X\in\X$ such that $t:=\pd X\ge\max\{u,0\}$.
Then $X\in\k_0(R)$, $\pd X=t\ge0$ and $u\le t$.
Applying Corollary \ref{19}(2), we observe $Y\in\k_0^u(R)\subseteq\k_0^t(R)=\res_{\k(R)}X\subseteq\X$.
This gives a contradiction.
\end{proof}

Now we can give a proof of the main result of this section.
It provides an explicit description of the resolving subcategories of $\k(R)$ contained in $\k_0(R)$; in particular, they form a totally ordered set.

\begin{thm}\label{13}
Suppose that $R$ is a local ring with $e=\edim R$ and $K=\K_R$.
Then one has strict inclusions
\begin{equation}\label{45}
\E_R=\k_0^0(R)\subsetneq\k_0^1(R)\subsetneq\cdots\subsetneq\k_0^n(R)\subsetneq\k_0^{n+1}(R)\subsetneq\cdots\subsetneq\k_0(R)
\end{equation}
of resolving subcategories of $\k(R)$ such that $K[n-e]\in\k_0^n(R)\setminus\k_0^{n-1}(R)$ for each $n\ge1$.
Moreover, all the resolving subcategories of $\k(R)$ contained in $\k_0(R)$ appear in the above chain of subcategories of $\k(R)$.
\end{thm}

\begin{proof}
Proposition \ref{27}(6) says $\E_R=\k^0(R)\supseteq\k_0^0(R)$.
Since $\k_0^0(R)$ is resolving by Proposition \ref{21}(4) and $\E_R$ is the minimum resolving subcategory, the equality $\E_R=\k_0^0(R)$ holds.
For each $n\ge1$ it is clear that $\k_0^{n-1}(R)\subseteq\k_0^n(R)$, while $K[n-e]\in\k_0^n(R)\setminus\k_0^{n-1}(R)$ by Proposition \ref{21}(3).
The first assertion now follows.

Now, let us show the second assertion.
Let $\X$ be a resolving subcategory of $\k(R)$ contained in $\k_0(R)$.
We may assume that $\X$ is different from $\k_0(R)$.
Corollary \ref{20} says that $t:=\sup_{X\in\X}\{\pd X\}$ is finite, and in particular, $\X$ is contained in $\k_0^t(R)$.
Choose an an object $X\in\X$ such that $\pd X=t$.
We have $t\ge0$ as $R$ is in $\X$.
Using Corollary \ref{19}(2), we see that $\k_0^t(R)=\res_{\k(R)}X\subseteq\X$.
The equality $\X=\k_0^t(R)$ follows.
\end{proof}

\section{NE-loci of objects and subcategories of $\d(R)$}

Recall that the {\em nonfree locus} $\nf(M)$ of each object $M$ of $\mod R$ is by definition the set of prime ideals $\p$ of $R$ such that the localization $M_\p$ is nonfree as an $R_\p$-module.
Also, recall that the {\em nonfree locus} $\nf(\X)$ of a subcategory $\X$ of $\mod R$ is defined as the union of $\nf(M)$ where $M$ runs through the objects of $\X$. 
In this section, we introduce and study NE-loci $\NE(-)$, which extend nonfree loci $\nf(-)$ to objects and subcategories of $\d(R)$.
We begin with defining the NE-loci of objects of $\d(R)$.

\begin{dfn}
Let $X$ be an object of $\d(R)$.
We denote by $\NE(X)$ the set of prime ideals $\p$ of $R$ such that $X_\p\notin\E_{R_\p}$, and call it the {\em NE-locus} of $X$.
According to Proposition \ref{27}(6), this is equal to the set of prime ideals $\p$ of $R$ such that the $R_\p$-complex $X_\p$ has positive (possibly infinite) projective dimension.
Thus we may also call $\NE(X)$ the {\em positive projective dimension locus} of $X$.
Clearly, the equality $\NE(M)=\nf(M)$ holds for each finitely generated $R$-module.
Note that $\NE(X)$ is contained in $\supp X$, where the latter set is the {\em support} of $X$, which is defined by the equality $\supp X=\{\p\in\spec R\mid X_\p\ncong0\text{ in }\d(R_\p)\}$.
\end{dfn}

We state a basic fact on free resolutions and truncations of complexes, which is frequently used later. 

\begin{rem}\label{40}
Let $X\in\d(R)$ be a complex.
Put $t=\inf X$ and $s=\sup X$.
Then there exists a complex
$$
F=(\cdots\xrightarrow{\partial^{t-1}}F^t\xrightarrow{\partial^t}F^{t+1}\xrightarrow{\partial^{t+1}}\cdots\xrightarrow{\partial^{s-2}}F^{s-1}\xrightarrow{\partial^{s-1}}F^s\to0)
$$
of finitely generated free $R$-modules such that $X\cong F$ in $\d(R)$; see \cite[(A.3.2)]{C} for instance.
\begin{enumerate}[(1)]
\item
Let $C$ be the cokernel of the differential map $\partial_{t-1}$, and let $P=(0\to F^{t+1}\to\cdots\to F^s\to0)$ be a truncation of $F$.
Then $P$ is a perfect complex and one has an exact triangle $P\to X\to C[-t]\rightsquigarrow$ in $\d(R)$.
\item
Let $P=(0\to F^1\to\cdots\to F^s\to0)$ and $Y=(\cdots\to F^{-1}\to F^0\to0)$ be truncations of $F$.
Then $Y$ is an object of $\d(R)$ with $\sup Y\le0$.
There is an exact triangle $P\to X\to Y\rightsquigarrow$ in $\d(R)$, which induces an exact triangle $X\to Y\to P[1]\rightsquigarrow$ in $\d(R)$.
Proposition \ref{27}(6) implies that $P$ and $P[1]$ belong to $\E_R$. 
\end{enumerate}
\end{rem}

Here are several basic properties of the NE-loci of objects of the derived category $\d(R)$.

\begin{lem}\label{31}
\begin{enumerate}[\rm(1)]
\item
Let $X$ be an object of $\d(R)$.
Then, the set $\NE(X)$ is empty if and only if $X$ belongs to $\E_R$.
\item
For any objects $X_1,\dots,X_n$ of $\d(R)$ one has the equality $\NE(\bigoplus_{i=1}^nX_i)=\bigcup_{i=1}^n\NE(X_i)$.
\item
For every $X\in\d(R)$ there exists $Y\in\d(R)$ with $\sup Y\le0$, $\res_{\d(R)}X=\res_{\d(R)}Y$ and $\NE(X)=\NE(Y)$.
\item
For an exact triangle $X\to Y\to Z\rightsquigarrow$ one has $\NE(X)\subseteq\NE(Y)\cup\NE(Z)$ and $\NE(Y)\subseteq\NE(X)\cup\NE(Z)$.
\end{enumerate}
\end{lem}

\begin{proof}
(1) By \cite[Proposition 5.3.P]{AF} and Proposition \ref{27}(6), we get $\NE(X)=\emptyset$ if and only if $X_\p\in\E_{R_\p}$ for all $\p\in\spec R$, if and only if $\pd_{R_\p}X_\p\le0$ for all $\p\in\spec R$, if and only if $\pd_RX\le0$, if and only if $X\in\E_R$.

(2) Let $\p$ be a prime ideal of $R$.
Since $\E_{R_\p}$ is a resolving subcategory of $\d(R_\p)$, we have that $\bigoplus_{i=1}^n(X_i)_\p=(\bigoplus_{i=1}^nX_i)_\p\in\E_{R_\p}$ if and only if $(X_i)_\p\in\E_{R_\p}$ for all $1\le i\le n$.
The assertion follows from the contrapositive.

(3) According to Remark \ref{40}(2), there exists an exact triangle $X\to Y\to Z\rightsquigarrow$ in $\d(R)$ such that $\sup Y\le0$ and $Z\in\E_R$.
For each resolving subcategory $\X$ of $\d(R)$ we have $X\in\X$ if and only if $Y\in\X$.
In particular, it holds that $\res_{\d(R)}X=\res_{\d(R)}Y$.
For every $\p\in\spec R$ there exists an exact triangle $X_\p\to Y_\p\to Z_\p\rightsquigarrow$ in $\d(R_\p)$.
Since $Z_\p$ belongs to $\E_{R_\p}$, we again have $X_\p\in\Y$ if and only if $Y_\p\in\Y$ for each resolving subcategory $\Y$ of $\d(R_\p)$.
In particular, $X_\p\notin\E_{R_\p}$ if and only if $Y_\p\notin\E_{R_\p}$.
Hence the equality $\NE(X)=\NE(Y)$ holds.

(4) Fix $\p\in\spec R$.
By Proposition \ref{27}(3), if $\pd Y_\p\le0$ and $\pd Z_\p\le0$, then $\pd X_\p\le\sup\{\pd Y_\p,\pd Z_\p-1\}\le0$.
Also, if $\pd X_\p\le0$ and $\pd Z_\p\le0$, then $\pd Y_\p\le\sup\{\pd X_\p,\pd Z_\p\}\le0$.
The assertion follows.
\end{proof}

\begin{rem}
In view of Lemma \ref{31}(4), we may wonder if the inclusion $\NE(Z)\subseteq\NE(X)\cup\NE(Y)$ holds for every exact triangle $X\to Y\to Z\rightsquigarrow$ in $\d(R)$.
This is not true in general.
In fact, the exact triangle $R\xrightarrow{=}R\to0\rightsquigarrow$ induces an exact triangle $R\to 0\to R[1]\rightsquigarrow$.
Then $\NE(R[1])=\spec R$ because $\pd(R[1])_\p=\pd R_\p[1]=\pd R_\p+1=1$ for each $\p\in\spec R$ by Proposition \ref{27}(1), while $\NE(R)\cup\NE(0)$ is an empty set.
\end{rem}

We provide a generalization (or a derived category version) of \cite[Proposition 2.10 and Corollary 2.11(1)]{res}.

\begin{prop}\label{25}
For every object $X$ of $\d(R)$ there is an equality $\NE(X)=\supp_R\Hom_{\ld(R)}(X,X)$.
In particular,  the NE-loci of objects of $\d(R)$ are closed subsets of $\spec R$ in the Zariski topology.
\end{prop}

\begin{proof}
A prime ideal $\p$ of $R$ does not belong to the support of the $R$-module $\Hom_{\ld(R)}(X,X)$ if and only if $\Hom_{\ld(R)}(X,X)_\p=0$, if and only if $\Hom_{\ld(R_\p)}(X_\p,X_\p)=0$ by Lemma \ref{16}(1), if and only if $X_\p$ belongs to $\E_{R_\p}$ by Proposition \ref{23}(2), if and only if $\p$ is not in $\NE(X)$.
It follows that $\NE(X)=\supp\Hom_{\ld(R)}(X,X)$.
\end{proof}

Next we state the definition of the NE-locus of a subcategory of $\d(R)$.

\begin{dfn}
For a subcategory $\X$ of $\d(R)$, we set $\NE(\X)=\bigcup_{X\in\X}\NE(X)$ and call it the {\em NE-locus} of $\X$.
Since each $\NE(X)$ in the union is a Zariski-closed subset of $\spec R$ by Proposition \ref{25}, the subset $\NE(\X)$ of $\spec R$ is specialization-closed; this statement is a generalization of \cite[Corollary 2.11(2)]{res}.
\end{dfn}

The following proposition is regarded as a derived category version of \cite[Corollary 3.6]{res}.

\begin{prop}\label{61}
For every subcategory $\X$ of $\d(R)$ the equality $\NE(\res_{\d(R)}\X)=\NE(\X)$ holds.
\end{prop}

\begin{proof}
Since $\X$ is contained in $\res_{\d(R)}\X$, we see that $\NE(\X)$ is contained in $\NE(\res\X)$.
Let $\p$ be a prime ideal of $R$ with $\p\notin\NE(\X)$.
We have $\pd_{R_\p}X_\p\le0$ for every $X\in\X$, so that $\X$ is contained in the subcategory $\Y$ of $\d(R)$ consisting of complexes $Y$ such that $\pd_{R_\p}Y_\p\le0$.
Clearly, $\Y$ contains $R$.
By the projective dimension equality in Proposition \ref{27}(4), we see that $\Y$ is closed under direct summands. 
The first projective dimension inequality in Proposition \ref{27}(3) shows that $\Y$ is closed under extensions.
The projective dimension equality in Proposition \ref{27}(1) implies that $\Y$ is closed under negative shifts.
Thus, $\Y$ is a resolving subcategory of $\d(R)$ containing $\X$, so that it contains $\res\X$.
It follows that $\p\notin\NE(\res\X)$.
Thus, $\NE(\X)=\NE(\res\X)$.
\end{proof}

Recall that a finitely generated $R$-module $M$ is called a {\em maximal Cohen--Macaulay module} provided that it satisfies the inequality $\depth_{R_\p}M_\p\ge\dim R_\p$ for all prime ideals $\p$ of $R$.
Now, extending this, we introduce the notion of a maximal Cohen--Macaulay complex.
This plays an important role in the rest of this paper.

\begin{dfn}
\begin{enumerate}[(1)]
\item
We call an object $X$ of $\d(R)$ a {\em maximal Cohen--Macaulay complex} if $\depth_{R_\p}X_\p\ge\dim R_\p$ for all prime ideals $\p$.
By definition, a finitely generated $R$-module $M$ is maximal Cohen--Macaulay if and only if the complex $(0\to M\to0)$ concentrated in degree zero is maximal Cohen--Macaulay.
\item
We denote by $\c(R)$ the subcategory of $\d(R)$ consisting of all maximal Cohen--Macaulay $R$-complexes.
\end{enumerate}
\end{dfn}

Recall that $\sing R$ stands for the {\em singular locus} of $R$, that is to say, the set of prime ideals $\p$ of $R$ such that the local ring $R_\p$ is not regular.
In the proposition below, we state some properties of maximal Cohen--Macaulay complexes, whose module category version can be found in \cite[Example 2.9]{res}.

\begin{prop}\label{58}
\begin{enumerate}[\rm(1)]
\item
Let $X\in\d(R)$ be maximal Cohen--Macaulay.
Then $\NE(X)$ is contained in $\sing R$.
\item
The subcategory $\c(R)$ of $\d(R)$ is closed under direct summands, extensions and negative shifts.
If $R$ is a Cohen--Macaulay ring, then $\c(R)$ is a resolving subcategory of $\d(R)$ and vice versa.
\end{enumerate}
\end{prop}

\begin{proof}
(1) Let $\p$ be a prime ideal of $R$ with $\p\notin\sing R$.
Then $R_\p$ is a regular local ring, so that $\pd_{R_\p}X_\p<\infty$.
By Proposition \ref{27}(2), we have $\pd_{R_\p}X_\p=\depth R_\p-\depth X_\p=\dim R_\p-\depth X_\p\le0$.
Thus $\p\notin\NE(X)$.

(2) The depth equality in Proposition \ref{27}(4) shows $\c(R)$ is closed under direct summands.
Using the first inequality in Proposition \ref{27}(3), we observe that $\c(R)$ is closed under extensions.
It is seen from the depth equality in Proposition \ref{27}(1) that $\c(R)$ is closed under negative shifts.
The ring $R$ is Cohen--Macaulay if and only if $R$ belongs to $\c(R)$, if and only if $\c(R)$ is a resolving subcategory of $\d(R)$.
\end{proof}

Recall that a {\em thick} subcategory $\X$ of $\cm(R)$ is by definition a subcategory of $\cm(R)$ closed under direct summands and such that for every short exact sequence $0\to L\to M\to N\to0$ of maximal Cohen--Macaulay $R$-modules, if two of $L,M,N$ belong to $\X$, then so does the third.
Also, for each set $\Phi$ of prime ideals of $R$, the subcategory $\nf^{-1}(\Phi)$ is defined as the subcategory of $\mod R$ consisting of modules whose nonfree loci are contained in $\Phi$, and $\nf^{-1}_\cm(\Phi)$ is defined to be the intersection of $\nf^{-1}(\Phi)$ with $\cm(R)$.
These three subcategories play important roles in \cite{stcm,crs,thd}.
Now we introduce their derived category versions.

\begin{dfn}
\begin{enumerate}[(1)]
\item
We say that a subcategory $\X$ of $\c(R)$ is {\em thick} provided that $\X$ is closed under direct summands in the additive category $\c(R)$, and that for each exact triangle $A\to B\to C\rightsquigarrow$ in $\d(R)$ such that $A,B,C\in\c(R)$, if two of $A,B,C$ belong to $\X$, then so does the third.
(We should be careful not to confuse a thick subcategory of $\c(R)$ with a thick subcategory of $\d(R)$ in the sense of Definition \ref{100}.)
\item
For a subset $\Phi$ of $\spec R$, we denote by $\NE^{-1}(\Phi)$ the subcategory of $\d(R)$ consisting of complexes whose NE-loci are contained in $\Phi$.
We define the subcategory $\NE^{-1}_\c(\Phi)$ of $\c(R)$ by $\NE^{-1}_\c(\Phi)=\NE^{-1}(\Phi)\cap\c(R)$.
\end{enumerate}
\end{dfn}

The following proposition includes a derived category version of \cite[Propositions 1.15(3), 4.2 and Theorem 4.10(3)]{stcm}.
Compare this proposition with Theorem \ref{76} stated later.

\begin{prop}\label{60}
The following statements hold true.
\begin{enumerate}[\rm(1)]
\item
Every thick subcategory of $\d(R)$ contained in $\c(R)$ is a thick subcategory of $\c(R)$.
Every thick subcategory of $\c(R)$ containing $R$ is a resolving subcategory of $\d(R)$ contained in $\c(R)$.
\item
For $\Phi\subseteq\spec R$ the subcategory $\NE^{-1}(\Phi)$ of $\d(R)$ is resolving.
If $R$ is Cohen--Macaulay, then $\NE^{-1}_\c(\Phi)$ is a thick subcategory of $\c(R)$ containing $R$, and a resolving subcategory of $\d(R)$ contained in $\c(R)$.
\end{enumerate}
\end{prop}

\begin{proof}
(1) First of all, note that since $\c(R)$ is closed under direct summands as a subcategory of $\d(R)$ by Proposition \ref{58}(2), being closed under direct summands as a subcategory of $\c(R)$ implies being closed under direct summands as a subcategory of $\d(R)$.
The first assertion now follows.
To show the second, let $\X$ be a thick subcategory of $\c(R)$ containing $R$.
Then $\X$ is closed under direct summands as a subcategory of $\d(R)$.
Let $A\to B\to C\rightsquigarrow$ be an exact triangle in $\d(R)$ with $C\in\X$.
Then $C$ is in $\c(R)$, and we observe from Propositions \ref{58}(2) and \ref{42}(1) that $A\in\c(R)$ if and only if $B\in\c(R)$.
Since $\X$ is a thick subcategory of $\c(R)$, it is easy to verify that $A\in\X$ if and only if $B\in\X$.
Thus, $\X$ is a resolving subcategory of $\d(R)$.

(2) Since $\NE(R)=\emptyset\subseteq\Phi$, we have $R\in\NE^{-1}(\Phi)$.
Using Lemma \ref{31}(2), we see that $\NE^{-1}(\Phi)$ is closed under direct summands.
Let $X\to Y\to Z\rightsquigarrow$ be an exact triangle in $\d(R)$ with $Z\in\NE^{-1}(\Phi)$.
Then $\Phi$ contains $\NE(Z)$.
It follows from Lemma \ref{31}(4) that $\Phi$ contains $\NE(X)$ if and only if $\Phi$ contains $\NE(Y)$.
This means that $X\in\NE^{-1}(\Phi)$ if and only if $Y\in\NE^{-1}(\Phi)$.
Thus, $\NE^{-1}(\Phi)$ is a resolving subcategory of $\d(R)$.

Let $R$ be Cohen--Macaulay.
The first assertion of (2) shows $\NE^{-1}(\Phi)$ is a resolving subcategory of $\d(R)$, and so is $\c(R)$ by Proposition \ref{58}(2), whence so is $\NE^{-1}(\Phi)\cap\c(R)=\NE_\c^{-1}(\Phi)$.
Let $A\to B\to C\rightsquigarrow$ be an exact triangle in $\d(R)$ with $A,B\in\NE_\c^{-1}(\Phi)$ and $C\in\c(R)$.
Assume $C\notin\NE_\c^{-1}(\Phi)$.
Then we find $\p\in\NE(C)$ such that $\p\notin\Phi$.
Hence $\p$ does not belong to $\NE(A)$ or $\NE(B)$, which means that $\pd A_\p\le0$ and $\pd B_\p\le0$.
The exact triangle $A_\p\to B_\p\to C_\p\rightsquigarrow$ in $\d(R_\p)$ and Proposition \ref{27}(3) show $\pd C_\p\le1<\infty$, and we get
$$
\pd C_\p=\depth R_\p-\depth C_\p=\dim R_\p-\depth C_\p\le0,
$$
where the first equality comes from Proposition \ref{27}(2), the second equality holds since the local ring $R_\p$ is Cohen--Macaulay, and the inequality holds as $C$ is maximal Cohen--Macaulay.
This is a contradiction because $\p\in\NE(C)$.
It follows that $C\in\NE_\c^{-1}(\Phi)$.
Thus $\NE_\c^{-1}(\Phi)$ is a thick subcategory of $\c(R)$ (containing $R$).
\end{proof}


\begin{rem}
The converse of the first assertion of Proposition \ref{60}(1) does not necessarily hold true.
In fact, $\c(R)$ is itself a thick subcategory of $\c(R)$, but it is not necessarily a thick subcategory of $\d(R)$.
For example, let $R$ be a Cohen--Macaulay local ring of positive Krull dimension.
Then there exists a non-zerodivisor $x$ in the maximal ideal of $R$, which gives rise to an exact sequence $0\to R\xrightarrow{x}R\to R/(x)\to0$ in $\mod R$, which induces an exact triangle $R\xrightarrow{x}R\to R/(x)\rightsquigarrow$ in $\d(R)$.
We have $R\in\c(R)$ but $R/(x)\notin\c(R)$.
Therefore, $\c(R)$ is not a thick subcategory of $\d(R)$.
This argument also shows the module category version:
a thick subcategory of $\cm(R)$ is not necessarily a thick subcategory of $\mod R$ contained in $\cm(R)$.
\end{rem}

\section{Classification of resolving subcategories of $\k(R)$}

The purpose of this section is to give a complete classification of the resolving subcategories of $\k(R)$.
We start by defining, for each object of the derived category $\d(R)$, another object by tensoring a Koszul complex and taking a shift.
This is used in the proof of the first main result of this section.

\begin{dfn}
Let $X\in\d(R)$.
For $x\in R$, set $X(x)=\K(x)\lten_RX[-1]\in\d(R)$.
For a sequence $\xx=x_1,\dots,x_n$ in $R$, we inductively define $X(\xx)\in\d(R)$ by $X(x_1,\dots,x_i)=(X(x_1\dots,x_{i-1}))(x_i)$ for $1\le i\le n$.
\end{dfn}

We make a list of basic properties of the object $X(x)$ for $X\in\d(R)$ and $x\in R$.

\begin{lem}\label{24}
Let $X$ be an object of $\d(R)$, and let $x$ be an element of $R$.
\begin{enumerate}[\rm(1)]
\item
If $x$ is a unit of $R$, then there is an isomorphism $X(x)\cong0$ in $\d(R)$.
\item
There exists an exact triangle $X(x)\to X\xrightarrow{x}X\rightsquigarrow$ in $\d(R)$.
In particular, one has the containment $X(x)\in\res_{\d(R)}X$ and the isomorphisms $X(x)\cong\rhom_R(\K(x),X)\cong\Hom_R(\K(x),X)$ in $\d(R)$.
\item
Let $R$ be a local ring with maximal ideal $\m$ and residue field $k$.
Let $x\in\m$.
Then one has the equalities $\depth_RX(x)=\depth_RX$ and $\pd_RX(x)=\pd_RX$.
In particular, $X\in\E_R$ if and only if $X(x)\in\E_R$.
\item
There is an equality $\NE(X(x))=\NE(X)\cap\V(x)$ of subsets of $\spec R$.
\end{enumerate}
\end{lem}

\begin{proof}
(1) If $x$ is a unit of $R$, then $\K(x)\cong0$ in $\d(R)$, and hence $X(x)=\K(x)\lten_RX[-1]\cong0$ in $\d(R)$.

(2) There exists an exact triangle $e:R\xrightarrow{x}R\to\K(x)\rightsquigarrow$ in $\d(R)$.
Applying the functor $-\lten_RX$ to $e$ gives rise to an exact triangle $X\xrightarrow{x}X\to\K(x)\lten_RX\rightsquigarrow$ in $\d(R)$, which induces an exact triangle $a: X(x)\to X\xrightarrow{x}X\rightsquigarrow$ in $\d(R)$.
Hence $X(x)$ belongs to $\res_{\d(R)}X$.
Applying the functor $\rhom_R(-,X)$ to $e$ yields an exact triangle $b: \rhom_R(\K(x),X)\to X\xrightarrow{x}X\rightsquigarrow$.
It follows from $a$ and $b$ that $X(x)\cong\rhom_R(\K(x),X)$.

(3) As $x$ belongs to $\m$, we have $\K(x)\lten_Rk\cong(0\to k\xrightarrow{x}k\to0)=(0\to k\xrightarrow{0}k\to0)\cong k\oplus k[1]$.
Hence
$$
\begin{array}{l}
\rhom_R(k,X(x))\cong\rhom_R(k,\rhom_R(\K(x),X))\cong\rhom_R(\K(x)\lten_Rk,X)\\
\phantom{\rhom_R(k,X(x))}
\cong\rhom_R(k\oplus k[1],X)\cong\rhom_R(k,X)\oplus\rhom_R(k,X)[-1]\text{ by (2), and}\\
X(x)\lten_Rk\cong\K(x)\lten_RX[-1]\lten_Rk\cong(\K(x)\lten_Rk)\lten_RX[-1]\\
\phantom{X(x)\lten_Rk\cong\K(x)\lten_RX[-1]\lten_Rk}\cong(k\oplus k[1])\lten_RX[-1]\cong(X\lten_Rk)\oplus(X\lten_Rk)[-1].
\end{array}
$$
As $\inf Y[-1]=\inf Y+1$ for any complex $Y$, we get $\depth X(x)=\inf\rhom_R(k,X(x))=\inf\rhom_R(k,X)=\depth X$ and $\pd X(x)=-\inf(X(x)\lten_Rk)=-\inf(X\lten_Rk)=\pd X$, where we use Proposition \ref{27}(2) for the latter statement.
By virtue of Proposition \ref{27}(6), we have $X\in\E_R$ if and only if $X(x)\in\E_R$.

(4) To show $(\supseteq)$, let $\p\in\NE(X)\cap\V(x)$.
Then $\pd_{R_\p}X_\p>0$ and $\frac{x}{1}\in\p R_\p$.
By (3) we have $\pd_{R_\p}X_\p(\frac{x}{1})=\pd_{R_\p}X_\p>0$, whence $\p\in\NE(X(x))$.
To show $(\subseteq)$, let $\p\in\NE(X(x))$.
Then $\pd_{R_\p}X_\p(\frac{x}{1})>0$.
In particular, $X_\p(\frac{x}{1})\ncong0$ in $\d(R)$.
Hence $x\in\p$ by (1).
By (3) we get $0<\pd_{R_\p}X_\p(\frac{x}{1})=\pd_{R_\p}X_\p$, so that $\p\in\NE(X)$.
\end{proof}

The theorem below is the first main result of this section.
The first and last assertions of the theorem are regarded as derived category versions of \cite[Theorem 4.3]{res} and \cite[Lemma 4.6]{radius} respectively, which concern the nonfree locus and the resolving closure of an object in the module category $\mod R$.

\begin{thm}\label{26}
Let $X$ be an object of $\d(R)$.
Let $W$ be a closed subset of $\spec R$ contained in $\NE(X)$.
Then:
\begin{enumerate}[\rm(1)]
\item
There exists an object $Y\in\res_{\d(R)}X$ such that $W=\NE(Y)$.
\item
If $R$ is local and $W$ is nonempty, then $Y$ can be chosen so that $\pd Y=\pd X$ and $\depth Y=\depth X$. 
\item
If $W$ is irreducible, then $Y$ can be chosen so that $\pd Y_\p=\pd X_\p$ and $\depth Y_\p=\depth X_\p$ for all $\p\in W$.
\end{enumerate}
\end{thm}

\begin{proof}
When $W$ is empty, we can take $Y:=R$.
Assume that $W\ne\emptyset$.
Then there exist prime ideals $\p_1,\dots,\p_n$ of $R$ such that $n>0$ and $W=\V(\p_1)\cup\cdots\cup\V(\p_n)$.
Each $\V(\p_i)$ is contained in $\NE(X)$.
If we find an object $Y_i\in\res_{\d(R)}X$ such that $\V(\p_i)=\NE(Y_i)$, then $Y:=Y_1\oplus\cdots\oplus Y_n$ belongs to $\res X$ and satisfies $W=\NE(Y)$ by Lemma \ref{31}(2).
If $R$ is local, $\pd_RY_i=\pd_RX$ and $\depth_RY_i=\depth_RX$ for all $1\le i\le n$, then $\pd_RY=\sup_{1\le i\le n}\{\pd_RY_i\}=\pd_RX$ and $\depth_RY=\inf_{1\le i\le n}\{\depth_RY_i\}=\depth_RX$ by Proposition \ref{27}(4).
Thus, it suffices to show that in the case where $W=\V(\p)$ for some prime ideal $\p$ of $R$ there exists $Y\in\res_{\d(R)}X$ such that $W=\NE(Y)$, $\pd_{R_P}Y_P=\pd_{R_P}X_P$ and $\depth_{R_P}Y_P=\depth_{R_P}X_P$ for all $P\in W$.

The set $\NE(X)$ contains $\V(\p)$.
If $\NE(X)=\V(\p)$, then we are done by letting $Y=X$.
Suppose that $\NE(X)$ strictly contains $\V(\p)$ and choose an element $\q\in\NE(X)\setminus\V(\p)$.
Then $\q$ does not contain $\p$, and we can choose an element $x\in\p\setminus\q$.
Using Lemma \ref{24}(4), we get $\p\in\NE(X)\cap\V(x)=\NE(X(x))$.
As $\NE(X(x))$ is Zariski-closed by Proposition \ref{25}, it contains $\V(\p)$.
It follows that $\V(\p)\subseteq\NE(X(x))=\NE(X)\cap\V(\p)\subseteq\NE(X)$ and the fact that $\q\in\NE(X)\setminus\V(\p)$ says $\NE(X)\cap\V(\p)\ne\NE(X)$.
We conclude $\V(\p)\subseteq\NE(X(x))\subsetneq\NE(X)$.

By Lemma \ref{24}(2)(3), we have that $X(x)\in\res X$, and that $\pd_{R_P}X(x)_P=\pd_{R_P}X_P(\frac{x}{1})=\pd_{R_P}X_P$ and $\depth_{R_P}X(x)_P=\depth_{R_P}X_P(\frac{x}{1})=\depth_{R_P}X_P$ for all $P\in\V(\p)$ since $x\in\p\subseteq P$.
If $\NE(X(x))$ is equal to $\V(\p)$, we are done by letting $Y=X(x)$.
If $\NE(X(x))$ strictly contains $\V(\p)$, we apply the above argument to find $y\in\p$ with $\V(\p)\subseteq\NE(X(x,y))\subsetneq\NE(X(x))$.
Iterating this procedure, we get an ascending chain
$$
\V(\p)\subseteq\cdots\subsetneq\NE(x,y,z,w)\subsetneq\NE(x,y,z)\subsetneq\NE(X(x,y))\subsetneq\NE(X(x))\subsetneq\NE(X)
$$
of subsets of $\spec R$ with $x,y,z,w,\ldots\in\p$.
However, we can do this only finitely many times, since each NE-locus appearing in the above chain is Zariski-closed, and the topological space $\spec R$ is noetherian.

We thus obtain a sequence $\xx=x_1,\dots,x_n$ in $\p$ with $\NE(X(\xx))=\V(\p)$ and $X(\xx)\in\res X$, $\pd_{R_P}X(\xx)_P=\pd_{R_P}X_P$ and $\depth_{R_P}X(\xx)_P=\depth_{R_P}X_P$ for all $P\in\V(\p)$.
The theorem follows by letting $Y=X(\xx)$.
\end{proof}

From the above theorem we can deduce the following corollary.
Thanks to this result, for each object $X$ in a fixed resolving subcategory of $\d(R)$, one may often assume that $X$ belongs to $\d_0(R)$.

\begin{cor}\label{32}
Let $R$ be a local ring.
For every object $X\in\d(R)$ there exists an object $Y\in\res_{\d(R)}X\cap\d_0(R)$ such that $\pd_RY=\pd_RX$ and $\depth_RY=\depth_RX$.
\end{cor}

\begin{proof}
When $X$ belongs to $\E_R$, we put $Y:=X$ and are done.
Let $X$ be outside of $\E_R$.
Then the maximal ideal $\m$ of $R$ belongs to $\NE(X)$ by Lemma \ref{31}(1) and Proposition \ref{25}, and hence $\V(\m)$ is contained in $\NE(X)$.
Applying Theorem \ref{26} to $\V(\m)$, we find an object $Y\in\res_{\d(R)}X$ such that $\NE(Y)=\V(\m)=\{\m\}$, $\pd Y=\pd X$ and $\depth Y=\depth X$.
The equality $\NE(Y)=\{\m\}$ implies that $Y$ belongs $\d_0(R)$.
\end{proof}

Next, we prove the following lemma, which is thought of as a derived category version of \cite[Lemma 3.2 and Proposition 3.3]{crspd}.
For a partially ordered set $P$ we denote by $\min P$ the set of mimimal elements of $P$.

\begin{lem}\label{29}
Let $\X$ be a subcategory of $\d(R)$.
\begin{enumerate}[\rm(1)]
\item
Let $S$ be a multiplicatively closed subset of $R$.
Suppose that $\X$ is a resolving subcategory of $\d(R)$.
Then $\add_{\d(R_S)}\X_S$ is a resolving subcategory of $\d(R_S)$.
Hence, the equality $\add_{\d(R_S)}\X_S=\res_{\d(R_S)}\X_S$ holds.
\item
Suppose that $\X$ contains $R$ and is closed under finite direct sums.
Let $Z$ be a nonempty finite subset of $\spec R$.
Let $C\in\d(R)$ be such that $C_\p\in\add_{\d(R_\p)}\X_\p$ for all $\p\in Z$.
Then there exist exact triangles
$$
K\to X\to C\to K[1],\qquad
L\to K\oplus C\to X\to L[1]
$$
in $\d(R)$ such that $X\in\X$, that $\NE(L)\subseteq\NE(C)$, that $\supp L\cap Z=\emptyset$, and that $\pd_{R_\p}L_\p\le\pd_{R_\p}C_\p$ and $\depth_{R_\p}L_\p\ge\depth_{R_\p}C_\p$ for all prime ideals $\p$ of $R$.
\item
Assume that $\X$ is a resolving subcategory of $\d(R)$. 
The following are equivalent for each $C\in\d(R)$.
\begin{enumerate}[\rm(a)]
\item
The object $C$ belongs to $\X$.
\item
The localization $C_\p$ belongs to $\add_{\d(R_\p)}\X_\p$ for all prime ideals $\p$ of $R$.
\item
The localization $C_\m$ belongs to $\add_{\d(R_\m)}\X_\m$ for all maximal ideals $\m$ of $R$.
\end{enumerate}
\end{enumerate}
\end{lem}

\begin{proof}
(1) By definition, the additive closure $\add\X_S$ is closed under direct summands.
As $R$ is in $\X$, we have $R_S\in\X_S\subseteq\add\X_S$.
Let $A\in\add\X_S$.
Then $A\oplus B$ is isomorphic to $X_S$ for some $B\in\d(R_S)$ and $X\in\X$, whence $A[-1]\oplus B[-1]$ is isomorphic to $(X[-1])_S$.
As $\X$ is closed under negative shifts, we see that $A[-1]$ is in $\add\X_S$.
Therefore, $\add\X_S$ is closed under negative shifts.
Let $L\to M\to N\rightsquigarrow$ be an exact triangle in $\d(R_S)$ with $L,N\in\add\X_S$.
Then $L\oplus L'\cong X_S$ and $N\oplus N'\cong Y_S$ for some $L',N'\in\d(R_S)$ and $X,Y\in\X$.
Taking the direct sum with the exact triangles $L'\to L'\to0\rightsquigarrow$ and $0\to N'\to N'\rightsquigarrow$, we observe that there exists an exact triangle $L'\oplus M\oplus N'\to Y_S\xrightarrow{f}X[1]_S\rightsquigarrow$ in $\d(R_S)$.
Write $f=\frac{g}{s}$, where $g:Y\to X[1]$ is a morphism in $\d(R)$ and $s$ is an element of $S$; see \cite[Lemma 5.2(b)]{AF}.
There is an exact triangle $X\to Z\to Y\xrightarrow{g}X[1]$ in $\d(R)$.
Since $\X$ is closed under extensions, $Z$ is in $\X$.
Also, we see that $Z_S$ is isomorphic to $L'\oplus M\oplus N'$ in $\d(R_S)$.
It follows that $M$ belongs to $\add\X_S$, which shows that $\add\X_S$ is closed under extensions.

(2) Write $Z=\{\p_1,\dots,\p_n\}$.
Fix $1\le i\le n$.
There exists $X_i\in\X$ such that $C_{\p_i}$ is a direct summand of $(X_i)_{\p_i}$ in $\d(R_{\p_i})$.
We have a split epimorphism $f_i:(X_i)_{\p_i}\to C_{\p_i}$ in $\d(R_{\p_i})$, so that there is a morphism $\alpha_i:C_{\p_i}\to(X_i)_{\p_i}$ in $\d(R_{\p_i})$ with $f_i\alpha_i=\id_{C_{\p_i}}$.
Choose a morphism $g_i:X_i\to C$ in $\d(R)$ and an element $s_i\in R\setminus\p_i$ such that $\frac{g_i}{s_i}=f_i$.
Set $X=X_1\oplus\cdots\oplus X_n\in\X$ and consider the morphism $g=(g_1,\dots,g_n):X\to C$ in $\d(R)$.
Then $\frac{g}{1}=(\frac{g_1}{1},\dots,\frac{g_n}{1}):X_{\p_i}\to C_{\p_i}$ is a split epimorphism in $\d(R_{\p_i})$ for each $i$, since letting $\beta_i:C_{\p_i}\to X_{\p_i}$ be the transpose of $(0,\dots,0,\frac{1}{s_i}\alpha_i,0,\dots,0)$, we have $\frac{g}{1}\beta_i=\id_{C_{\p_i}}$.
There is an exact triangle $K\to X\xrightarrow{g}C\xrightarrow{h}K[1]$ in $\d(R)$.
For any integer $1\le i\le n$ it holds that $h_{\p_i}=\frac{h}{1}=0$ in $\d(R_{\p_i})$, which means that the annihilator $\ann_Rh$ of $h\in\Hom_{\d(R)}(C,K[1])$ is not contained in $\p_i$.
By prime avoidance, we find an element $s\in\ann_Rh$ such that $s\notin\p$ for all $\p\in Z$.
The octahedral axiom gives rise to a commutative diagram
$$
\xymatrix@R-1pc@C+2pc{
C\ar[r]^-s\ar@{=}[d]& C\ar[d]^-h\ar[r]& \k(s)\lten_RC\ar[r]\ar[d]&C[1]\ar@{=}[d]\\
C\ar[r]^-{hs}_-0\ar[d]& K[1]\ar@{=}[d]\ar[r]& K[1]\oplus C[1]\ar[d]\ar[r]& C[1]\ar[d]\\
C\ar[r]^-h\ar[d]& K[1]\ar[r]\ar[d]& X[1]\ar@{=}[d]\ar[r]^-{g[1]}& C[1]\ar[d]\\
\k(s)\lten_RC\ar[r]& K[1]\oplus C[1]\ar[r]& X[1]\ar[r]& (\k(s)\lten_RC)[1]
}
$$
in $\d(R)$ whose rows are exact triangles.
The bottom row in the diagram induces an exact triangle $L\to K\oplus C\to X\to L[1]$ in $\d(R)$, where we put $L:=C(s)=\k(s)\lten_RC[-1]$.

Let $\p\in Z$.
Then the element $\frac{s}{1}$ of $R_\p$ is a unit, and hence it holds in $\d(R_\p)$ that $\k(s)_\p=\k(\frac{s}{1},R_\p)\cong0$.
Therefore, $L_\p=\k(s)_\p\lten_{R_\p}C_\p[-1]\cong0$ in $\d(R_\p)$.
It follows that the intersection $\supp L\cap Z$ is an empty set.

Fix a prime ideal $\p$ of $R$.
Note that we have $L_\p=C(s)_\p=C_\p(\frac{s}{1})$.
If $s$ is in $\p$, then $\pd C_\p(\frac{s}{1})=\pd C_\p$ and $\depth C_\p(\frac{s}{1})=\depth C_\p$ by Lemma \ref{24}(3).
If $s$ is not in $\p$, then $C_\p(\frac{s}{1})=0$ by Lemma \ref{24}(1), so that $\pd C_\p(\frac{s}{1})=-\infty$ and $\depth C_\p(\frac{s}{1})=\infty$.
Therefore, there are inequalities $\pd L_\p\le\pd C_\p$ and $\depth L_\p\ge\depth C_\p$.
If $\p$ is not in $\NE(C)$, then $C_\p$ is in $\E_{R_\p}$, and so is $L_\p$.
Hence $\NE(L)$ is contained in $\NE(C)$.

(3) Localization shows the implications (a) $\Rightarrow$ (c) $\Rightarrow$ (b).
Assume that (b) holds and $C\notin\X$.
Then the set
$$
A=\{\NE(Y)\mid\text{$Y\in\d(R)$, $Y\notin\X$ and $Y_\p\in\add_{\d(R_\p)}\X_\p$ for all prime ideals $\p$ of $R$}\}
$$
is nonempty.
Since $\spec R$ is a noetherian space and each $\NE(Y)$ is Zariski-closed by Proposition \ref{25}, the set $A$ contains a minimal element $\NE(B)$ with $B\in\d(R)$, $B\notin\X$ and $B_\p\in\add\X_\p$ for every prime ideal $\p$ of $R$.
If $\NE(B)$ is an empty set, then we have $B\in\E_R\subseteq\X$ by Lemma \ref{31}(1), which gives a contradiction.
Thus $\NE(B)$ is a nonempty Zariski-closed set, which implies that $\min\NE(B)$ is nonempty and finite.
It follows from (2) that there exist exact triangles $K\to X\to B\rightsquigarrow$ and $L\to K\oplus B\to X\rightsquigarrow$ in $\d(R)$ such that $X\in\X$, $\NE(L)\subseteq\NE(B)$ and $\NE(L)\cap\min\NE(B)=\emptyset$.
In particular, $\NE(L)$ is strictly contained in $\NE(B)$.

We claim that $L_\p$ is in $\add\X_\p$ for every $\p\in\spec R$.
In fact, there is an exact triangle $K_\p\to X_\p\to B_\p\rightsquigarrow$.
It follows from (1) that $\add\X_\p$ is a resolving subcategory of $\d(R_\p)$.
Since $X_\p$ and $B_\p$ belong to $\add\X_\p$, so does $K_\p$. 
There is an exact triangle $L_\p\to K_\p\oplus B_\p\to X_\p\rightsquigarrow$.
As $K_\p\oplus B_\p$ and $X_\p$ are in $\add\X_\p$, so is $L_\p$.

Now the minimality of $\NE(B)$ forces $L$ to be in $\X$.
The exact triangle $L\to K\oplus B\to X\rightsquigarrow$ implies that $B$ belongs to $\X$, which contradicts the choice of $B$.
We thus conclude that the object $C$ belongs to $\X$.
\end{proof}

\begin{rem}
In Lemma \ref{29}(2), the object $L$ is taken in such a way that $\supp L\cap Z=\emptyset$.
Comparing this with the module category version of Lemma \ref{29}(2) given in  \cite[Lemma 3.2]{crspd}, we see that the expected condition satisfied by the object $L$ in Lemma \ref{29}(2) is the weaker condition that $\NE(L)\cap Z=\emptyset$.
It is an advantage the derived category possesses against the module category that one can get $L$ so that $\supp L\cap Z=\emptyset$. 
By the way, only for the purpose of proving our main results, it does suffice to have the equality $\NE(L)\cap Z=\emptyset$.
\end{rem}

Here we define assignments between subcategories of $\d(R)$ and maps from $\spec R$ to $\Z\cup\{\pm\infty\}$, and state a couple of properties which are used in the next main result of this section and its proof.

\begin{dfn}\label{53}
For a subcategory $\X$ of $\d(R)$, we define the map $\Phi(\X):\spec R\to\Z\cup\{\pm\infty\}$ by $\Phi(\X)(\p)=\sup_{X\in\X}\{\pd_{R_\p}X_\p\}$ for $\p\in\spec R$.
For a map $f:\spec R\to\Z\cup\{\pm\infty\}$, we denote by $\Psi(f)$ the subcategory of $\d(R)$ consisting of objects $X$ with $\pd_{R_\p}X_\p\le f(\p)$ for all $\p\in\spec R$.
We equip the sets $\spec R$ and $\Z\cup\{\pm\infty\}$ with the partial orders given by the inclusion relation $(\subseteq)$ and the inequality relation $(\le)$, respectively.
\end{dfn}

\begin{lem}\label{28}
\begin{enumerate}[\rm(1)]
\item
Let $\X$ be a subcategory of $\d(R)$ which contains $R$.
Then $\Phi(\X)$ defines an order-preserving map from $\spec R$ to $\N\cup\{\infty\}$.
\item
Let $f:\spec R\to\N\cup\{\infty\}$ be a map.
Then $\Psi(f)$ is a resolving subcategory of $\d(R)$.
\end{enumerate}
\end{lem}

\begin{proof}
(1) Fix a prime ideal $\p$ of $R$.
Since $R$ belongs to $\X$, it holds that $\sup_{X\in\X}\{\pd X_\p\}\ge\pd R_\p=0$.
Thus, $\Phi(\X)(\p)$ is an element of $\N\cup\{\infty\}$.
If $\q$ is a prime ideal of $R$ that contains $\p$, then we have $\pd_{R_\p}X_\p=\pd_{(R_\q)_{\p R_\q}}(X_\q)_{\p R_\q}\le\pd_{R_\q}X_\q$ for each $X\in\d(R)$ (see \cite[Proposition 5.1(P)]{AF}), whence $\Phi(\X)(\p)\le\Phi(\X)(\q)$.

(2)  We have $\pd R_\p=0\le f(\p)$ for every prime ideal $\p$ of $R$, which shows that $\Psi(f)$ contains $R$.
If $X$ is an object of $\Psi(f)$ and $Y$ is a direct summand of $X$ in $\d(R)$, then $\pd Y_\p\le\pd X_\p\le f(\p)$ for all $\p\in\spec R$ by Proposition \ref{27}(4), which shows that $Y$ belongs to $\Psi(f)$.
Let $X\to Y\to Z\rightsquigarrow$ be an exact triangle in $\d(R)$ with $Z\in\Psi(f)$. 
Then for each $\p\in\spec R$ there is an exact triangle $X_\p\to Y_\p\to Z_\p\rightsquigarrow$ in $\d(R_\p)$ and $\pd Z_\p\le f(\p)$.
It is seen from Proposition \ref{27}(3) that $\pd X_\p\le f(\p)$ if and only if $\pd Y_\p\le f(\p)$.
Therefore, $X$ is in $\Psi(f)$ if and only if $Y$ is in $\Psi(f)$.
We now conclude that $\Psi(f)$ is a resolving subcategory of $\d(R)$.
\end{proof}

Now we have reached the stage to state and prove the second main result of this section, which provides a complete classification of the resolving subcategories of $\k(R)$.

\begin{thm}\label{2}
The assignments $\X\mapsto\Phi(\X)$ and $f\mapsto\Psi(f)\cap\k(R)$ give mutually inverse bijections between the resolving subcategories of $\k(R)$, and the order-preserving maps from $\spec R$ to $\N\cup\{\infty\}$.
\end{thm}

\begin{proof}
Fix a resolving subcategory $\X$ of $\k(R)$ and an order-preserving map $f:\spec R\to\N\cup\{\infty\}$.
Lemma \ref{28} implies that $\Phi(\X):\spec R\to\N\cup\{\infty\}$ is an order-preserving map and $\Psi(f)$ is a resolving subcategory of $\d(R)$.
Hence $\Psi(f)\cap\k(R)$ is a resolving subcategory of $\k(R)$.
Fix a prime ideal $\p$ of $R$.
It is clear that
$$
\Phi(\Psi(f)\cap\k(R))(\p)=\sup\{\pd_{R_\p}P_\p\mid P\in\k(R)\text{ and }\pd_{R_\q}P_\q\le f(\q)\text{ for all }\q\in\spec R\}\le f(\p).
$$
Let $\xx=x_1,\dots,x_s$ be a system of generators of $\p$, and let $\q$ be a prime ideal of $R$.
First, we consider the case where $f(\p)=:n<\infty$.
Set $P=\K(\xx)[n-s]\in\k(R)$.
We have $\pd_{R_\p}P_\p=n=f(\p)$ by Proposition \ref{27}(1)(7).
If $\p$ is contained in $\q$, then $\pd_{R_\q}P_\q=\pd_{R_\q}\K(\xx,R_\q)+(n-s)\le s+(n-s)=n=f(\p)\le f(\q)$.
If $\p$ is not contained in $\q$, then $\pd_{R_\q}P_\q=-\infty\le f(\q)$ by Proposition \ref{27}(7).
Thus $\Phi(\Psi(f)\cap\k(R))(\p)=f(\p)$.
Next, we consider the case where $f(\p)=\infty$.
Then for any integer $n$ we set $P=\K(\xx)[n-s]$ to have $\pd_{R_\p}P_\p=n$.
If $\p$ is contained in $\q$, then $\infty=f(\p)\le f(\q)$, so that $\pd_{R_\q}P_\q\le\infty=f(\q)$.
If $\p$ is not contained in $\q$, then $\pd_{R_\q}P_\q=-\infty\le f(\q)$.
We get $\Phi(\Psi(f)\cap\k(R))(\p)=\infty=f(\p)$.
It now follows that $\Phi(\Psi(f)\cap\k(R))=f$.

It remains to prove that $\Psi(\Phi(\X))\cap\k(R)=\X$.
Note that there are an equality and an inclusion
$$
\textstyle\Psi(\Phi(\X))\cap\k(R)=\{P\in\k(R)\mid\text{$\pd_{R_\p}P_\p\le\sup_{X\in\X}\{\pd_{R_\p}X_\p\}$ for all $\p\in\spec R$}\}\supseteq\X.
$$
Let $P\in\Psi(\Phi(\X))\cap\k(R)$.
All we need to do is show that $P$ is in $\X$.
Fix a maximal ideal $\m$ of $R$ and a prime ideal $\p$ of $R$ contained in $\m$.
Then $\add\X_\m$ is a resolving subcategory of $\k(R_\m)$ by Lemma \ref{29}(1).
We have
\begin{equation}\label{33}
\begin{array}{l}
\pd_{(R_\m)_{\p R_\m}}(P_\m)_{\p R_\m}
=\pd_{R_\p}P_\p
\le\textstyle\sup_{X\in\X}\{\pd_{R_\p}X_\p\}
=\textstyle\sup_{X\in\X}\{\pd_{(R_\m)_{\p R_\m}}(X_\m)_{\p R_\m}\}\\
\phantom{\pd_{(R_\m)_{\p R_\m}}(P_\m)_{\p R_\m}=\pd_{R_\p}P_\p
\le\textstyle\sup_{X\in\X}\{\pd_{R_\p}X_\p\}}\le\textstyle\sup_{Y\in\add\X_\m}\{\pd_{(R_\m)_{\p R_\m}}Y_{\p R_\m}\}.
\end{array}
\end{equation}
By Lemma \ref{29}(3), it suffices to show $P_\m\in\add\X_\m$.
Thus we may assume that $(R,\m)$ is local.
Then the dimension $n:=\dim\NE(P)$ of the Zariski-closed set $\NE(P)$ is finite.
We prove by induction on $n$ that $P\in\X$.

When $n\le0$, the set $\NE(P)$ is contained in $\{\m\}$, which means that $P$ belongs to $\k_0(R)$.
By the choice of $P$, we have $\pd_RP\le\sup_{X\in\X}\{\pd_RX\}$, which implies that there is an object $X\in\X$ such that $\pd_RP\le\pd_RX$.
Corollary \ref{32} gives rise to an object $X'\in\d_0(R)\cap\res_{\d(R)}X$ with $\pd_RX'=\pd_RX$.
Hence $X'\in\X\cap\k_0(R)$ and $\pd_RP\le\pd_RX'$.
Replacing $X$ with $X'$, we may assume that $X\in\k_0(R)$.
Setting $t=\pd_RX$, we have the inequality $\pd_RP\le t$, so that $P\in\k_0^t(R)=\res_{\k(R)}X\subseteq\X$, where the equality holds by Theorem \ref{13}.

Now we deal with the case where $n>0$.
Then $\min\NE(P)$ is nonempty and finite.
Write $\min\NE(P)=\{\p_1,\dots,\p_r\}$ and fix $1\le i\le r$.
We have $\NE(P_{\p_i})=\{\p_i R_{\p_i}\}$, whence $\dim\NE(P_{\p_i})=0$.
The subcategory $\add\X_{\p_i}$ of $\k(R_{\p_i})$ is resolving by Lemma \ref{29}(1).
Similarly as in \eqref{33}, the inequality $\pd_{(R_{\p_i})_{\q R_{\p_i}}}(X_{\p_i})_{\q R_{\p_i}}\le\sup_{Y\in\add\X_{\p_i}}\{\pd_{(R_{\p_i})_{\q R_{\p_i}}}Y_{\q R_{\p_i}}\}$ holds for every prime ideal $\q$ of $R$ contained in $\p_i$.
The induction basis implies that $P_{\p_i}$ belongs to $\add\X_{\p_i}$.
Lemma \ref{29}(2) yields exact triangles $K\to Z\to P\rightsquigarrow$ and $L\to K\oplus P\to Z\rightsquigarrow$ in $\d(R)$ such that $Z\in\X$, $\NE(L)\subseteq\NE(P)$, $\NE(L)\cap\min\NE(P)=\emptyset$, and $\pd_{R_\p}L_\p\le\pd_{R_\p}P_\p$ for every prime ideal $\p$ of $R$.
Since $Z$ and $P$ are in $\k(R)$, so is $K$, and so is $L$.
The inequality $\dim\NE(L)<\dim\NE(P)=n$ holds, while $\pd_{R_\p}L_\p\le\pd_{R_\p}P_\p\le\sup_{X\in\X}\{\pd_{R_\p}X_\p\}$ for all $\p\in\spec R$.
Hence we can apply the induction hypothesis to $L$ to deduce that $L$ is in $\X$.
The exact triangle $L\to K\oplus P\to Z\rightsquigarrow$ shows that $P$ is in $\X$.
\end{proof}

\section{Restricting the classification of resolving subcategories of $\k(R)$}

In this section, we compare our results obtained so far with the results of Dao and Takahashi concerning resolving subcategories of $\mod R$ contained in $\fpd R$.
For this purpose, we begin with recalling some notation.

\begin{dfn}
Let $R$ be a local ring.
We set $\mod_0(R)=\mod R\cap\d_0(R)$.
Note that $\mod_0(R)$ consists of the finitely generated $R$-modules which are locally free on the punctured spectrum of $R$.
We also put
$$
\begin{array}{l}
\fpd_0(R)=\fpd R\cap\mod_0(R)=\k_0(R)\cap\mod R=\{M\in\mod_0(R)\mid\pd_RM<\infty\},\\
\fpd_0^n(R)=\k_0^n(R)\cap\mod R=\k^n(R)\cap\mod_0(R)=\{M\in\fpd_0(R)\mid\pd_RM\le n\}\ \ \text{for }n\in\Z.
\end{array}
$$
\end{dfn}

The following theorem is \cite[Theorem 2.1]{crspd}.
Denote by $\syz$ and $\tr$ the syzygy and transpose functors.

\begin{thm}[Dao--Takahashi]\label{44}
Let $R$ be a local ring of depth $t$ and with residue field $k$.
Then one has
\begin{equation}\label{46}
\proj R=\fpd_0^0(R)\subsetneq\fpd_0^1(R)\subsetneq\cdots\subsetneq\fpd_0^t(R)=\fpd_0^{t+1}(R)=\cdots=\fpd_0(R)
\end{equation}
such that $\tr\syz^{n-1}k\in\fpd_0^n(R)\setminus\fpd_0^{n-1}(R)$ for each $t\ge n\ge1$.
Moreover, all the resolving subcategories of $\mod R$ contained in $\fpd_0(R)$ appear in the above chain.
\end{thm}

\begin{rem}
\begin{enumerate}[(1)]
\item
Theorem \ref{13} is viewed as a derived category version of Theorem \ref{44}.
\item
A remarkable difference between Theorems \ref{13} and \ref{44} is that the former says that there exist only finitely many resolving subcategories of $\mod R$ contained in $\fpd_0(R)$, while the latter says that there exist infinitely (but countably) many resolving subcategories of $\k(R)$ contained in $\k_0(R)$.
\item
Although both have similar configurations, the proof of Theorem \ref{13} is completely different from that of Theorem \ref{44}.
Indeed, the latter requires much more complicated arguments on modules which involve syzygies and transposes; the whole of \cite[\S2]{crspd} is devoted to giving a proof of Theorem \ref{44}. 
\item
The restriction of \eqref{45} to $\mod R$ coincides with \eqref{46}.
Indeed, Proposition \ref{27}(6) says that $\E_R\cap\mod R=\proj R$, while by definition we have $\k_0^n(R)\cap\mod R=\fpd_0^n(R)$ for each integer $n$.
The Auslander--Buchsbaum formula \cite[Theorem 1.3.3]{BH} shows $\fpd_0^n(R)=\fpd_0(R)$ for all integers $n\ge t$.
\end{enumerate}
\end{rem}

In the proof of Theorem \ref{44}, the resolving closure $\res_{\mod R}k$ in $\mod R$ of the residue field $k$ of $R$ does play a crucial role; it coincides with $\mod_0(R)$.
Here we consider a derived category version of this fact.

\begin{prop}\label{30}
Let $R$ be a local ring with residue field $k$.
Let $X\in\d_0(R)$ and put $h=\depth X$.
One then has $X\in\res_{\d(R)}(k[-h])$.
In particular, it holds that $\d_0(R)=\res_{\d(R)}\{k[i]\mid i\in\Z\}$.
\end{prop}

\begin{proof}
Take a system of parameters $\xx=x_1,\dots,x_d$ of $R$.
Set $Y=\K(\xx)\lten_RX$.
It follows from Lemma \ref{17}(3) that $X\in\res_{\d(R)}(Y[-d])$.
Taking soft truncations of the complex $Y$ implies that $Y$ is in the extension closure $\ext_{\d(R)}\{\h^iY[-i]\mid\inf Y\le i\le\sup Y\}$.
Localization at nonmaximal prime ideals shows that each $\h^iY$ has finite length as an $R$-module (see Proposition \ref{27}(7)), so that it is in $\ext_{\d(R)}k$.
We have $Y\in\ext_{\d(R)}\{k[-i]\mid\inf Y\le i\le\sup Y\}\subseteq\res_{\d(R)}(k[-\inf Y])$, where the inclusion comes from the fact that every resolving subcategory is closed under negative shifts.
Using \cite[Theorem I]{FI}, we get $\inf Y=h-d$, which implies $Y\in\res_{\d(R)}(k[d-h])$.
Therefore, the object $X$ belongs to $\res_{\d(R)}(k[-h])$ by Proposition \ref{52}(2a).
\end{proof}

We recall the definition of a grade-consistent function which has been introduced in \cite{crspd}.

\begin{dfn}
A {\em grade-consistent function} on $\spec R$ is by definition an order-preserving map $f:\spec R\to\N$ such that the inequality $f(\p)\le\grade\p$ holds for all prime ideals $\p$ of $R$.
\end{dfn}

The grade condition in the definition of a grade-consistent function can be changed to a depth condition.

\begin{lem}\label{34}
Let $f:\spec R\to\N\cup\{\infty\}$ be an order-preserving map.
Then $f$ is a grade-consistent function on $\spec R$ if and only if $f(\p)\le\depth R_\p$ for all prime ideals $\p$ of $R$.
\end{lem}

\begin{proof}
Fix $\p\in\spec R$.
The equality $\grade\p=\inf\{\depth R_\q\mid\q\in\V(\p)\}$ holds by \cite[Proposition 1.2.10(a)]{BH}.
In particular, one has $\grade\p\le\depth R_\p$, which shows the `only if' part of the lemma.
To show the `if' part, suppose $f(\q)\le\depth R_\q$ for all $\q\in\spec R$.
Then the image of $f$ is contained in $\N$.
If $\q\in\V(\p)$, then $\p\subseteq\q$ and $f(\p)\le f(\q)\le\depth R_\q$.
This shows $f(\p)\le\inf\{\depth R_\q\mid\q\in\V(\p)\}=\grade\p$.
Thus, we are done.
\end{proof}

Applying the above lemma, we can show the following result on the assignments used in Theorem \ref{2}.

\begin{prop}\label{35}
Let $\Phi$ and $\Psi$ be the ones introduced in Definition \ref{53}.
\begin{enumerate}[\rm(1)]
\item
Let $\X$ be a subcategory of $\mod R$ containing $R$.
Then $\Phi(\X)$ is a grade-consistent function on $\spec R$.
\item
Let $f:\spec R\to\N$ be a map.
Then the equality $\Psi(f)\cap\mod R=(\Psi(f)\cap\k(R))\cap\mod R$ holds, and it is a resolving subcategory of $\mod R$ contained in $\fpd R$.
\end{enumerate}
\end{prop}

\begin{proof}
(1) Lemma \ref{28}(1) implies that $\Phi(\X)$ is an order-preserving map from $\spec R$ to $\N\cup\{\infty\}$.
For each $\p\in\spec R$ we have $\Phi(\X)(\p)=\sup_{X\in\X}\{\pd X_\p\}\le\depth R_\p$ since the Auslander--Buchsbaum formula implies $\pd X_\p=\depth R_\p-\depth X_\p\le\depth R_\p$.
Lemma \ref{34} shows $\Phi(\X)$ is a grade-consistent function on $\spec R$.

(2) According to Lemma \ref{28}(2), the subcategory $\Psi(f)$ of $\d(R)$ is resolving.
It follows from Proposition \ref{43}(4) that $\Psi(f)\cap\mod R$ is a resolving subcategory of $\mod R$, and $(\Psi(f)\cap\k(R))\cap\mod R=\Psi(f)\cap\fpd R=(\Psi(f)\cap\mod R)\cap\fpd R$ is a resolving subcategory of $\mod R$ contained in $\fpd R$.
Let $M\in\Psi(f)\cap\mod R$.
Then for every prime ideal $\p$ of $R$ one has $\pd_{R_\p}M_\p\le f(\p)\in\N$, which particularly says that $\pd_{R_\p}M_\p<\infty$.
By \cite[Lemma 4.5]{BM}, we get $\pd_RM<\infty$.
Therefore, $\Psi(f)\cap\mod R$ coincides with $(\Psi(f)\cap\mod R)\cap\fpd R$.
\end{proof}

The following theorem is shown in \cite[Theorem 1.2]{crspd}, which is one of the main results of \cite{crspd}.

\begin{thm}[Dao--Takahashi]\label{36}
By the assignments $\X\mapsto\Phi(\X)$ and $f\mapsto\Psi(f)\cap\mod R$, the resolving subcategories of $\mod R$ contained in $\fpd R$ bijectively correspond to the grade-consistent functions on $\spec R$.
\end{thm}

\begin{rem}\label{56}
Proposition \ref{35} says that Theorem \ref{36} is regarded as the restriction of Theorem \ref{2} to $\mod R$.
\end{rem}

\section{Classification of certain preaisles of $\k(R)$}

In this section, we consider classifying certain preaisles of the triangulated category $\k(R)$.
We begin with recalling the definitions of preaisles and several related notions.

\begin{dfn}\cite[\S1.1]{AJS}\cite[\S1.3]{BBD}\cite[\S1.1]{KV}
Let $\T$ be a triangulated category.
A {\em preaisle} (resp. {\em precoaisle}) of $\T$ is by definition a subcategory of $\T$ closed under extensions and positive (resp. negative) shifts.
A preaisle (resp. precoaisle) $\X$ of $\T$ is called an {\em aisle} (resp. a {\em coaisle}) if the inclusion functor $\X\hookrightarrow\T$ has a right (resp. left) adjoint.
For an aisle $\X$ and a coaisle $\Y$ of $\T$, the pair $(\X,\Y[1])$ is called a {\em $t$-structure} of $\T$.
\end{dfn}

Denote by $(-)^\ast$ the $R$-dual functor $\Hom_R(-,R)$.
The assignment $P\mapsto P^\ast$ gives a duality of $\k(R)$, which sends an exact triangle $A\xrightarrow{f}B\xrightarrow{g}C\xrightarrow{h}A[1]$ to the exact triangle $C^\ast\xrightarrow{g^\ast}B^\ast\xrightarrow{f^\ast}A^\ast\xrightarrow{h^\ast[1]}C^\ast[1]$.
For a subcategory $\X$ of $\k(R)$, we denote by $\X^\ast$ the subcategory of $\k(R)$ consisting of complexes of the form $X^\ast$ with $X\in\X$.
The following lemma is straightforward from the definitions of preaisles and precoaisles.

\begin{lem}\label{3}
The assignment $\X\mapsto\X^\ast$ produces a one-to-one correspondence
$$
\left\{\begin{matrix}
\text{preaisles of $\k(R)$ containing $R$}\\
\text{and closed under direct summands}
\end{matrix}\right\}\cong\left\{\begin{matrix}
\text{precoaisles of $\k(R)$ containing $R$}\\
\text{and closed under direct summands}
\end{matrix}\right\}.
$$
\end{lem}

\begin{rem}
In \cite{DS} a preaisle closed under direct summands is called a {\em thick preaisle}.
\end{rem}

Let us recall the definition of a certain fundamental filtration of subsets of $\spec R$.

\begin{dfn}
A {\em filtration by supports} or {\em sp-filtration} of $\spec R$ is by definition an order-reversing map $\phi:\Z\to2^{\spec R}$ such that for each $i\in\Z$ the subset $\phi(i)$ of $\spec R$ is specialization-closed.
\end{dfn}

Here we need to introduce some notation.
$$
\begin{array}{l}
\text{$\F(\phi)(\p)=\sup\{j\in\Z\mid\p\in\phi(j)\}+1$ for a map $\phi:\Z\to2^{\spec R}$ and $\p\in\spec R$},\\
\text{$\P(f)(i)=\{\p\in\spec R\mid f(\p)>i\}$ for a map $f:\spec R\to\Z\cup\{\pm\infty\}$ and $i\in\Z$},\\
\text{$\G(f)=\{X\in\k(R)\mid\text{$\pd X^\ast_\p\le f(\p)$ for all $\p\in\spec R$}\}$ for a map $f:\spec R\to\Z\cup\{\pm\infty\}$},\\
\text{$\Q(\X)(\p)=\sup\{\pd X^\ast_\p\mid X\in\X\}$ for a subcategory $\X$ of $\k(R)$ and $\p\in\spec R$}.
\end{array}
$$

Now we can state and prove the main result of this section, which classifies certain preaisles of $\k(R)$.

\begin{thm}\label{49}
There are one-to-one correspondences
$$
\xymatrix@C-0.1pc{
{\left\{\begin{matrix}
\text{preaisles of $\k(R)$ containing $R$}\\
\text{and closed under direct summands}
\end{matrix}\right\}}
\ar@<.7mm>[r]^-\Q&
{\left\{\begin{matrix}
\text{order-preserving maps}\\
\text{$\spec R\to\N\cup\{\infty\}$}
\end{matrix}\right\}}
\ar@<.7mm>[r]^-\P\ar@<.7mm>[l]^-\G&
{\left\{\begin{matrix}
\text{sp-filtrations $\phi$ of $\spec R$}\\
\text{with $\phi(-1)=\spec R$}
\end{matrix}\right\}.}
\ar@<.7mm>[l]^-\F}
$$
\end{thm}

\begin{proof}
Note that the resolving subcategories of $\k(R)$ are precisely the precoaisles of $\k(R)$ that contain $R$ and are closed under direct summands.
Thus, it immediately follows from Lemma \ref{3} and Theorem \ref{2} that the maps $(\Q,\G)$ appearing in the assertion are mutually inverse bijections.
If $\phi$ is an sp-filtration of $\spec R$ with $\phi(-1)=\spec R$, then $\F(\phi)(\p)=\sup\{i\in\Z\mid\p\in\phi(i)\}+1\ge(-1)+1=0$ for each prime ideal $\p$ of $R$, and hence $\F(\phi)$ is regarded as a map from $\spec R$ to $\N\cup\{\infty\}$.
If $f:\spec R\to\N\cup\{\infty\}$ is an order-preserving map, then $\P(f)(-1)=\{\p\in\spec R\mid f(\p)\ge0\}=\spec R$.
Thus, it follows from \cite[Proposition 4.3]{ft} that the maps $(\P,\F)$ appearing in the assertion are mutually inverse bijections.
Now the proof is completed.
\end{proof}

\begin{rem}
Theorem \ref{49} can also be proved by using the techniques of the unbounded derived category $\dd(\Mod R)$ of all $R$-modules.
We present the argument by an anonymous reader and Tsutomu Nakamura.
Let $A$ be the set of preaisles of $\k(R)$ closed under direct summands, $B$ the set of aisles of compactly generated $t$-structures of $\dd(\Mod R)$, and $C$ the set of sp-filtrations of $\spec R$.
Then, combining \cite[Proposition 1.9(ii)]{Na} with \cite[Theorem A.7]{KN} implies that the map $f:A\to B$ is bijective, which sends each $\X\in A$ to the aisle of $\dd(\Mod R)$ generated by $\X$.
In \cite[Theorem 3.11]{AJS} it is proved that the map $g:B\to C$ is bijective, which sends each $\Y\in B$ to the map $\phi:\Z\to2^{\spec R}$ given by $\phi(n)=\{\p\in\spec R\mid(R/\p)[-n]\in\Y\}$ for each $n\in\Z$.
\end{rem}

To compare the above theorem with classification of aisles of $\d(R)$, we need to recall some notions.

\begin{dfn}
\begin{enumerate}[(1)]
\item
A map $f:\spec R\to\Z\cup\{\pm\infty\}$ is called a {\em $t$-function} on $\spec R$ if for each inclusion $\p\subseteq\q$ in $\spec R$ there are inequalities $f(\p)\le f(\q)\le f(\p)+\height\q/\p$.
\item
An sp-filtration $\phi$ is said to satisfy the {\em weak Cousin condition} provided that for all integers $i$ and for all saturated inclusions $\p\subsetneq\q$ in $\spec R$, if $\q$ belongs to $\phi(i)$, then $\p$ belongs to $\phi(i-1)$.
\item
We say that $R$ is {\em CM-excellent} if $R$ is universally catenary, the formal fibers of the localizations of $R$ are Cohen--Macaulay, and the Cohen--Macaulay locus of each finitely generated $R$-algebra is Zariski-open.
\end{enumerate}
\end{dfn}

Takahashi \cite[Theorem 5.5]{ft} proved the following, which yields a complete classification of the $t$-structures of $\d(R)$ when $R$ is a CM-excellent ring of finite Krull dimension.
We use the following notation:
$$
\begin{array}{l}
\text{$\H(f)=\{X\in\d(R)\mid\text{$\h^{\ge f(\p)}(X_\p)=0$ for all $\p\in\spec R$}\}$ for a map $f:\spec R\to\Z\cup\{\pm\infty\}$},\\
\text{$\R(\X)(\p)=\sup\{i\in\Z\mid(R/\p)[-i]\in\X\}+1$ for a subcategory $\X$ of $\d(R)$ and $\p\in\spec R$}.
\end{array}
$$

\begin{thm}[Takahashi]\label{48}
When $R$ is CM-excellent and $\dim R<\infty$, there are one-to-one correspondences
$$
\xymatrix{
{\left\{\begin{matrix}
\text{aisles}\\
\text{of $\d(R)$}
\end{matrix}\right\}}
\ar@<.7mm>[r]^-\R&
{\left\{\begin{matrix}
\text{$t$-functions}\\
\text{on $\spec R$}
\end{matrix}\right\}}
\ar@<.7mm>[r]^-\P\ar@<.7mm>[l]^-\H&
{\left\{\begin{matrix}
\text{sp-filtrations of $\spec R$}\\
\text{satisfying the weak Cousin condition}
\end{matrix}\right\}.}
\ar@<.7mm>[l]^-\F}
$$
\end{thm}

In the proposition below, we record a relationship between Theorem \ref{49} and the restriction of Theorem \ref{48} to $\k(R)$.
Note that the intersection of the set of order-preserving maps from $\spec R$ to $\N\cup\{\infty\}$ and the set of $t$-functions on $\spec R$ consists of the $t$-functions whose images are contained in $\N\cup\{\infty\}$.

\begin{prop}\label{51}
Let $R$ be a CM-excellent ring of finite Krull dimension.
Let $f$ be a $t$-function on $\spec R$ whose image is contained in $\N\cup\{\infty\}$.
Then there is an equality $\G(f)[1]=\H(f)\cap\k(R)$. 
\end{prop}

\begin{proof}
We claim that if $R$ is local, $X\in\k(R)$ and $n\in\Z$, then $\pd X^\ast\le n$ if and only if $\h^{>n}(X)=0$.
In fact, letting $k$ be the residue field of $R$, we have that $\pd X^\ast=\sup\rhom_R(X^\ast,k)$, that $\rhom_R(X^\ast,k)\cong X\lten_Rk$ and that $\sup(X\lten_Rk)=\sup X$ by \cite[(A.5.7.3), (A.4.24) and (A.6.3.2)]{C}, respectively.

Let $X$ be an object of $\k(R)$.
Using the above claim, we see that $X\in\G(f)[1]$ if and only if $X[-1]\in\G(f)$, if and only if $\pd(X[-1])^\ast_\p\le f(\p)$ for all $\p\in\spec R$, if and only if $\pd X^\ast_\p\le f(\p)-1$ for all $\p\in\spec R$, if and only if $\h^{\ge f(\p)}(X_\p)=0$, if and only if $X\in\H(f)$.
It follows that $\G(f)[1]=\H(f)\cap\k(R)$.
\end{proof}

\begin{ques}\label{87}
Let $R$ be a CM-excellent ring of finite Krull dimension.
Is there any relationship between Theorem \ref{49} and the restriction of Theorem \ref{48} to $\k(R)$, other than the one shown in Proposition \ref{51}?
\end{ques}

\begin{rem}
In view of what we have stated so far, it is quite natural to ask if the aisles of $\k(R)$ can be classified.
If $R$ is regular, then $\k(R)$ coincides with $\d(R)$, and Theorem \ref{48} gives an answer.
In case $R$ is singular, it is known that $\k(R)$ possesses only trivial aisles under mild assumptions:
Smith \cite[Theorems 1.2 and 1.3]{S} proved that if $R$ has finite Krull dimension, then $\k(R)$ has no bounded $t$-structure, and if moreover $R$ is irreducible, then $0$ and $\k(R)$ are the only aisles of $\k(R)$.
The former statement has recently been extended to schemes by Neeman \cite[Theorem 0.1]{N3}, which resolves a conjecture of Antieau, Gepner and Heller \cite{AGH}.
\end{rem}

\section{Separating the resolving subcategories of $\d(R)$}

In this section, for a complete intersection $R$ we separate the resolving subcategories of $\d(R)$ into the resolving subcategories contained in $\k(R)$ and the resolving subcategories contained in $\c(R)$.
For this, we need several preparations.
We start by recalling the definition of the cosyzygies of a finitely generated module.

\begin{dfn}
Let $M$ be a finitely generated $R$-module and $n\ge1$ an integer.
We denote by $\syz_R^{-n}M$ the {\em $n$th cosyzygy} of $M$.
This is defined inductively as follows.
Let $\syz_R^{-1}M$ be the cokernel of a homomorphism $f:M\to P$ such that $P$ is a finitely generated projective $R$-module and the map $\Hom_R(f,R):\Hom_R(P,R)\to\Hom_R(M,R)$ is surjective.
For $n\ge2$ we set $\syz_R^{-n}M=\syz_R^{-1}(\syz_R^{-(n-1)}M)$.
The $n$th cosyzygy of $M$ is uniquely determined by $M$ and $n$ up to projective summands.
For details, see \cite[Sections 2 and 7]{catgp} for instance.
\end{dfn}

Next we recall the definition of a certain numerical invariant for complexes.

\begin{dfn}
The {\em (large) restricted flat dimension} $\rfd_RX$ of an $R$-complex $X\in\d(R)$ is defined by
$$
\textstyle\rfd_RX=\sup_{\p\in\spec R}\{\depth R_\p-\depth X_\p\}.
$$
One has inequalities $-\inf X\le\rfd_RX<\infty$; see \cite[Theorem 1.1]{AIL} and \cite[Proposition (2.2) and Theorem (2.4)]{CFF}.
Also, note that if $R$ is Cohen--Macaulay, then $X$ is maximal Cohen--Macaulay if and only if $\rfd_RX\le0$.
\end{dfn}

For a complex $X\in\d(R)$, we denote by $\gdim_RX$ the {\em Gorenstein dimension} ({\em G-dimension} for short) of $X$.
Recall that a {\em totally reflexive module} is defined to be a finitely generated module of G-dimension at most zero.
For the details of G-dimension and totally reflexive modules, we refer the reader to \cite{C}.
In the following lemma, we make a list of properties of G-dimension we need to use later.
Assertions (2) and (3) of the lemma correspond to assertions (1) and (3) of Proposition \ref{27} concerning projective dimension.

\begin{lem}\label{66}
\begin{enumerate}[\rm(1)]
\item
Let $Y,Z\in\d(R)$.
If $\pd_RY<\infty$ and $\gdim_RZ\le0$, then $\Ext_R^i(Z,Y)=0$ for all $i>\sup Y$.
\item
For every $X\in\d(R)$ and every $n\in\Z$ the equality $\gdim_R(X[n])=\gdim_RX+n$ holds.
\item
Let $X\to Y\to Z\rightsquigarrow$ be an exact triangle in $\d(R)$.
One then has $\gdim_RX\le\sup\{\gdim_RY,\gdim_RZ-1\}$, $\gdim_RY\le\sup\{\gdim_RX,\gdim_RZ\}$ and $\gdim_RZ\le\sup\{\gdim_RX+1,\gdim_RY\}$.
\item
An object $X\in\d(R)$ is isomorphic to a totally reflexive module if and only if $\gdim_RX\le0$ and $\sup X\le0$.
\item
Suppose that the ring $R$ is Gorenstein.
Then every complex $X$ of $\d(R)$ satisfies $\gdim_RX=\rfd_RX<\infty$.
In particular, $X$ is a maximal Cohen--Macaulay $R$-complex if and only if one has $\gdim_RX\le0$.
\end{enumerate}
\end{lem}

\begin{proof}
In what follows, \cite[(2.3.8)]{C} is our fundamental tool.
Also, we set $(-)^\star=\rhom_R(-,R)$, and for each complex $C\in\d(R)$ such that $C^\star\in\d(R)$, let $f_C:C\to C^{\star\star}$ stand for the natural morphism.

(1) If $R$ is local, then $\sup\rhom(Z,Y)\le\gdim Z+\sup Y\le\sup Y$ by \cite[(2.4.1)]{C}, and the assertion is deduced.
Suppose $R$ is nonlocal, and fix $\p\in\spec R$.
Then $\pd_{R_\p}Y_\p<\infty$, $Z_\p\in\d(R_\p)$, and $\gdim_{R_\p}Z_\p\le0$ by \cite[(2.3.11)]{C}.
The assertion in the local case shows $\Ext_{R_\p}^i(Z_\p,Y_\p)=0$ for all $i>\sup Y_\p$.
As $\sup Y_\p\le\sup Y$, we have $\Ext_R^i(Z,Y)_\p=\Ext_{R_\p}^i(Z_\p,Y_\p)=0$ for all $i>\sup Y$.
Therefore, $\Ext_R^i(Z,Y)=0$ for all $i>\sup Y$.

(2) The assertion is straightforward (from the definition of G-dimension or \cite[(2.3.8)]{C}).

(3) We have only to verify $\gdim Y\le\sup\{\gdim X,\gdim Z\}$, because once it is done, applying it to the exact triangles $Y\to Z\to X[1]\rightsquigarrow$ and $Z[-1]\to X\to Y\rightsquigarrow$ and using (2) will give the inequalities $\gdim Z\le\sup\{\gdim Y,\gdim X+1\}$ and $\gdim X\le\sup\{\gdim Z-1,\gdim Y\}$.
The inequality is obvious if either $\gdim X$ or $\gdim Z$ is infinite.
We may assume $\gdim X$ and $\gdim Z$ are both finite.
Then $X^\star,Z^\star$ belong to $\d(R)$, and $f_X,f_Z$ are isomorphisms.
The induced exact triangle $Z^\star\to Y^\star\to X^\star\rightsquigarrow$ (in the derived category of left-bounded $R$-complexes) shows that $Y^\star$ is in 
$\d(R)$.
There is a commutative diagram
$$
\xymatrix@R-1pc@C5pc{
X\ar[r]\ar[d]^-{f_X}&Y\ar[r]\ar[d]^-{f_Y}&Z\ar[r]\ar[d]^-{f_Z}&X[1]\ar[d]^-{f_X[1]}\\
X^{\star\star}\ar[r]&Y^{\star\star}\ar[r]&Z^{\star\star}\ar[r]&X^{\star\star}[1]
}
$$
of exact triangles in $\d(R)$.
Since $f_X$ and $f_Z$ are isomorphisms, so is $f_Y$.
We conclude that $\gdim Y$ is finite, and get $\gdim Y=\sup\rhom(Y,R)=\sup Y^\star\le\sup\{\sup X^\star,\sup Z^\star\}=\sup\{\gdim X,\gdim Z\}$.

(4) The ``only if'' part is obvious.
By \cite[(2.3.3)]{C} there is an inequality $\gdim X\ge-\inf X$.
Suppose that $\gdim X\le0$ and $\sup X\le0$.
We then have $\sup X\le0\le-\gdim X\le\inf X$, which implies $\sup X=\inf X=0$ or $X\cong0$ in $\d(R)$.
Hence, $X$ is isomorphic in $\d(R)$ to a totally reflexive module.
Thus the ``if'' part follows.

(5) The second assertion follows from the first and the fact that for a Cohen--Macaulay ring $R$ one has $X\in\c(R)$ if and only if $\rfd_RX\le0$.
To show the first assertion, put $r=\rfd_RX$.
Fix $\p\in\spec R$.
As $R_\p$ is a Gorenstein local ring, $\gdim_{R_\p}X_\p=\depth R_\p-\depth X_\p\le r$; see \cite[(2.3.13) and (2.3.14)]{C}.
Hence $\Ext_R^i(X,R)_\p=\Ext_{R_\p}^i(X_\p,R_\p)=0$ for all $i>r$.
Therefore, $\Ext_R^i(X,R)=0$ for all $i>r$.
We see that $(f_X)_\p=f_{X_\p}$ is an isomorphism for all $\p\in\spec R$, so that $f_X$ is an isomorphism (see \cite[(A.4.5) and (A.8.4.1)]{C}).
Hence $\gdim_RX=\sup X^\star\le r$.
If $\gdim_RX\le r-1$, then $\depth R_\p-\depth X_\p=\gdim_{R_\p}X_\p\le r-1$ for all $\p\in\spec R$ by \cite[(2.3.11) and (2.3.13)]{C}, and $\rfd_RX\le r-1$, a contradiction.
Thus $\gdim_RX=r$.
\end{proof}

In the lemma below we study the structure of a resolving subcategory of a certain form.
The idea of the proof comes from the proof of \cite[Lemma 7.2]{crspd}.

\begin{lem}\label{67}
Let $\Y$ and $\ZZ$ be resolving subcategories of $\d(R)$.
Let $\X$ be the subcategory of $\d(R)$ consisting of comlexes $X$ which fits into an exact triangle $Z\to X\oplus X'\to Y\rightsquigarrow$ in $\d(R)$ with $Z\in\ZZ$ and $Y\in\Y$.
\begin{enumerate}[\rm(1)]
\item
There is an inclusion $\X\subseteq\res_{\d(R)}(\Y\cup\ZZ)$ of subcategories of $\d(R)$.
\item
The subcategory $\X$ coincides with the subcategory of $\d(R)$ consisting of objects $X$ which fits into an exact triangle $Z\to X\oplus X'\to Y\rightsquigarrow$ in $\d(R)$ with $Z\in\ZZ$, $Y\in\Y$ and $\sup Y\le0$.
\item
Suppose that $\pd_RY<\infty$ for each $Y\in\Y$ and $\gdim_RZ\le0$ for each $Z\in\ZZ$.
Then $\X=\res_{\d(R)}(\Y\cup\ZZ)$.
\end{enumerate}
\end{lem}

\begin{proof}
(1) Let $Z\to X\oplus X'\to Y\rightsquigarrow$ be an exact triangle in $\d(R)$ such that $Z\in\ZZ$ and $Y\in\Y$.
As $\res(\Y\cup\ZZ)$ contains $\ZZ$ and $\Y$, the objects $Z,Y$ are in $\res(\Y\cup\ZZ)$.
The triangle implies that $X$ belongs to $\res(\Y\cup\ZZ)$.

(2) Let $Z\to X\oplus X'\to Y\rightsquigarrow$ be an exact triangle with $Z\in\ZZ$ and $Y\in\Y$.
Remark \ref{40}(2) gives an exact triangle $P\to Y\to Y'\rightsquigarrow$ with $P\in\E_R$ and $\sup Y'\le0$.
The octahedral axiom yields a commutative diagram
$$
\xymatrix@R-1.5pc@C5pc{
X\oplus X'\ar[r]\ar@{=}[d]& Y\ar[r]\ar[d]& Z[1]\ar[r]\ar[d]& (X\oplus X')[1]\ar@{=}[d]\\
X\oplus X'\ar[r]\ar[d]& Y'\ar[r]\ar@{=}[d]& Z'[1]\ar[r]\ar[d]& (X\oplus X')[1]\ar[d]\\
Y\ar[r]\ar[d]& Y'\ar[r]\ar[d]& P[1]\ar[r]\ar@{=}[d]& Y[1]\ar[d]\\
Z[1]\ar[r]& Z'[1]\ar[r]& P[1]\ar[r]& Z[2]
}
$$
of exact triangles, and the bottom row induces an exact triangle $Z\to Z'\to P\rightsquigarrow$.
As $Z$ and $P$ are in $\ZZ$, so is $Z'$.
An exact triangle $Z'\to X\oplus X'\to Y'\rightsquigarrow$ is induced from the second row.
Now the assertion follows.

(3) The inclusion $(\subseteq)$ is shown in (1).
We prove the opposite inclusion $(\supseteq)$.
For $Y\in\Y$ and $Z\in\ZZ$ there are exact triangles $0\to Y\oplus0\to Y\rightsquigarrow$ and $Z\to Z\oplus0\to0\rightsquigarrow$.
This shows that $\Y\cup\ZZ\subseteq\X$.
We will be done once we prove that $\X$ is a resolving subcategory of $\d(R)$.
As $R$ belongs to $\Y$ (and $\ZZ$), it belongs to $\X$.

Let $X$ be an object of $\X$, and let $W$ be a direct summand of $X$ in $\d(R)$.
Then there is an exact triangle $Z\to X\oplus X'\to Y\rightsquigarrow$ in $\d(R)$ such that $Z\in\ZZ$ and $Y\in\Y$, and also $X=W\oplus V$ for some $V\in\d(R)$.
Setting $W'=V\oplus X'$, we have an exact triangle $Z\to W\oplus W'\to Y\rightsquigarrow$.
Hence $\X$ is closed under direct summands.

Every exact triangle $Z\to X\oplus X'\to Y\rightsquigarrow$ with $Z\in\ZZ$ and $Y\in\Y$ induces an exact triangle $Z[-1]\to X[-1]\oplus X'[-1]\to Y[-1]\rightsquigarrow$, and we have $Z[-1]\in\ZZ$ and $Y[-1]\in\Y$.
Thus $\X$ is closed under negative shifts.

It remains to prove that $\X$ is closed under extensions.
Let $L\xrightarrow{f}M\xrightarrow{g}N\rightsquigarrow$ be an exact triangle in $\d(R)$ with $L,N\in\X$.
Then there exist exact triangles $Z_1\to L\oplus L'\to Y_1\rightsquigarrow$ and $Z_2\to N\oplus N'\to Y_2\rightsquigarrow$ in $\d(R)$ such that $Z_1,Z_2\in\ZZ$ and $Y_1,Y_2\in\Y$.
In view of (2), we may assume that $\sup Y_1\le0$.
The octahedral axiom yields the following two commutative diagrams of exact triangles in $\d(R)$.
$$
\xymatrix@R-1pc{
Z_1\ar[r]\ar@{=}[d]&L\oplus L'\ar[r]\ar[d]^-{\left(\begin{smallmatrix}f&0\\0&1\end{smallmatrix}\right)}&Y_1\ar[r]\ar[d]&Z_1[1]\ar@{=}[d]\\
Z_1\ar[r]\ar[d]&M\oplus L'\ar[r]_-{\left(\begin{smallmatrix}p&q\end{smallmatrix}\right)}\ar@{=}[d]&A\ar[r]_-r\ar[d]^-h&Z_1[1]\ar[d]\\
L\oplus L'\ar[r]^-{\left(\begin{smallmatrix}f&0\\0&1\end{smallmatrix}\right)}\ar[d]&M\oplus L'\ar[r]^-{\left(\begin{smallmatrix}g&0\end{smallmatrix}\right)}\ar[d]&N\ar[r]\ar@{=}[d]&(L\oplus L')[1]\ar[d]\\
Y_1\ar[r]&A\ar[r]^-h&N\ar[r]^-k&Y_1[1]
}
\qquad
\xymatrix@R-1pc{
A\oplus N'\ar[r]^-{\left(\begin{smallmatrix}h&0\\0&1\end{smallmatrix}\right)}\ar@{=}[d]&N\oplus N'\ar[r]^-{\left(\begin{smallmatrix}k&0\end{smallmatrix}\right)}\ar[d]&Y_1[1]\ar[r]\ar[d]&(A\oplus N')[1]\ar@{=}[d]\\
A\oplus N'\ar[r]\ar[d]&Y_2\ar[r]\ar@{=}[d]&B[1]\ar[r]\ar[d]&(A\oplus N')[1]\ar[d]\\
N\oplus N'\ar[r]\ar[d]&Y_2\ar[r]\ar[d]&Z_2[1]\ar[r]\ar@{=}[d]&(N\oplus N')[1]\ar[d]\\
Y_1[1]\ar[r]&B[1]\ar[r]&Z_2[1]\ar[r]^-{\delta[1]}&Y_1[2]
}
$$
The bottom row of the right diagram induces an exact triangle $\sigma:Y_1\to B\to Z_2\xrightarrow{\delta}Y_1[1]$ in $\d(R)$.
We have $\delta\in\Hom_{\d(R)}(Z_2,Y_1[1])\cong\Ext_R^1(Z_2,Y_1)=0$ by Lemma \ref{66}(1).
Hence $\sigma$ splits (see \cite[Corollary 1.2.7]{N2}), which gives an isomorphism $B\cong Y_1\oplus Z_2$.
An exact triangle $Y_1\oplus Z_2\to A\oplus N'\to Y_2\rightsquigarrow$ is induced from the second row of the right diagram.
Applying the octahedral axiom again, we obtain commutative diagrams
$$
\xymatrix@R-1.5pc@C3pc{
Z_2\ar[r]^-{\left(\begin{smallmatrix}0\\1\end{smallmatrix}\right)}\ar@{=}[d]&Y_1\oplus Z_2\ar[r]^-{\left(\begin{smallmatrix}1&0\end{smallmatrix}\right)}\ar[d]&Y_1\ar@{~>}[r]\ar[d]&\\
Z_2\ar[r]\ar[d]&A\oplus N'\ar[r]\ar@{=}[d]&Y\ar@{~>}[r]\ar[d]&\\
Y_1\oplus Z_2\ar[r]\ar[d]&A\oplus N'\ar[r]\ar[d]&Y_2\ar@{~>}[r]\ar@{=}[d]&\\
Y_1\ar[r]&Y\ar[r]&Y_2\ar@{~>}[r]&
}
\qquad
\xymatrix@R-1.5pc@C3pc{
M\oplus L'\oplus N'\ar[r]^-{\left(\begin{smallmatrix}p&q&0\\0&0&1\end{smallmatrix}\right)}\ar@{=}[d]&A\oplus N'\ar[r]^-{\left(\begin{smallmatrix}r&0\end{smallmatrix}\right)}\ar[d]&Z_1[1]\ar@{~>}[r]\ar[d]&\\
M\oplus L'\oplus N'\ar[r]\ar[d]&Y\ar[r]\ar@{=}[d]&Z[1]\ar@{~>}[r]\ar[d]&\\
A\oplus N'\ar[r]\ar[d]&Y\ar[r]\ar[d]&Z_2[1]\ar@{~>}[r]\ar@{=}[d]&\\
Z_1[1]\ar[r]&Z[1]\ar[r]&Z_2[1]\ar@{~>}[r]&
}
$$
of exact triangles.
We obtain exact triangles $Y_1\to Y\to Y_2\rightsquigarrow$ and $Z_1\to Z\to Z_2\rightsquigarrow$, which imply $Y\in\Y$ and $Z\in\ZZ$.
We also have an exact triangle $Z\to M\oplus L'\oplus N'\to Y\rightsquigarrow$, which shows that $M$ belongs to $\X$.
\end{proof}

In the following lemma, we investigate syzygies and cosyzygies for complexes.

\begin{lem}\label{69}
Let $X$ be an object of $\d(R)$.
Then the following two statements hold true.
\begin{enumerate}[\rm(1)]
\item
There exists an exact triangle $Y\to E\to X\rightsquigarrow$ in $\d(R)$ such that $E\in\E_R$, $\sup E\le\sup\{\sup X,0\}$ and $\sup Y\le0$.
If $\sup X\le0$, then $E$ is isomorphic in $\d(R)$ to a projective $R$-module.
\item
Suppose that $\gdim_RX\le0$.
Then there exists an exact triangle $X\to P\to Y\rightsquigarrow$ in $\d(R)$ such that $P$ is a projective $R$-module and $\gdim_RY\le0$.
\end{enumerate}
\end{lem}

\begin{proof}
(1) Put $u=\sup\{\sup X,0\}$.
Thanks to \cite[(A.3.2)]{C}, we can choose a complex $F=(\cdots\to F^{u-1}\to F^u\to0)$ of finitely generated projective $R$-modules which is isomorphic to $X$ in $\d(R)$.
(In the case $\sup X<0$, we can put $F^i=0$ for all integers $i$ such that $\sup X+1\le i\le0=u$.)
As $u\ge0$, we can take the truncation $E=(0\to F^0\to\cdots\to F^u\to0)$ of $F$.
We have $\sup E\le u$, and $E\in\E_R$ by Proposition \ref{27}(6).
Taking the truncation $X'=(\cdots\to F^{-2}\to F^{-1}\to0)$ of $F$, we get an exact triangle $E\to X\to X'\rightsquigarrow$ in $\d(R)$.
Setting $Y=X'[-1]$, we get an exact triangle $Y\to E\to X\rightsquigarrow$ in $\d(R)$, and it holds that $\sup Y\le0$.

Suppose that $\sup X\le0$.
Then $u=\sup\{\sup X,0\}=0$.
Therefore, we have $E=(0\to F^0\to0)$, which is isomorphic in $\d(R)$ to the finitely generated projective $R$-module $F^0$.

(2) In view of Remark \ref{40}(2), we can take an exact triangle $X\to X'\to E\rightsquigarrow$ in $\d(R)$ such that $\sup X'\le0$ and $E\in\E_R$.
It is observed from Lemma \ref{66}(3) and \cite[(2.3.10)]{C} that $\gdim X'\le0$.
By Lemma \ref{66}(4), we may assume that $X'$ is a totally reflexive $R$-module.
Hence, there is an exact sequence $0\to X'\to P\to\syz^{-1}X'\to0$ in $\mod R$ such that $P$ is projective, which induces an exact triangle $X'\to P\to\syz^{-1}X'\rightsquigarrow$ in $\d(R)$.
By the octahedral axiom, we obtain a commutative diagram
$$
\xymatrix@R-1.5pc@C5pc{
X\ar[r]\ar@{=}[d]&X'\ar[r]\ar[d]&E\ar[r]\ar[d]&X[1]\ar@{=}[d]\\
X\ar[r]\ar[d]&P\ar[r]\ar@{=}[d]&Y\ar[r]\ar[d]&X[1]\ar[d]\\
X'\ar[r]\ar[d]&P\ar[r]\ar[d]&\syz^{-1}X'\ar[r]\ar@{=}[d]&X'[1]\ar[d]\\
E\ar[r]&Y\ar[r]&\syz^{-1}X'\ar[r]&E[1]
}
$$
of exact triangles in $\d(R)$.
Using Lemma \ref{66}(3) and \cite[(2.3.10)]{C} again, we see from the bottom row that $\gdim Y\le0$.
Thus the second row in the above diagram is such an exact triangle as in the assertion.
\end{proof}

To show our next proposition, we need to prepare one more lemma.

\begin{lem}\label{68}
Let $Z\to X\oplus W\to Y\rightsquigarrow$ be an exact triangle in $\d(R)$.
Then there exists an exact triangle $Z\to X'\oplus W\to Y'\rightsquigarrow$ in $\d(R)$ such that $\res_{\d(R)}X=\res_{\d(R)}X'$, $\res_{\d(R)}Y=\res_{\d(R)}Y'$ and $\sup X'\le0$.
\end{lem}

\begin{proof}
By Remark \ref{40}(2), there exists an exact triangle $X\xrightarrow{a}X'\xrightarrow{b}E\rightsquigarrow$ in $\d(R)$ such that $\sup X'\le0$ and $E\in\E_R$.
Note then that $\res X=\res X'$.
The octahedral axiom gives rise to a commutative diagram
$$
\xymatrix@R-1.5pc@C5pc{
Z\ar[r]\ar@{=}[d]&X\oplus W\ar[r]\ar[d]&Y\ar[r]\ar[d]&Z[1]\ar@{=}[d]\\
Z\ar[r]\ar[d]&X'\oplus W\ar[r]\ar@{=}[d]&Y'\ar[r]\ar[d]&Z[1]\ar[d]\\
X\oplus W\ar[r]^-{\left(\begin{smallmatrix}a&0\\0&1\end{smallmatrix}\right)}\ar[d]&X'\oplus W\ar[r]^-{\left(\begin{smallmatrix}b&0\end{smallmatrix}\right)}\ar[d]&E\ar[r]\ar@{=}[d]&(X\oplus W)[1]\ar[d]\\
Y\ar[r]&Y'\ar[r]&E\ar[r]&Y[1]
}
$$
of exact triangles in $\d(R)$.
From the bottom row in the commutative diagram we observe that $\res Y=\res Y'$.
Thus the second row in the commutative diagram is an exact triangle as in the assertion of the lemma.
\end{proof}

We denote by $\g(R)$ the subcategory of $\d(R)$ consisting of objects $X$ satisfying the inequality $\gdim_RX\le0$.
We can now prove the proposition below, which includes a derived category version of \cite[Proposition 7.3]{crspd}.

\begin{prop}\label{70}
Let $\Y$ and $\ZZ$ be resolving subcategories of $\d(R)$ such that $\Y\subseteq\k(R)$ and $\ZZ\subseteq\g(R)$.
Then:
$$
\Y=\res_{\d(R)}(\Y\cup\ZZ)\cap\k(R),\qquad
\ZZ=\res_{\d(R)}(\Y\cup\ZZ)\cap\g(R).
$$
\end{prop}

\begin{proof}
Let us begin with the first equality.
The inclusion $(\subseteq)$ is clear.
To show $(\supseteq)$, pick $X\in\res(\Y\cup\ZZ)\cap\k(R)$.
According to Lemma \ref{67}(3), there is an exact triangle $Z\to X\oplus W\to Y\rightsquigarrow$ in $\d(R)$ with $Z\in\ZZ$ and $Y\in\Y$.
What we want to show is that $X$ is in $\Y$.
For this purpose, thanks to Lemma \ref{68}, we may assume $\sup X\le0$.
Lemma \ref{69}(2) yields an exact triangle $Z\to P\to V\rightsquigarrow$ in $\d(R)$ such that $P$ is a projective module and $\gdim V\le0$.
The octahedral axiom gives the following commutative diagrams of exact triangles in $\d(R)$.
$$
\xymatrix@R-1pc{
Y[-1]\ar[r]\ar@{=}[d]&Z\ar[r]\ar[d]&X\oplus W\ar[r]\ar[d]^-{\left(\begin{smallmatrix}f&g\end{smallmatrix}\right)}&Y\ar@{=}[d]\\
Y[-1]\ar[r]\ar[d]&P\ar[r]\ar@{=}[d]&Y'\ar[r]\ar[d]&Y\ar[d]\\
Z\ar[r]\ar[d]&P\ar[r]\ar[d]&V\ar[r]\ar@{=}[d]&Z[1]\ar[d]\\
X\oplus W\ar[r]^-{\left(\begin{smallmatrix}f&g\end{smallmatrix}\right)}&Y'\ar[r]&V\ar[r]&X[1]
}\qquad\qquad
\xymatrix@R-1pc{
W\ar[r]^-{\left(\begin{smallmatrix}0\\1\end{smallmatrix}\right)}\ar@{=}[d]&X\oplus W\ar[r]^-{\left(\begin{smallmatrix}1&0\end{smallmatrix}\right)}\ar[d]^-{\left(\begin{smallmatrix}f&g\end{smallmatrix}\right)}&X\ar[r]\ar[d]^-l&W[1]\ar@{=}[d]\\
W\ar[r]\ar[d]&Y'\ar[r]^-h\ar@{=}[d]&U\ar[r]\ar[d]&W[1]\ar[d]\\
X\oplus W\ar[r]^-{\left(\begin{smallmatrix}f&g\end{smallmatrix}\right)}\ar[d]&Y'\ar[r]\ar[d]&V\ar[r]\ar@{=}[d]&(X\oplus W)[1]\ar[d]\\
X\ar[r]^-l&U\ar[r]&V\ar[r]^-\delta&X[1]
}
$$
From the exact triangle $P\to Y'\to Y\rightsquigarrow$ we see that $Y'$ is in $\Y$.
The equality $h(f\ g)=l(1\ 0)$ implies $l=hf$.
Also, by Lemma \ref{66}(1) we have $\delta\in\Hom_{\d(R)}(V,X[1])=\Ext_R^1(V,X)=0$.
There are commutative diagrams
$$
\xymatrix@R-1pc{
X\ar[r]^-l\ar@{=}[d]&U\ar[r]\ar[d]_-\cong^-{\left(\begin{smallmatrix}s\\t\end{smallmatrix}\right)}&V\ar[r]^-0\ar@{=}[d]&X[1]\ar@{=}[d]\\
X\ar[r]^-{\left(\begin{smallmatrix}1\\0\end{smallmatrix}\right)}&X\oplus V\ar[r]^-{\left(\begin{smallmatrix}0&1\end{smallmatrix}\right)}&V\ar[r]&X[1]
}\qquad\qquad
\xymatrix@R-1pc{
X\ar[r]^-f\ar@{=}[d]&Y'\ar[r]^-k\ar[d]_-\cong^-{\left(\begin{smallmatrix}sh\\k\end{smallmatrix}\right)}&C\ar[r]\ar@{=}[d]&X[1]\ar@{=}[d]\\
X\ar[r]^-{\left(\begin{smallmatrix}1\\0\end{smallmatrix}\right)}&X\oplus C\ar[r]^-{\left(\begin{smallmatrix}0&1\end{smallmatrix}\right)}&C\ar[r]&X[1]
}$$
of exact triangles in $\d(R)$.
Indeed, we get the left diagram by \cite[Proof of Corollary 1.2.7]{N2}, while the equality $\binom{1}{0}=\binom{s}{t}l$ coming from the commutativity of the left diagram shows $1=sl=shf$, so that the right diagram is obtained by \cite[Remark 1.2.9]{N2}.
The isomorphism $X\oplus C\cong Y'$ shows that $X$ belongs to $\Y$, as desired.

Next, we prove the second equality (the proof has the same stream as that of the first equality, but there are actually various different places).
It is evident that $(\subseteq)$ holds.
To prove $(\supseteq)$, let $X$ be an object in the subcategory $\res(\Y\cup\ZZ)\cap\g(R)$.
By Lemma \ref{67}(2)(3), there is an exact triangle $Z\to X\oplus W\to Y\rightsquigarrow$ in $\d(R)$ with $Z\in\ZZ$, $Y\in\Y$ and $\sup Y\le0$.
We want to show that $X$ belongs to $\ZZ$.
Using Lemma \ref{69}(1), we get an exact triangle $Y'\to P\to Y\rightsquigarrow$ in $\d(R)$ such that $P$ is a projective module and $\sup Y'\le0$.
Note then that $Y'$ is in $\Y$.
The octahedral axiom gives the following commutative diagrams of exact triangles in $\d(R)$.
$$
\xymatrix@R-1pc{
X\oplus W\ar[r]\ar@{=}[d]&Y\ar[r]\ar[d]&Z[1]\ar[r]\ar[d]&(X\oplus W)[1]\ar@{=}[d]\\
X\oplus W\ar[r]^-{\left(\begin{smallmatrix}f&g\end{smallmatrix}\right)}\ar[d]&Y'[1]\ar[r]\ar@{=}[d]&Z'[1]\ar[r]\ar[d]&(X\oplus W)[1]\ar[d]\\
Y\ar[r]\ar[d]&Y'[1]\ar[r]\ar[d]&P[1]\ar[r]\ar@{=}[d]&Y[1]\ar[d]\\
Z[1]\ar[r]&Z'[1]\ar[r]&P[1]\ar[r]&Z[2]
}\qquad
\xymatrix@R-1pc{
X\ar[r]^-{\left(\begin{smallmatrix}1\\0\end{smallmatrix}\right)}\ar@{=}[d]&X\oplus W\ar[r]^-{\left(\begin{smallmatrix}0&1\end{smallmatrix}\right)}\ar[d]^-{\left(\begin{smallmatrix}f&g\end{smallmatrix}\right)}&W\ar[r]\ar[d]&X[1]\ar@{=}[d]\\
X\ar[r]^-f\ar[d]^-{\left(\begin{smallmatrix}1\\0\end{smallmatrix}\right)}&Y'[1]\ar[r]\ar@{=}[d]&V\ar[r]^-h\ar[d]^-k&X[1]\ar[d]^-{\left(\begin{smallmatrix}1\\0\end{smallmatrix}\right)}\\
X\oplus W\ar[r]^-{\left(\begin{smallmatrix}f&g\end{smallmatrix}\right)}\ar[d]&Y'[1]\ar[r]\ar[d]&Z'[1]\ar[r]^-{\left(\begin{smallmatrix}p\\q\end{smallmatrix}\right)}\ar@{=}[d]&X[1]\oplus W[1]\ar[d]\\
W\ar[r]&V\ar[r]^-k&Z'[1]\ar[r]&W[1]
}
$$
The induced exact triangle $Z\to Z'\to P\rightsquigarrow$ shows that the object $Z'$ belongs to $\ZZ$.
Lemma \ref{66}(1) implies $f\in\Hom_{\d(R)}(X,Y'[1])=\Ext_R^1(X,Y')=0$.
By \cite[Proof of Corollary 1.2.7]{N2} there is a commutative diagram
$$
\xymatrix@R-1pc{
X\ar[r]\ar@{=}[d]&Y'[1]\ar[r]^-{\left(\begin{smallmatrix}1\\0\end{smallmatrix}\right)}\ar@{=}[d]&Y'[1]\oplus X[1]\ar[r]^-{\left(\begin{smallmatrix}0&1\end{smallmatrix}\right)}\ar[d]_-{\left(\begin{smallmatrix}s&t\end{smallmatrix}\right)}^-\cong&X[1]\ar@{=}[d]\\
X\ar[r]^-0&Y'[1]\ar[r]&V\ar[r]^-h&X[1]\\
}
$$
of exact triangles in $\d(R)$, which gives $ht=1$.
The equality $\binom{p}{q}k=\binom{1}{0}h$ implies $h=pk$.
Hence $pkt=1$, and it follows from \cite[Lemma 1.2.8]{N2} that $X[1]$ is isomorphic in $\d(R)$ to a direct summand of $Z'[1]$.
This implies that $X$ is isomorphic in $\d(R)$ to a direct summand of $Z'$.
Since $Z'$ belongs to $\ZZ$, so does $X$.
\end{proof}

Recall that $R$ is said to be {\em locally} a {\em complete intersection} if the local ring $R_\p$ is a complete intersection for each prime ideal $\p$ of $R$.
We simply say that $R$ is a {\em complete intersection} if it is locally a complete intersection.
The exact triangle appearing in the third assertion of the following proposition is regarded as a derived category version of a {\em finite projective hull} in the sense of Auslander and Buchweitz \cite{AB}.

\begin{prop}\label{71}
Let $R$ be a complete intersection.
Then the following statements hold true.
\begin{enumerate}[\rm(1)]
\item
Let $M$ be a maximal Cohen--Macaulay $R$-module.
Then the cosyzygy $\syz_R^{-1}M$ belongs to $\res_{\mod R}M$.
\item
For every maximal Cohen--Macaulay complex $X\in\d(R)$, there exists an exact triangle $X\to P\to Y\rightsquigarrow$ in $\d(R)$ such that $P$ is a projective module and $Y$ belongs to the resolving closure $\res_{\d(R)}X$.
\item
Each $X\in\d(R)$ admits an exact triangle $X\to P\to Y\rightsquigarrow$ with $P\in\k(R)$ and $Y\in(\res_{\d(R)}X)\cap\c(R)$.
\end{enumerate}
\end{prop}

\begin{proof}
(1) Fix a prime ideal $\p$ of $R$.
Then $M_\p$ is a maximal Cohen--Macaulay $R_\p$-module.
In $\mod R_\p$ we have
$$
(\syz_R^{-1}M)_\p\approx\syz_{R_\p}^{-1}(M_\p)\in\res_{\mod R_\p}M_\p\subseteq\add_{\mod R_\p}(\res_{\mod R}M)_\p.
$$
Here, by $A\approx B$ we mean $A\cong B$ up to free summands.
The containment and the inclusion follow from \cite[Theorem 4.15]{radius} and \cite[Lemma 3.2(1)]{crspd}, respectively.
By \cite[Proposition 3.3]{crspd}, we get $\syz_R^{-1}M\in\res_{\mod R}M$.

(2) By Remark \ref{40}(2) there exists an exact triangle $X\to X'\to E\rightsquigarrow$ in $\d(R)$ such that $\sup X'\le0$ and $E\in\E_R$.
As $X$ belongs to $\c(R)$, so does $X'$ by Proposition \ref{58}(2).
By Lemma \ref{66}(4)(5), we may assume that $X'$ is a totally reflexive module.
There is an exact sequence $0\to X'\to P\to\syz^{-1}X'\to0$ in $\mod R$ with $P$ projective, and $\syz^{-1}X'$ belongs to $\res_{\mod R}X'$ by (1).
The octahedral axiom gives a commutative diagram
$$
\xymatrix@R-1.5pc@C5pc{
X\ar[r]\ar@{=}[d]&X'\ar[r]\ar[d]&E\ar[r]\ar[d]&X[1]\ar@{=}[d]\\
X\ar[r]\ar[d]&P\ar[r]\ar@{=}[d]&Y\ar[r]\ar[d]&X[1]\ar[d]\\
X'\ar[r]\ar[d]&P\ar[r]\ar[d]&\syz^{-1}X'\ar[r]\ar@{=}[d]&X'[1]\ar[d]\\
E\ar[r]&Y\ar[r]&\syz^{-1}X'\ar[r]&E[1]
}
$$
of exact triangles in $\d(R)$.
It is seen from the bottom row that $Y$ belongs to $\res_{\d(R)}X'$, which coincides with $\res_{\d(R)}X$ (by the first row).
Thus, the second row provides such an exact triangle as we want.

(3) We may assume that $\sup X\le0$.
In fact, by Remark \ref{40}(2) there exists an exact triangle $X\to X'\to E\rightsquigarrow$ in $\d(R)$ such that $\sup X'\le0$ and $E\in\E_R$.
Suppose that we have got an exact triangle $X'\to P\to C\rightsquigarrow$ in $\d(R)$ such that $P\in\k(R)$ and $C\in(\res X')\cap\c(R)$.
Then we have $C\in(\res X)\cap\c(R)$ as $\res X'=\res X$.
The octahedral axiom gives rise to the following commutative diagram of exact triangles in $\d(R)$.
$$
\xymatrix@R-1.5pc@C5pc{
X\ar[r]\ar@{=}[d]&X'\ar[r]\ar[d]&E\ar[r]\ar[d]&X[1]\ar@{=}[d]\\
X\ar[r]\ar[d]&P\ar[r]\ar@{=}[d]&Y\ar[r]\ar[d]&X[1]\ar[d]\\
X'\ar[r]\ar[d]&P\ar[r]\ar[d]&C\ar[r]\ar@{=}[d]&X'[1]\ar[d]\\
E\ar[r]&Y\ar[r]&C\ar[r]&E[1]
}
$$
The bottom row in the above diagram shows that $Y$ belongs to $(\res X)\cap\c(R)$ by Proposition \ref{58}(2).
Consequently, the second row in the above diagram is such an exact triangle as in the assertion.

Since the ring $R$ is a complete intersection, it is Gorenstein.
By Lemma \ref{66}(5), the number $n:=\gdim_RX$ is finite.
We use induction on $n$.
Let $n\le0$.
Then $X$ is a maximal Cohen--Macaulay complex; see Lemma \ref{66}(5).
Thus the assertion follows from (2); note that $\res X=(\res X)\cap\c(R)$.
Let $n>0$.
Since $\sup X\le0$, by Lemma \ref{69}(1) there is an exact triangle $Y\to P\to X\rightsquigarrow$ in $\d(R)$ such that $P$ is a projective module and $\sup Y\le0$.
Note then that $Y\in\res X$, so that $\res Y\subseteq\res X$.
As $n-1\ge0$, Lemma \ref{66}(3) implies $\gdim Y\le\sup\{\gdim P,\gdim X-1\}=n-1$.
The induction hypothesis yields an exact triangle $Y\to K\to C\rightsquigarrow$ in $\d(R)$ such that $K\in\k(R)$ and $C\in(\res Y)\cap\c(R)$.
By the octahedral axiom, we get the commutative diagram (a) of exact triangles in $\d(R)$.
The second row in (a) shows $C'\in\c(R)$ and $\res C'\subseteq\res C$.
By (2) there is an exact triangle $C'\to Q\to C''\to C'[1]$ in $\d(R)$ such that $Q$ is a projective module and $C''\in\res C'$.
Applying the octahedral axiom again, we obtain the commutative diagram (b) of exact triangles in $\d(R)$.
$$
{\rm(a)}:\xymatrix@R-1.5pc{
C[-1]\ar[r]\ar@{=}[d]&Y\ar[r]\ar[d]&K\ar[r]\ar[d]&C\ar@{=}[d]\\
C[-1]\ar[r]\ar[d]&P\ar[r]\ar@{=}[d]&C'\ar[r]\ar[d]&C\ar[d]\\
Y\ar[r]\ar[d]&P\ar[r]\ar[d]&X\ar[r]\ar@{=}[d]&Y[1]\ar[d]\\
K\ar[r]&C'\ar[r]&X\ar[r]&K[1]
}
\qquad\qquad
{\rm(b)}:\xymatrix@R-1.5pc{
K\ar[r]\ar@{=}[d]&C'\ar[r]\ar[d]&X\ar[r]\ar[d]&K[1]\ar@{=}[d]\\
K\ar[r]\ar[d]&Q\ar[r]\ar@{=}[d]&K'\ar[r]\ar[d]&K[1]\ar[d]\\
C'\ar[r]\ar[d]&Q\ar[r]\ar[d]&C''\ar[r]\ar@{=}[d]&C'[1]\ar[d]\\
X\ar[r]&K'\ar[r]&C''\ar[r]&X[1]
}
$$
The second row in (b) shows $K'$ is in $\k(R)$.
We have $C''\in\res C'\subseteq\res C\subseteq(\res Y)\cap\c(R)\subseteq(\res X)\cap\c(R)$.
Consequently, the bottom row in (b) provides such an exact triangle as in the assertion.
\end{proof}

We record here a direct consequence of the second assertion of the above proposition.

\begin{cor}\label{85}
Let $R$ be a complete intersection.
Let $X\in\d(R)$ be a maximal Cohen--Macaulay complex.
Then $X$ belongs to the resolving closure $\res_{\d(R)}(X[-i])$ for every nonnegative integer $i$.
\end{cor}

\begin{proof}
Proposition \ref{71}(2) gives rise to an exact triangle $X\to P\to Y\rightsquigarrow$ in $\d(R)$ such that $P$ is a projective module and $Y\in\res X$.
An exact triangle $Y[-1]\to X\to P\rightsquigarrow$ is induced, which shows that $X$ is in $\res(Y[-1])$.
Proposition \ref{52}(2a) implies $Y[-1]\in(\res X)[-1]\subseteq\res(X[-1])$.
Hence $X\in\res(X[-1])$.
If $X\in\res(X[-j])$ for an integer $j$, then $X[-1]\in(\res(X[-j]))[-1]\subseteq\res(X[-j][-1])=\res(X[-j-1])$ by Proposition \ref{52}(2a) again, and we get $X\in\res(X[-1])\subseteq\res(X[-j-1])$.
It follows that $X$ belongs to $\res(X[-i])$ for all $i\ge0$.
\end{proof}

Now we have reached the stage to achieve the main purpose of this section; we shall state and prove the following theorem, which is viewed as a derived category version of \cite[Theorem 7.4]{crspd}.

\begin{thm}\label{72}
Suppose that $R$ is a complete intersection.
Then there are mutually inverse bijections
$$
\xymatrix{
{\left\{\begin{matrix}
\text{resolving subcategories}\\
\text{of $\d(R)$}
\end{matrix}\right\}}
\ar@<.7mm>[r]^-\phi&
{\left\{\begin{matrix}
\text{resolving subcategories}\\
\text{of $\d(R)$ contained in $\k(R)$}
\end{matrix}\right\}}\times
{\left\{\begin{matrix}
\text{resolving subcategories}\\
\text{of $\d(R)$ contained in $\c(R)$}
\end{matrix}\right\}},
\ar@<.7mm>[l]^-\psi&
}
$$
where the maps $\phi,\psi$ are given by $\phi(\X)=(\X\cap\k(R),\X\cap\c(R))$ and $\psi(\Y,\ZZ)=\res_{\d(R)}(\Y\cup\ZZ)$.
\end{thm}

\begin{proof}
Clearly, the maps $\phi,\psi$ are well-defined.
Lemma \ref{66}(5) implies $\g(R)=\c(R)$.
Proposition \ref{70} says $\phi\psi=\id$.
Let $\X$ be a resolving subcategory of $\d(R)$.
Then $\psi\phi(\X)=\res((\X\cap\k(R))\cup(\X\cap\c(R)))$ is clearly contained in $\X$.
Let $X$ be any object in $\X$.
It follows from Proposition \ref{71}(3) that there is an exact triangle $X\to P\to Y\rightsquigarrow$ in $\d(R)$ such that $P\in\k(R)$ and $Y\in(\res X)\cap\c(R)\subseteq\X\cap\c(R)$.
We see that $P$ is in $\X\cap\k(R)$, so that $X$ is in $\res((\X\cap\k(R))\cup(\X\cap\c(R)))$.
Thus $X$ belongs to $\psi\phi(\X)$, and we obtain $\psi\phi=\id$.
\end{proof}

\section{Classification of resolving subcategories and certain preaisles of $\d(R)$}

The main goal of this section is to give a complete classification of resolving subcategories of $\d(R)$ and preaisles of $\d(R)$ containing $R$ and closed under direct summands, in the case where $R$ belongs to a certain class of complete intersection rings.
First of all, applying the main result of the previous section, we prove the following theorem.
The bijections given in the theorem say that classifying the resolving subcategories of maximal Cohen--Macaulay complexes is equivalent to classifying the thick subcategories containing $R$.
The equality given in the theorem is a derived category version of \cite[Corollary 4.16]{radius}.

\begin{thm}\label{76}
Let $R$ be a complete intersection.
There are mutually inverse bijections and an equality
$$
\xymatrix{
{\left\{\begin{matrix}
\text{thick subcategories}\\
\text{of $\d(R)$ containing $R$}
\end{matrix}\right\}}
\ar@<.7mm>[rr]^-{(-)\cap\c(R)}&&
{\left\{\begin{matrix}
\text{resolving subcategories}\\
\text{of $\d(R)$ contained in $\c(R)$}
\end{matrix}\right\}}
={\left\{\begin{matrix}
\text{thick subcategories}\\
\text{of $\c(R)$ containing $R$}
\end{matrix}\right\}}.
\ar@<.7mm>[ll]^-{\thick_{\d(R)}(-)}&
}
$$
\end{thm}

\begin{proof}
We start by proving the equality, using the bijections.
By Proposition \ref{60}(1), it suffices to show that each resolving subcategory $\X$ of $\d(R)$ contained in $\c(R)$ is a thick subcategory of $\c(R)$, and for this it is enough to verify that for each exact triangle $A\to B\to C\rightsquigarrow$ in $\d(R)$ with $A,B,C\in\c(R)$, if $A$ and $B$ belong to $\X$, then so does $C$.
By the first assertion of the theorem we have $\X=(\thick_{\d(R)}\X)\cap\c(R)$.
Hence $A$ and $B$ are in $\thick_{\d(R)}\X$, and so is $C$.
It follows that $C$ belongs to $(\thick_{\d(R)}\X)\cap\c(R)=\X$, and we are done.

We proceed with showing the bijections.
Clearly, the two maps are well-defined.
Fix a thick subcategory $\X$ of $\d(R)$ containing $R$ and a resolving subcategory $\ZZ$ of $\d(R)$ contained in $\c(R)$.
Then $\X$ is a resolving subcategory of $\d(R)$, so that Theorem \ref{72} shows $\X=\psi\phi(\X)=\res((\X\cap\k(R))\cup(\X\cap\c(R)))$.
Since $\X$ is thick and contains $R$, it contains $\k(R)=\thick R$; see Proposition \ref{43}(3).
Hence $\X\cap\k(R)=\k(R)$, and
$$
\X
=\res(\k(R)\cup(\X\cap\c(R)))
\subseteq\thick(\k(R)\cup(\X\cap\c(R)))
=\thick(\X\cap\c(R))
\subseteq\X.
$$
Therefore, $\X=\thick(\X\cap\c(R))$.
On the other hand, applying Theorem \ref{72} again, we have
$$
(\k(R),\ZZ)=\phi\psi(\k(R),\ZZ)=(\res(\k(R)\cup\ZZ)\cap\k(R),\res(\k(R)\cup\ZZ)\cap\c(R)),
$$
which gives us the equality $\ZZ=\res(\k(R)\cup\ZZ)\cap\c(R)$.

We claim that $\res(\k(R)\cup\ZZ)$ is a thick subcategory of $\d(R)$.
Indeed, it suffices to verify that $\res(\k(R)\cup\ZZ)$ is closed under positive shifts.
Using Proposition \ref{52}(2b), we get equalities
\begin{equation}\label{75}
(\res(\k(R)\cup\ZZ))[1]
=\res((\k(R)\cup\ZZ)[1]\cup\{R[1]\})
=\res((\k(R)\cup\ZZ)[1])
=\res(\k(R)\cup\ZZ[1]).
\end{equation}
Pick $Z\in\ZZ$.
Then $Z$ is maximal Cohen--Macaulay, and Corollary \ref{85} implies $Z\in\res(Z[-1])$.
We obtain
$$
Z[1]\in(\res(Z[-1]))[1]=\res\{Z,R[1]\}\subseteq\res(\k(R)\cup\ZZ),
$$
where for the equality we apply Proposition \ref{52}(2b) again.
It follows that $\ZZ[1]$ is contained in $\res(\k(R)\cup\ZZ)$, which and \eqref{75} yield that $(\res(\k(R)\cup\ZZ))[1]$ is contained in $\res(\k(R)\cup\ZZ)$.
Thus the claim follows.

The above claim guarantees that $\res(\k(R)\cup\ZZ)=\thick(\k(R)\cup\ZZ)=\thick\ZZ$, and we obtain an equality $\ZZ=(\thick\ZZ)\cap\c(R)$.
Now we conclude that the two maps in the assertion are mutually inverse bijections.
\end{proof}

Denote by $\s(R)$ the {\em singularity category} $\ds(R)$ of $R$, that is, the Verdier quotient of $\d(R)$ by $\k(R)$.
The following lemma enables us to obtain a classification of preaisles in the next theorem.

\begin{lem}\label{74}
\begin{enumerate}[\rm(1)]
\item
There is a natural one-to-one correspondence
$$
{\left\{\begin{matrix}
\text{thick subcategories}\\
\text{of $\s(R)$}
\end{matrix}\right\}}
\cong
{\left\{\begin{matrix}
\text{thick subcategories}\\
\text{of $\d(R)$ containing $R$}
\end{matrix}\right\}}.
$$
\item
Suppose that $R$ is Gorenstein.
Assigning to each subcategory $\X$ of $\d(R)$ the subcategory $\rhom_R(\X,R)$ of $\d(R)$ consisting of objects of the form $\rhom_R(X,R)$ with $X\in\X$, one gets a one-to-one correspondence
$$
{\left\{\begin{matrix}
\text{preaisles of $\d(R)$}\\
\text{containing $R$ and closed}\\
\text{under direct summands}
\end{matrix}\right\}}
\cong
{\left\{\begin{matrix}
\text{precoaisles of $\d(R)$}\\
\text{containing $R$ and closed}\\
\text{under direct summands}
\end{matrix}\right\}}
=
{\left\{\begin{matrix}
\text{resolving}\\
\text{subcategories}\\
\text{of $\d(R)$}
\end{matrix}\right\}}.
$$
\end{enumerate}
\end{lem}

\begin{proof}
(1) The assertion comes from a general fact on Verdier quotients; see \cite[Chapitre II, Proposition 2.3.1]{V}, \cite[Lemma 3.1]{thd} and \cite[Lemma 10.5]{dlr}.

(2) The equality follows by definition.
As $R$ is Gorenstein, for each $C\in\d(R)$ the complex $\rhom(C,R)$ is bounded, so that it is in $\d(R)$; see \cite[(2.3.8)]{C} and Lemma \ref{66}(5).
Thus, the contravariant exact (additive) functor $\rhom(-,R)$ gives a duality of $\d(R)$.
Since $\rhom(R,R)=R$, we can easily get the bijection.
\end{proof}

Combining Theorems \ref{2}, \ref{72}, \ref{76} and Lemma \ref{74}, we obtain the theorem below.
Thanks to this theorem, to classify the resolving subcategories of $\d(R)$ we have only to classify the thick subcategories of $\s(R)$.

\begin{thm}\label{88}
Let $R$ be a complete intersection.
Then there are one-to-one correspondences
$$
{\left\{\begin{matrix}
\text{preaisles of $\d(R)$}\\
\text{containing $R$ and closed}\\
\text{under direct summands}
\end{matrix}\right\}}
\cong
{\left\{\begin{matrix}
\text{resolving}\\
\text{subcategories}\\
\text{of $\d(R)$}
\end{matrix}\right\}}
\cong
{\left\{\begin{matrix}
\text{order-preserving}\\
\text{maps from $\spec R$}\\
\text{to $\N\cup\{\infty\}$}
\end{matrix}\right\}}\times
{\left\{\begin{matrix}
\text{thick}\\
\text{subcategories}\\
\text{of $\s(R)$}
\end{matrix}\right\}}.
$$
\end{thm}

We need to recall the definition of a hypersurface,  related notions and basic properties.

\begin{dfn}
\begin{enumerate}[(1)]
\item
Let $R$ be a local ring.
We denote by $\codim R$ and $\codepth R$ the {\em codimension} and the {\em codepth} of $R$, respectively, that is to say, $\codim R=\edim R-\dim R$ and $\codepth R=\edim R-\depth R$.
\item
For a local ring $R$, the following three conditions are equivalent; see \cite[\S5.1]{A}.\\
(a) There is an inequality $\codepth R\le1$.\quad
(b) The local ring $R$ is Cohen--Macaulay and $\codim R\le1$.\\
(c) The completion of $R$ is isomorphic to the residue ring of a regular local ring by a single element.\\
When one of these equivalent conditions holds, the local ring $R$ is called a {\em hypersurface}.
By \cite[Corollary 7.4.6]{A}, if a local ring $R$ is a hypersurface, then so is the local ring $R_\p$ for every prime ideal $\p$ of $R$. 
\item
We say that $R$ is {\em locally} a {\em hypersurface} if the localization $R_\p$ is a hypersurface local ring for every prime ideal $\p$ of $R$.
In what follows, we simply call $R$ a {\em hypersurface} if it is locally a hypersurface.
\end{enumerate}
\end{dfn}

To state theorems of Stevenson, Dao and Takahashi, and ours, we establish the following setup.

\begin{setup}\label{96}
Let $(R,V)$ be a pair that satisfies either of the following two conditions.
\begin{enumerate}[(1)]
\item
$R$ is a hypersurface and $V=\sing R$.
\item
$R=S/(\aa)$ where $S$ is a regular ring of finite Krull dimension and $\aa=a_1,\dots,a_c$ is an $S$-regular sequence, and $V=\sing Y=\{y\in Y\mid\text{$\OO_{Y,y}$ is not regular}\}$ where $X=\PP_S^{c-1}=\Proj(S[x_1,\dots,x_c])$ and $Y$ is the zero subscheme of $a_1x_1+\cdots+a_cx_c\in\Gamma(X,\OO_X(1))$.
\end{enumerate}
\end{setup}

\begin{rem}
In view of \cite[Theorem 2.10]{BW}, \cite[Corollary 7.9 and the beginning of Section 10]{St}, Setup \ref{96}(2) is equivalent to the following condition.
\begin{enumerate}[(2')]
\item
$R=S/(\aa)$ where $S$ is a regular ring of finite Krull dimension and $\aa=a_1,\dots,a_c$ is an $S$-regular sequence, and $V=\sing Y$ where $Y=\Proj G$ and $G=S[x_1,\dots,x_c]/(f)$ is the {\em generic hypersurface}, that is, the homogeneous $S$-algebra ($\deg(s)=0$ for $s\in S$ and $\deg(x_i)=1$ for $i=1,\dots,c$) defined as the quotient ring of the polynomial ring over $S$ in $c$ variables $x_1,\dots,x_c$ by the polynomial $f=a_1x_1+\cdots+a_cx_c$.
\end{enumerate}
\end{rem}

The following is the theorem of Stevenson \cite{St}.
Its assertion for Setup \ref{96}(1) is shown in \cite[Theorem 6.13]{St}, whose local case is \cite[Theorem 3.13(1)]{thd}.
Its assertion for Setup \ref{96}(2) is shown in \cite[Theorem 8.8]{St}.
The first one-to-one correspondence in the theorem is the one given in Lemma \ref{74}(1).

\begin{thm}[Stevenson]\label{77}
Let $(R,V)$ be as in Setup \ref{96}.
Then there is a one-to-one correspondence
$$
{\left\{\begin{matrix}
\text{thick subcategories}\\
\text{of $\s(R)$}
\end{matrix}\right\}}
\cong
{\left\{\begin{matrix}
\text{thick subcategories}\\
\text{of $\d(R)$ containing $R$}
\end{matrix}\right\}}
\overset{\rm(a)}{\cong}
{\left\{\begin{matrix}
\text{specialization-closed}\\
\text{subsets of $V$}
\end{matrix}\right\}}.
$$
\end{thm}

We obtain the following bijections by applying Theorems \ref{88} and \ref{77}.

\begin{cor}\label{41}
Let $(R,V)$ be as in Setup \ref{96}.
Then there are one-to-one correspondences
$$
{\left\{\begin{matrix}
\text{preaisles of $\d(R)$}\\
\text{containing $R$ and closed}\\
\text{under direct summands}
\end{matrix}\right\}}
\cong
{\left\{\begin{matrix}
\text{resolving}\\
\text{subcategories}\\
\text{of $\d(R)$}
\end{matrix}\right\}}
\overset{\rm(b)}{\cong}
{\left\{\begin{matrix}
\text{order-preserving}\\
\text{maps from $\spec R$}\\
\text{to $\N\cup\{\infty\}$}
\end{matrix}\right\}}\times
{\left\{\begin{matrix}
\text{specialization-closed}\\
\text{subsets of $V$}
\end{matrix}\right\}}.
$$
\end{cor}

\begin{rem}\label{98}
By Proposition \ref{60}(1), thick subcategories of $\d(R)$ containing $R$ are resolving subcategories of $\d(R)$.
Restricting the bijection (b) in Corollary \ref{41} to the thick subcategories of $\d(R)$ containing $R$, one recovers the bijection (a) in Theorem \ref{77}.
In fact, let $\X$ be a thick subcategory of $\d(R)$ containing $R$.
Then $\X$ contains $\thick_{\d(R)}R=\k(R)$ by Proposition \ref{43}(3).
Hence $\X\cap\k(R)=\k(R)$, and $\sup_{X\in\X\cap\k(R)}\{\pd X_\p\}=\infty$ for each prime ideal $\p$ of $R$.
Note that this actually holds for $\X:=\k(R)$.
Define the map $\xi:\spec R\to\N\cup\{\infty\}$ by $\xi(\p)=\infty$ for every $\p\in\spec R$.
It is observed along the way to get Corollary \ref{41} that the bijection (b) in Corollary \ref{41} restricts to the bijection below, which can be identified with the bijection (a) in Theorem \ref{77}.
$$
\{\text{thick subcategories of $\d(R)$ containing $R$}\}\cong\{\xi\}\times\{\text{specialization-closed subsets of $V$}\}.
$$
\end{rem}

It may be interesting to consider the following quetion which is similar to Question \ref{87}.

\begin{ques}
Let $R$ be as in Corollary \ref{41}.
Hence $R$ is Cohen--Macaulay, so it is CM-excellent.
Assume that $R$ has finite Krull dimension.
Then, the aisles of $\d(R)$ containing $R$ and closed under direct summands are classified by both Theorem \ref{48} and Corollary \ref{41}.
Are these two classifications (essentially) the same?
\end{ques}

We close the section by giving, in the case of a hypersurface, an explicit description in terms of NE-loci of the restriction of the one-to-one correspondence (b) in Corollary \ref{41} to the resolving subcategories of maximal Cohen--Macaulay complexes.
For a subcategory $\C$ of $\d(R)$, denote by $\ipd(\C)$ the set of prime ideals $\p$ of $R$ with $\pd_{R_\p}X_\p=\infty$ for some $X\in\C$.
For a set $\Phi$ of prime ideals of $R$, denote by $\ipd^{-1}(\Phi)$ the subcategory of $\d(R)$ consisting of complexes $X$ such that every prime ideal $\p$ of $R$ with $\pd_{R_\p}X_\p=\infty$ belongs to $\Phi$.

\begin{prop}\label{65}
Let $R$ be a hypersurface.
One then has the following mutually inverse bijections.
$$
\xymatrix{
{\left\{\begin{matrix}
\text{resolving subcategories of $\d(R)$}\\
\text{contained in $\c(R)$}
\end{matrix}\right\}}
\ar@<.7mm>[rr]^-{\NE(-)}&&
{\left\{\begin{matrix}
\text{specialization-closed subsets}\\
\text{of $\sing R$}
\end{matrix}\right\}}
\ar@<.7mm>[ll]^-{\NE^{-1}_\c(-)}&&
}
$$
\end{prop}

\begin{proof}
Fix a resolving subcategory $\X$ of $\d(R)$ contained in $\c(R)$, and a specialization-closed subset $W$ of $\sing R$.
By Proposition \ref{27}(2) and \cite[Remark 10.2(8)]{dlr}, we get $\ipd(\thick_{\d(R)}\X)=\NE(\X)$ and $\ipd^{-1}(W)\cap\c(R)=\NE^{-1}_\c(W)$.
The assertion follows by combining this with Theorem \ref{76} and \cite[Theorem 3.13(1)]{thd}.
\end{proof}

\begin{rem}
Another way in the case where $R$ is a hypersurface to deduce the equality given in Theorem \ref{76} is obtained by the combination of Propositions \ref{65} and \ref{60}.
\end{rem}

\section{Restricting the classification of resolving subcategories of $\d(R)$}

In this section, restricting the classification theorem of resolving subcategories of $\d(R)$ obtained in the previous section, we consider the resolving subcategories of $\mod R$.
We begin with establishing a lemma.

\begin{lem}\label{82}
Let $R$ be a complete intersection.
Let $\X$ be a resolving subcategory of $\mod R$.
\begin{enumerate}[\rm(1)]
\item
Let $\p$ be a prime ideal of $R$.
One has the equality $\sup_{X\in\X}\{\depth R_\p-\depth X_\p\}=\sup_{Y\in\X\cap\fpd R}\{\pd Y_\p\}$.
\item
There is an equality $\thick_{\d(R)}\X=\thick_{\d(R)}(\X\cap\cm(R))$ of thick closures in $\d(R)$.
\end{enumerate}
\end{lem}

\begin{proof}
(1) The inequality $(\ge)$ holds by the Auslander--Buchsbaum formula.
To show the opposite inequality $(\le)$, put $t=\sup_{X\in\X}\{\depth R_\p-\depth X_\p\}$.
Then $t=\depth R_\p-\depth X_\p$ for some $X\in\X$.
Set $n=\rfd_RX$.
We see that $\syz^nX$ is a maximal Cohen--Macaulay $R$-module.
By \cite[Proof of Theorem 7.4]{crspd} there is an exact sequence $0\to X\to L\to D\to0$ in $\mod R$ such that $L$ has finite projective dimension and $D=\syz^{-n-1}\syz^nX$ is maximal Cohen--Macaulay.
Applying Proposition \ref{71}(1) to $\syz^nX\in\X$, we get $D\in\X$, and hence $L\in\X\cap\fpd R$.
As $D_\p$ is a maximal Cohen--Macaulay $R_\p$-module,  we have $\depth D_\p\ge\height\p$.
The depth lemma implies
$$
\depth X_\p\ge\inf\{\depth L_\p,\depth D_\p+1\}\ge\inf\{\depth L_\p,\height\p+1\}=\depth L_\p.
$$
Hence $\pd L_\p=\depth R_\p-\depth L_\p\ge\depth R_\p-\depth X_\p=t$.
We obtain $\sup_{Y\in\X\cap\fpd R}\{\pd Y_\p\}\ge\pd L_\p\ge t$.

(2) It suffices to show that $\X$ is contained in $\thick_{\d(R)}(\X\cap\cm(R))$.
Fix an $R$-module $X\in\X$, and put $n=\rfd_RX$.
Since $R$ is Gorenstein, there is an exact sequence $0\to P\to\syz^{-n}\syz^nX\to X\to0$ in $\mod R$ such that $P$ has finite projective dimension; see \cite[(2.21) and (4.22)]{ABr}.
The $R$-module $\syz^nX$ is maximal Cohen--Macaulay.
Proposition \ref{71}(1) implies $\syz^{-n}\syz^nX\in\res_{\mod R}(\syz^nX)\subseteq\X$.
Hence $\syz^{-n}\syz^nX$ is in $\X\cap\cm(R)$.
As $P\in\thick_{\d(R)}R$ and $R\in\X\cap\cm(R)$, both $P$ and $\syz^{-n}\syz^nX$ are in $\thick_{\d(R)}(\X\cap\cm(R))$, and so is $X$.
\end{proof}

Using the above lemma and results in the previous sections, we can show that for each resolving subcategory of $\mod R$, taking the resolving closure in $\d(R)$ commutes with taking the restriction to $\k(R)$ and $\c(R)$.

\begin{prop}\label{80}
Suppose that $R$ is a complete intersection.
Let $\X$ be a resolving subcategory of $\mod R$.
\begin{enumerate}[\rm(1)]
\item
There are equalities $\res_{\d(R)}(\X\cap\k(R))=\res_{\d(R)}(\X\cap\fpd R)=(\res_{\d(R)}\X)\cap\k(R)$.
\item
There are equalities $\res_{\d(R)}(\X\cap\c(R))=\res_{\d(R)}(\X\cap\cm(R))=(\res_{\d(R)}\X)\cap\c(R)$.
\end{enumerate}
\end{prop}

\begin{proof}
The first equalities in the two assertions hold since $\X\cap\k(R)=\X\cap\fpd R$ and $\X\cap\c(R)=\X\cap\cm(R)$.
In what follows, we show the second equalities.

(1) Since $\k(R)$ is a resolving subcategory of $\d(R)$ by Proposition \ref{43}(3), we see that both $\res_{\d(R)}(\X\cap\fpd R)$ and $(\res_{\d(R)}\X)\cap\k(R)$ are resolving subcategories of $\d(R)$ contained in $\k(R)$.
Put
$$
\textstyle
a=\sup_{X\in\res_{\d(R)}(\X\cap\fpd R)}\{\pd X_\p\},\quad
b=\sup_{X\in(\res_{\d(R)}\X)\cap\k(R)}\{\pd X_\p\},\quad
c=\sup_{X\in\X\cap\fpd R}\{\pd X_\p\}.
$$
By Theorem \ref{2}, it is enough to verify that $a=b$.
Since $\X\cap\fpd R\subseteq\res_{\d(R)}(\X\cap\fpd R)\subseteq(\res_{\d(R)}\X)\cap\k(R)$, we have $c\le a\le b$.
Thus it suffices to show that $b\le c$, which can be shown as follows.
$$
\begin{array}{l}
b
\overset{\rm(i)}{=}\sup_{X\in(\res_{\d(R)}\X)\cap\k(R)}\{\depth R_\p-\depth X_\p\}
\overset{\rm(ii)}{\le}\sup_{X\in\res_{\d(R)}\X}\{\depth R_\p-\depth X_\p\}\\
\phantom{b=\sup_{X\in(\res_{\d(R)}\X)\cap\k(R)}\{\depth R_\p-\depth X_\p\}}
\overset{\rm(iii)}{=}\sup_{X\in\X}\{\depth R_\p-\depth X_\p\}
\overset{\rm(iv)}{=}c.
\end{array}
$$
Here, (i) and (iv) follow from Proposition \ref{27}(2) and Lemma \ref{82}(1), respectively.
The inclusion $(\res_{\d(R)}\X)\cap\k(R)\subseteq\res_{\d(R)}\X$ implies (ii).
As for (iii), the inequality $(\ge)$ holds since $\X$ is contained in $\res_{\d(R)}\X$.
It is observed from Proposition \ref{27}(2)(3) that the subcategory $\Y$ of $\d(R)$ consisting of objects $Y$ such that
$$
\textstyle
\depth R_\p-\depth Y_\p\le\sup_{X\in\X}\{\depth R_\p-\depth X_\p\}
$$
is resolving and contains $\X$.
Therefore, the subcategory $\Y$ contains $\res_{\d(R)}\X$.
Thus $(\le)$ follows.

(2) Note that both $\res_{\d(R)}(\X\cap\cm(R))$ and $(\res_{\d(R)}\X)\cap\c(R)$ are resolving subcategories of $\d(R)$ contained in $\c(R)$; see Proposition \ref{58}(2).
By virtue of Theorem \ref{76}, it is enough to show that $\thick_{\d(R)}(\res_{\d(R)}(\X\cap\cm(R)))$ coincides with $\thick_{\d(R)}((\res_{\d(R)}\X)\cap\c(R))$.
We have
$$
\begin{array}{l}
\thick_{\d(R)}(\res_{\d(R)}(\X\cap\cm(R)))
=\thick_{\d(R)}(\X\cap\cm(R))
=\thick_{\d(R)}\X
=\thick_{\d(R)}(\res_{\d(R)}\X)\\
\phantom{\thick_{\d(R)}(\res_{\d(R)}(\X\cap\cm(R)))}
\supseteq\thick_{\d(R)}((\res_{\d(R)}\X)\cap\c(R))
\supseteq\thick_{\d(R)}(\X\cap\cm(R)),
\end{array}
$$
where the first and third equalities and the inclusions are clear, while the second equality follows from Lemma \ref{82}(2).
Thus those two inclusions are equalities, and we obtain the desired equality of thick closures.
\end{proof}

Now we can state and prove the following proposition, where $\lcm(R)$ denotes the {\em stable category} of $\cm(R)$ (the definition of a thick subcategory of $\cm(R)$ is given in Section 4).
This proposition particularly says that, over a complete intersection, the resolving subcategories of maximal Cohen--Macaulay modules bijectively and naturally correspond to the resolving subcategories of maximal Cohen--Macaulay complexes.

\begin{prop}\label{94}
Let $R$ be a complete intersection.
Then there are natural one-to-one correspondences
\begin{align*}
{\left\{\begin{matrix}
\text{resolving subcategories}\\
\text{of $\mod R$}\\
\text{contained in $\cm(R)$}
\end{matrix}\right\}}
&=
{\left\{\begin{matrix}
\text{thick subcategories}\\
\text{of $\cm(R)$}\\
\text{containing $R$}
\end{matrix}\right\}}
\cong
{\left\{\begin{matrix}
\text{thick subcategories}\\
\text{of $\lcm(R)$}
\end{matrix}\right\}}
\cong
{\left\{\begin{matrix}
\text{thick subcategories}\\
\text{of $\s(R)$}
\end{matrix}\right\}}\\
&\cong
{\left\{\begin{matrix}
\text{thick subcategories}\\
\text{of $\d(R)$}\\
\text{containing $R$}
\end{matrix}\right\}}
\cong
{\left\{\begin{matrix}
\text{thick subcategories}\\
\text{of $\c(R)$}\\
\text{containing $R$}
\end{matrix}\right\}}
=
{\left\{\begin{matrix}
\text{resolving subcategories}\\
\text{of $\d(R)$}\\
\text{contained in $\c(R)$}
\end{matrix}\right\}}.
\end{align*}
In particular, one has the following one-to-one correspondence.
$$
\xymatrix{
{\left\{\begin{matrix}
\text{resolving subcategories}\\
\text{of $\mod R$ contained in $\cm(R)$}
\end{matrix}\right\}}
\ar@<.7mm>[rr]^-{\res_{\d(R)}(-)}&&
{\left\{\begin{matrix}
\text{resolving subcategories}\\
\text{of $\d(R)$ contained in $\c(R)$}
\end{matrix}\right\}}.
\ar@<.7mm>[ll]^-{(-)\cap\cm(R)}&
}
$$
\end{prop}

\begin{proof}
We start by showing the first assertion.
The first equality can be obtained by \cite[Corollary 4.16]{radius}, where the ring is assumed to be local, but the argument works if we replace \cite[Theorem 4.15(1)]{radius} used there with Proposition \ref{71}(1).
The first bijection follows from \cite[Proposition 6.2]{stcm}, where the ring is again assumed to be local but it is not used.
Since $R$ is Gorenstein, the assignment $M\mapsto M$ gives a triangle equivalence
$$
\eta:\lcm(R)=\lgp(R)\xrightarrow{\cong}\s(R),
$$
where $\lgp(R)$ denotes the stable category of the category $\gp(R)$ of totally reflexive $R$-modules; see \cite[1.3]{BDZ}.
The second bijection in the assertion is induced from the equivalence $\eta$. 
The third bijection is given in Lemma \ref{74}(1).
The last bijection and the last equality follow from Theorem \ref{76}.

From now on, we give a proof of the last assertion of the proposition.
There is a commutative diagram
$$
\xymatrix@R-1pc@C5pc{
\cm(R)\ar[r]^\inc\ar[d]^\varepsilon&\d(R)\ar[d]^\pi\\
\lcm(R)\ar[r]^\eta&\s(R)
}
$$
where $\inc$ is the inclusion functor, $\eta$ is the triangle equivalence, and $\varepsilon,\pi$ are the canonical quotient functors.

Fix a resolving subcategory $\X$ of $\mod R$ contained in $\cm(R)$.
The resolving subcategory of $\d(R)$ contained in $\c(R)$ that corresponds to $\X$ is $\pi^{-1}\eta\varepsilon(\X)\cap\c(R)$, which coincides with $\pi^{-1}\pi(\X)\cap\c(R)$.
As this is a resolving subcategory of $\d(R)$ containing $\X$, it contains $\res_{\d(R)}\X$ as well.
Pick an object $C\in\pi^{-1}\pi(\X)\cap\c(R)$.
Then $\pi(C)$ is in $\pi(\X)$, and $\pi(C)$ is isomorhic to $\pi(X)$ for some $X\in\X$.
There are exact triangles $\sigma:E\to C\to A\rightsquigarrow$ and $\tau:E\to X\to B\rightsquigarrow$ in $\d(R)$ with $A,B\in\k(R)$; see \cite[Proposition 2.1.35]{N2}.
We see from $\tau$ that $E$ is in $\thick_{\d(R)}\X$, and from $\sigma$ that $C$ is in $\thick_{\d(R)}\X$.
Hence $\pi^{-1}\pi(\X)\cap\c(R)\subseteq(\thick_{\d(R)}\X)\cap\c(R)=\res_{\d(R)}\X$, where the equality follows from Theorem \ref{76}.
We now conclude that $\pi^{-1}\pi(\X)\cap\c(R)=\res_{\d(R)}\X$.

Fix a resolving subcategory $\X$ of $\d(R)$ contained in $\c(R)$.
The resolving subcategory of $\mod R$ contained in $\cm(R)$ that corresponds to $\X$ is $\varepsilon^{-1}\eta^{-1}\pi(\thick_{\d(R)}\X)$.
Note that the equality $\pi^{-1}\pi(\Y)=\Y$ holds for each thick subcategory $\Y$ of $\d(R)$ containing $R$.
We get the following equalities of subcategories of $\cm(R)$.
$$
\begin{array}{l}
\varepsilon^{-1}\eta^{-1}\pi(\thick_{\d(R)}\X)
=\pi^{-1}\pi(\thick_{\d(R)}\X)\cap\cm(R)
=(\thick_{\d(R)}\X)\cap\cm(R)\\
\phantom{\varepsilon^{-1}\eta^{-1}\pi(\thick_{\d(R)}\X)}=(\thick_{\d(R)}\X)\cap\c(R)\cap\cm(R)
=\X\cap\cm(R).
\end{array}
$$
Here, the last equality follows from Theorem \ref{76}.

Now we obtain the mutually inverse bijections in the last assertion of the proposition.
\end{proof}

Proposition \ref{94} says that when $R$ is a complete intersection, the equality $\X=\res_{\d(R)}(\X\cap\cm(R))$ holds for every resolving subcategory $\X$ of $\d(R)$ contained in $\c(R)$.
This equality holds in a more general setting.

\begin{prop}\label{101}
The equality $\X=\res_{\d(R)}(\X\cap\gp(R))$ holds for every resolving subcategory $\X$ of $\d(R)$ contained in $\g(R)$.
In particular, if the ring $R$ is Gorenstein, then the equality $\X=\res_{\d(R)}(\X\cap\cm(R))$ holds for every resolving subcategory $\X$ of $\d(R)$ contained in $\c(R)$.
\end{prop}

\begin{proof}
The last assertion follows from the first and Lemma \ref{66}(5).
To show the first assertion, let $\X$ be a resolving subcategory of $\d(R)$ contained in $\g(R)$.
It clearly holds that $\X$ contains $\res_{\d(R)}(\X\cap\gp(R))$.
Pick any $X\in\X$.
Remark \ref{40}(2) gives an exact triangle $X\to Y\to E\rightsquigarrow$ in $\d(R)$ with $\sup Y\le0$ and $E\in\E_R$.
By Lemma \ref{66}(3)(4) there exists a totally reflexive $R$-module $T$ such that $Y\cong T$ in $\d(R)$.
Since $Y$ is in $\X$, we have $T\in\X\cap\gp(R)$.
Hence $Y$ is in $\res_{\d(R)}(\X\cap\gp(R))$, and so is $X$.
Thus the first assertion follows.
\end{proof}

In view of Propositions \ref{94} and \ref{101}, it is quite natural to ask the following question.
Proposition \ref{94} guarantees that the question has an affirmative answer in the case where $R$ is a complete intersection.

\begin{ques}
Suppose that the ring $R$ is Gorenstein.
Let $\X$ be a resolving subcategory of $\mod R$ contained in $\cm(R)$.
Then, does the equality $\X=(\res_{\d(R)}\X)\cap\cm(R)$ hold?
\end{ques}

To show our next result, we establish a lemma on projective dimension.

\begin{lem}\label{37}
Let $\X$ be a resolving subcategory of $\mod R$ contained in $\fpd R$.
Let $Y$ be an object in $\res_{\d(R)}\X$, and let $\p$ be a prime ideal of $R$.
Then one has the inequality $\pd_{R_\p}Y_\p\le\pd_{R_\p}X_\p$ for some object $X\in\X$.
\end{lem}

\begin{proof}
Let $\ZZ$ be the subcategory of $\d(R)$ consisting of complexes $Z$ such that $\pd Z_\p\le\pd X_\p$ for some $X\in\X$.
Clearly, $\X$ is contained in $\ZZ$, and in particular, $R$ is in $\ZZ$.
If $Z$ is an object in $\ZZ$ and $W$ is a direct summand of $Z$, then $\pd W_\p\le\pd Z_\p\le\pd X_\p$ for some $X\in\X$ by Proposition \ref{27}(4), and hence $W$ is also in $\ZZ$.
Let $A\to B\to C\rightsquigarrow$ be an exact triangle in $\d(R)$ with $C\in\ZZ$.
Then $\pd C_\p\le\pd X_\p$ for some $X\in\X$.
If $\pd A_\p$ (resp. $\pd B_\p$) is at most $\pd X'_\p$ for some $X'\in\X$, then $\pd B_\p$ (resp. $\pd A_\p$) is at most $\sup\{\pd A_\p,\pd C_\p\}$ (resp. $\sup\{\pd B_\p,\pd C_\p-1\}$) by Proposition \ref{27}(3), which is at most $\pd X''_\p$ where $X''=X\oplus X'\in\X$ by Proposition \ref{27}(4).
Hence $A\in\ZZ$ if and only if $B\in\ZZ$.
Thus, $\ZZ$ is a resolving subcategory of $\d(R)$ containing $\X$.
Then $\ZZ$ contains $\res_{\d(R)}\X$, and we get $Y\in\ZZ$.
We conclude $\pd Y_\p\le\pd X_\p$ for some $X\in\X$.
\end{proof}

Now we find out a close relationship of each resolving subcategory of $\mod R$ with its resolving closure in $\d(R)$ when $R$ is a complete intersection.
In the proof we use the map $\Phi$ which was defined in Definition \ref{53}.

\begin{prop}\label{84}
Let $\X$ be a resolving subcategory of $\mod R$.
Suppose either that $\X$ is contained in $\fpd R$ or that $R$ is a complete intersection.
Then the equality $\X=(\res_{\d(R)}\X)\cap\mod R$ holds true.
\end{prop}

\begin{proof}
We set up three steps, and in each step we prove the equality given in the proposition.

(1) Assume that $\X$ is contained in $\fpd R$.
Then $\X$ is contained in $\k(R)$, and so is $\res_{\d(R)}\X$ by Proposition \ref{43}(3).
Proposition \ref{43}(4) says that $(\res_{\d(R)}\X)\cap\mod R$ is a resolving subcategory of $\mod R$ contained in $\fpd R$.
There are inclusions $\X\subseteq\res_{\d(R)}\X\cap\mod R\subseteq\res_{\d(R)}\X$, which induce the inequalities $\Phi(\X)(\p)\le\Phi((\res_{\d(R)}\X)\cap\mod R)(\p)\le\Phi(\res_{\d(R)}\X)(\p)$ for each prime ideal $\p$ of $R$.
Lemma \ref{37} yields that
$$
\textstyle\Phi(\res_{\d(R)}\X)(\p)=\sup_{Y\in\res_{\d(R)}\X}\{\pd Y_\p\}\le\sup_{X\in\X}\{\pd X_\p\}=\Phi(\X)(\p),
$$
and therefore the equalities $\Phi(\X)(\p)=\Phi((\res_{\d(R)}\X)\cap\mod R)(\p)=\Phi(\res_{\d(R)}\X)(\p)$ hold.
This shows that $\Phi(\X)$ coincides with $\Phi((\res_{\d(R)}\X)\cap\mod R)$.
By Theorem \ref{36}, we obtain $\X=(\res_{\d(R)}\X)\cap\mod R$.

(2) Assume that $\X$ is contained in $\cm(R)$ and that $R$ is a complete intersection.
Proposition \ref{94} implies $\X=(\res_{\d(R)}\X)\cap\cm(R)$.
As $\c(R)$ is a resolving subcategory of $\d(R)$ by Proposition \ref{58}(2), it contains $\res_{\d(R)}\X$.
We obtain $\X=(\res_{\d(R)}\X)\cap\cm(R)=(\res_{\d(R)}\X)\cap\c(R)\cap\mod R=(\res_{\d(R)}\X)\cap\mod R$.

(3) Suppose that $R$ is a complete intersection.
Put $\Y=(\res_{\d(R)}\X)\cap\mod R$.
We want to prove $\X=\Y$.
By \cite[Theorem 7.4]{crspd}, it suffices to show $\X\cap\fpd R=\Y\cap\fpd R$ and $\X\cap\cm(R)=\Y\cap\cm(R)$.
We have 
$$
\begin{array}{l}
\Y\cap\fpd R
=(\res_{\d(R)}\X)\cap\fpd R
=(\res_{\d(R)}\X)\cap\k(R)\cap\mod R\\
\phantom{\Y\cap\fpd R
=(\res_{\d(R)}\X)\cap\fpd R}=\res_{\d(R)}(\X\cap\fpd R)\cap\mod R
=\X\cap\fpd R,
\end{array}
$$
where the fourth and third equalities follow by (1) and Proposition \ref{80}(1), respectively.
Similarly, we have
$$
\begin{array}{l}
\Y\cap\cm(R)
=(\res_{\d(R)}\X)\cap\cm(R)
=(\res_{\d(R)}\X)\cap\c(R)\cap\mod R\\
\phantom{\Y\cap\cm(R)
=(\res_{\d(R)}\X)\cap\cm(R)}=\res_{\d(R)}(\X\cap\cm(R))\cap\mod R
=\X\cap\cm(R),
\end{array}
$$
where the fourth and third equalities follow from (2) and Proposition \ref{80}(2), respectively.
\end{proof}

Let $f:A\to B$ and $g:B\to A$ be maps.
We call $(f,g)$ a {\em section-retraction pair} (resp. {\em bijection pair}) if $gf$ is an identity map (resp. $gf,fg$ are identity maps).
In this case, we denote it by $f\dashv g$ (resp. $f\sim g$).
Now we can state and prove the following theorem, which describes a natural relationship between the resolving subcategories of $\d(R)$ and the resolving subcategories of $\mod R$ in the case where $R$ is a complete intersection.

\begin{thm}\label{89}
Let $R$ be a complete intersection.
Then there is a diagram
$$
\xymatrix{
{\left\{\begin{matrix}
\text{resolving}\\
\text{subcategories}\\
\text{of $\d(R)$}
\end{matrix}\right\}}\ar@<2mm>[rrr]_-\wr^-{((-)\cap\k(R),(-)\cap\c(R))}\ar@<2mm>[d]_-\dashv^-{(-)\cap\mod R}&&&
{\left\{\begin{matrix}
\text{resolving}\\
\text{subcategories of $\d(R)$}\\
\text{contained in $\k(R)$}
\end{matrix}\right\}}
\times
{\left\{\begin{matrix}
\text{resolving}\\
\text{subcategories of $\d(R)$}\\
\text{contained in $\c(R)$}
\end{matrix}\right\}}\ar@<2mm>[d]_-\dashv^-{((-)\cap\mod R)\times((-)\cap\mod R)}\ar@<2mm>[lll]^-{\res_{\d(R)}(-\cup\cdots)}
\\
{\left\{\begin{matrix}
\text{resolving}\\
\text{subcategories}\\
\text{of $\mod R$}
\end{matrix}\right\}}\ar@<2mm>[rrr]_-\wr^-{((-)\cap\fpd R,(-)\cap\cm(R))}\ar@<2mm>[u]^-{\res_{\d(R)}(-)}&&&
{\left\{\begin{matrix}
\text{resolving}\\
\text{subcategories of $\mod R$}\\
\text{contained in $\fpd R$}
\end{matrix}\right\}}
\times
{\left\{\begin{matrix}
\text{resolving}\\
\text{subcategories of $\mod R$}\\
\text{contained in $\cm(R)$}
\end{matrix}\right\}}.\ar@<2mm>[u]^-{\res_{\d(R)}(-)\times\res_{\d(R)}(-)}\ar@<2mm>[lll]^-{\res_{\mod R}(-\cup\cdots)}
}
$$
The pairs of top (resp. bottom) horizontal arrows are bijection pairs given in Theorem \ref{72} (resp. \cite[Theorem 7.4]{crspd}).
The pairs of vertical arrows are section-retraction pairs.
The diagram with vertical arrows from the bottom (resp. top) to the top (resp. bottom) is commutative.
\end{thm}

\begin{proof}
It follows from Proposition \ref{84} that the pairs of maps $(\res_{\d(R)}(-),(-)\cap\mod R)$ and $(\res_{\d(R)}(-)\times\res_{\d(R)}(-),((-)\cap\mod R)\times((-)\cap\mod R)$ are section-retraction pairs.
Also, it holds that
$$
\begin{array}{l}
((-)\cap\k(R),(-)\cap\c(R))\circ\res_{\d(R)}(-)=(\res_{\d(R)}(-)\times\res_{\d(R)}(-))\circ((-)\cap\fpd R,(-)\cap\cm(R)),\\
((-)\cap\fpd R,(-)\cap\cm(R))\circ((-)\cap\mod R)=(((-)\cap\mod R)\times((-)\cap\mod R))\circ((-)\cap\k(R),(-)\cap\c(R)).
\end{array}
$$
Indeed, the first equality follows from Proposition \ref{80}, while it is straightforward to verify the second.
\end{proof}

\begin{rem}
The section-retraction pair $(\res_{\d(R)}(-),\,(-)\cap\mod R)$ in Theorem \ref{89} is {\em never} a bijection pair.
Indeed, if so, then $\res_{\d(R)}(\X\cap\mod R)=\X$ for every resolving subcategory $\X$ of $\d(R)$.
However, this equality does not hold even for $\X=\d(R)$, because in this case we have the following equalities
$$
\res_{\d(R)}(\X\cap\mod R)=\res_{\d(R)}(\mod R)=\{X\in\d(R)\mid\h^{<0}X=0\}
$$
by Proposition \ref{79}, which is strictly contained in $\X=\d(R)$.
\end{rem}

The corollary below is an immediate consequence of Theorem \ref{89}, Corollary \ref{41} and \cite[Theorem 1.5]{crspd}.
This corollary says that the classification of resolving subcategories of $\mod R$ due to Dao and Takahashi \cite{crspd} is a restriction of our classification of resolving subcategories of $\d(R)$.

\begin{cor}\label{83}
Let $(R,V)$ be as in Setup \ref{96}.
Then there is a commutative diagram
$$
\xymatrix{
{\left\{\begin{matrix}
\text{resolving subcategories}\\
\text{of $\d(R)$}
\end{matrix}\right\}}\ar@{<->}[r]^-\cong_-{(\alpha)}&
{\left\{\begin{matrix}
\text{order-preserving maps}\\
\text{from $\spec R$ to $\N\cup\{\infty\}$}
\end{matrix}\right\}}
\times
{\left\{\begin{matrix}
\text{specialization-closed}\\
\text{subsets of $V$}
\end{matrix}\right\}}
\\
{\left\{\begin{matrix}
\text{resolving subcategories}\\
\text{of $\mod R$}
\end{matrix}\right\}}\ar@{<->}[r]^-\cong_-{(\beta)}\ar[u]^-{\res_{\d(R)}(-)}&
{\left\{\begin{matrix}
\text{grade-consistent}\\
\text{functions on $\spec R$}
\end{matrix}\right\}}
\times
{\left\{\begin{matrix}
\text{specialization-closed}\\
\text{subsets of $V$}
\end{matrix}\right\}},\ar[u]^{\inc\times\id}
}
$$
where the bijections $(\alpha)$ and $(\beta)$ are the ones given in Corollary \ref{41} and \cite[Theorem 1.5]{crspd}, respectively. 
\end{cor}

Finally, we give a proof of our main result stated in the Introduction.

\begin{proof}[Proof of Theorem \ref{1}]
The assertion follows from Corollaries \ref{41}, \ref{83}, Proposition \ref{84} and Remark \ref{98}.
\end{proof}

\begin{ac}
The author thanks Hiroki Matsui, Tsutomu Nakamura and an anonymous reader for giving him useful and helpful comments.
\end{ac}

\end{document}